\documentclass{amsart}

\usepackage{amssymb,latexsym,amsmath,extarrows,amsthm, tikz}

\usepackage{graphicx}
\usepackage{xcolor}
\usepackage{mathabx}

\definecolor{blue}{rgb}{0,0,1}

\definecolor{red}{rgb}{1,0,.2}

\newtheorem*{theorem*}{Theorem}

\newtheorem{theorem}{Theorem}[section]
\newtheorem{lemma}[theorem]{Lemma}
\newtheorem{proposition}[theorem]{Proposition}
\newtheorem{remark}[theorem]{Remark}
\newtheorem{example}[theorem]{Example}

\newtheorem{corollary}[theorem]{Corollary}

\usepackage{hyperref}

\begin{document}
%\title[On finite point configurations in $\mathbb{R}^d$]
%{On finite point configurations in $\mathbb{R}^d$: Lebesgue measure, Hausdorff dimension, and prescribed distances}
\title[Finite Point Configurations]
{Finite Point Configurations and the Regular Value Theorem in a Fractal setting}
\author{Yumeng Ou and Krystal Taylor}
\thanks{Y. O. is funded in part by NSF DMS-1854148. K. T. is funded in part by Simons 523555. We would like to thank the referee for carefully reading our manuscript and offering constructive comments. }

\begin{abstract}
In this article, we study two problems concerning the size of the set of finite point configurations generated by a compact set $E\subset \mathbb{R}^d$. The first problem concerns how the Lebesgue measure or the Hausdorff dimension of the finite point configuration set depends on that of $E$. 
In particular, we show that if a planar set has dimension exceeding $\frac{5}{4}$, then there exists a point $x\in E$ so that for each integer $k\geq2$, the set of ``$k$-chains'' with initial point at $x$ has positive Lebesgue measure. 

The second problem is a continuous analogue of the Erd\H{o}s unit distance problem, which aims to determine the 
maximum 
number of times a point configuration with prescribed gaps can appear in $E$. 
For instance, given a triangle with prescribed sides and given a sufficiently regular planar set $E$ with Hausdorff dimension no less than 
$\frac{7}{4}$, we show that the dimension of the set of vertices in $E$ forming said triangle does not exceed $3\,\dim_{\mathcal{H}} (E)-3$.
In addition to the Euclidean norm, we consider more general distances given by functions satisfying the so-called Phong-Stein rotational curvature condition. 
We also explore a number of examples to demonstrate the extent to which our results are sharp.

\end{abstract}

\maketitle
%\tableofcontents

\section{Introduction}
We consider two problems concerning $k$-point configurations in subsets of $\mathbb{R}^d$.  
The first aim is to understand how large a subset of Euclidean space must be to ensure that 
it contains many distinct scaled copies of a given polyhedron or another geometric shape.
Upon fixing a scaling, the second problem is to determine how often a fixed shape occurs within a set of a given size.   
\vskip.05in

These questions are natural analogues of some famous open questions in discrete geometry and geometric measure theory. More precisely, the first question can be viewed as an extension of the Falconer distance set problem (whose predecessor is the celebrated Erd\H{o}s distinct distance problem in the discrete setting \cite{Erd45}), which conjectures that $|\Delta(E)|_1>0$ whenever the Hausdorff dimension of $E$ exceeds $\frac{d}{2}$ and remains open in all $d\geq 2$. 
Here, $\Delta(E)$ denotes the set of distances, $\{|x-y|: x,y\in E\}$, and $|\cdot|=|\cdot|_k$ denotes the $k$-dimensional Lebesgue measure of a set. 
%When it is clear from the context what the value of $k$ is, we often write $|\cdot|=|\cdot|_k$ for short.
\vskip.05in

The second question extends the Erd\H{o}s unit distance problem to the $k\geq 2$ case in the continuous setting. The unit distance conjecture in the plane \cite{Erd45} says that if $P$ is a planar point set with $N$ points, then the number of pairs of points in $P$ at a distance $1$ apart is bounded above by $C_\epsilon N^{1+\epsilon}$, $\forall \epsilon>0$. These questions have attracted a great amount of attention over the decades (see for instance \cite{DGOWWZ18, DZ18, Erd05, GIOW, GK15, SST, W99, Zahl} and the references therein). Their study has utilized and inspired ideas in many different fields, such as Fourier analysis (e.g. restriction theory, decoupling) and combinatorics (e.g. polynomial method).
\vskip.05in

%edge
In order to give formal statements of the two main questions of focus, we will need some notation. 
%We first introduce some notation before stating our main results. 
Define the chain set as follows: For $E\subset\mathbb{R}^d$, $d\geq 2$, and for integers $k\geq 2$, define the set of \emph{non-degenerate $k$-chains generated by $E$} as 
\[
S^k(E):=\{(|x_1-x_2|,\ldots, |x_{k}-x_{k+1}|)\in \mathbb{R}^k_+:\, x_i \in E, \,  x_1,\ldots, x_{k+1} \text{ are distinct}\}.
\]Note that $S^1(E)$ is simply the distance set of $E$, and it will be denoted by $\Delta(E)$ to be consistent with classical literature.
Further, we have the pinned version:
\[
S^k_x(E):=\{(|x-x_1|,\ldots, |x_{k-1}-x_k|)\in \mathbb{R}^k_+:\, x_i \in E, \, x, x_1,\ldots, x_k \text{ are distinct}\},
\]which consists of non-degenerate $k$-chains in $E$ that share a common starting place of $x\in E$. If $k=1$, the pinned distance set, $S^1_x(E)$, is denoted by $\Delta_x(E)$.
\vskip.05in

%%%definition of a tree
In addition to the edge-length sets generated by chains, we consider the edge-length sets of $k$-trees, triangles, and other configurations with or without loops.  
A tree is a graph in which each pair of vertices are connected by exactly one path (see, for instance, \cite{IT}). 
The edge-length set of a $k$-tree, denoted as $T^k(E)$, consists of the edge lengths of all trees of any particular fixed shape $\mathcal{T}^k$, with $k+1$ vertices in $E$ and $k$ edges.

In more detail, let $\mathcal{V}^{k+1}=\{x_1,\cdots, x_{k+1}\}$ be an ordered set of $(k+1)$ distinct vertices. Let $\mathcal{E}=\mathcal{E}(\mathcal{V}^{k+1})$ be an associated edge set such that
\[
\mathcal{E}=\mathcal{E}(\mathcal{V}^{k+1})\subset \{(x_i,x_j):\, x_i,x_j\in \mathcal{V}^{k+1},\, i<j\}.
\]
%path
Given distinct $x_i, x_j\in \mathcal{V}^{k+1}$, a \emph{path} connecting $x_i, x_j$ is defined to be a subset $\mathcal{P}\subset \mathcal{E}(\mathcal{V}^{k+1})$ satisfying both
 (i) $x_i, x_j$ each appears exactly once in elements of $\mathcal{P}$, and 
 (ii) each of the remaining vertices in $\mathcal{V}^{k+1}$ appears either $0$ or $2$ times in elements of $\mathcal{P}$. We say that the set $\mathcal{T}^k=\mathcal{T}^k(\mathcal{V}^{k+1},\mathcal{E})$ is a $k$-tree of shape $\mathcal{E}$ if every pair of distinct vertices in $\mathcal{V}^{k+1}$ are connected by a unique path in $\mathcal{E}$. When there is no need to specify the set $\mathcal{E}$, we also simply call this a tree of shape $\mathcal{T}^k$.
%Further, we say that $k$-trees $(\mathcal{V}^{k+1},\mathcal{E}^k)$ and $(\mathcal{W}^{k+1},\mathcal{F}^k)$ are of the same shape if, up to relabeling, $\mathcal{E}^k$ and $\mathcal{F}^k$ are equal. 
%We let $\mathcal{T}^k=\mathcal{T}^k(\mathcal{V}^{k+1},\mathcal{E}^k)$ denote the set of all $k$-trees that have the same shape as the $k$-tree $(\mathcal{V}^{k+1},\mathcal{E}^k)$.  
%%
It is easy to check that if $\mathcal{T}^k(\mathcal{V}^{k+1},\mathcal{E})$ is a $k$-tree, then $\mathcal{E}$ must contain exactly $k$ elements. Given any $k$-tree $\mathcal{T}^k(\mathcal{V}^{k+1},\mathcal{E})$,  we enumerate $\mathcal{E}$ by 
%For the sake of simplicity, given any $k$-tree \rk{$(\mathcal{V}^{k+1},\mathcal{E}^k)$}, we always associate with $\mathcal{E}^k$ a fixed order, with its $k$ elements enumerated as
\[
\left\{(x_{i_1},x_{i_2}),\, (x_{i_3},x_{i_4}),\,\cdots,\, (x_{i_{2k-1}}, x_{i_{2k}})     \right\}
\]where $i_1\leq i_3\leq\cdots \leq i_{2k-1}$, and $i_{2s}<i_{2t}$ whenever $s<t$, $i_{2s-1}=i_{2t-1}$,
%One can thus 
and we define the following edge-length vector:
%of $\mathcal{T}^k$:
\[
EL(\mathcal{V}^{k+1},\mathcal{E}):=\left(|x_{i_1}-x_{i_2}|,\, |x_{i_3}-x_{i_4}|,\,\cdots,\, |x_{i_{2k-1}}-x_{i_{2k}}|  \right) \in \mathbb{R}^k_+.
\]

Given any compact set $E\subset\mathbb{R}^d$, integer $k\geq 1$ and a fixed $k$-tree $\mathcal{T}^k(\mathcal{V}^{k+1},\mathcal{E})$, we define the \emph{edge-length set of $k$-trees of shape $\mathcal{E}$ (or $\mathcal{T}^k$) generated by $E$} as
\[
T^k(E):=\{EL(\mathcal{W}^{k+1},\mathcal{E}(\mathcal{W}^{k+1})):\, \mathcal{W}^{k+1}\subset E \}.
\]

We are also interested in pinned variants.  $T^k_x(E)$ will be used to denote the edge-length set of $k$-trees of a particular shape $\mathcal{T}^k_v$ (i.e. of shape $\mathcal{T}^k=\mathcal{T}^k(\mathcal{V}^{k+1},\mathcal{E})$ as defined above, with a particular vertex $v\in \mathcal{V}^{k+1}$ to be pinned at $x$). Note that (pinned) chains are special examples of (pinned) trees. 
\vskip.05in

In addition to the edge-length sets, we consider the vertex sets (of given point configurations). 
Given any sequence of distances $\vec{t}=(t_1,\cdots, t_k)\in \mathbb{R}_+^k$, define 
\begin{equation}\label{singledistance}
VS^k_{\vec{t}}(E):=\{(x_1,\cdots,x_{k+1})\in E^{k+1}:\, |x_i-x_{i+1}|=t_i,\, i=1,\ldots, k,\, \{x_i\}\, \text{distinct}\}.
\end{equation}
as the $k$-chain set generated by $E$ with prescribed distances $\vec{t}$,
where $E^k$ denotes the $k$-fold Cartesian product of $E$.
Similarly, vertex sets can be defined for the $k$-tree set of a particular shape, $VT^k_{\vec{t}}(E)$, as well as more general configurations containing loops, such as triangles,  $V{\rm Tri}_{\vec{t}}(E)$ (see \eqref{vtri}). 
\vskip.05in

This article concerns two main questions:
\begin{itemize}
\item[1.]
How does the size of the edge-length set $S^k(E)$, $T^k(E)$, or that generated by other point configurations depend on the size of the set $E$? 
\item[2.] Determine the number of times that a given $k$-chain, tree, or triangle with \emph{fixed side lengths} can repeat in $E$?
More precisely, determine the size of the vertex sets $VS^k_{\vec{t}}(E)$, $VT^k_{\vec{t}}(E)$, $V{\rm Tri}_{\vec{t}}(E)$.
\end{itemize}
The notion of size or number is made formal using Lebesgue measure, Hausdorff dimension, or Minkowski dimension.  
\vskip.05in

Concerning the first question, our main contribution is establishing a method that can serve as a bridge to extend all sufficiently good distance results to more intricate configurations, or more generally, to extend results concerning subgraphs to the whole graph given that they are \emph{glued} together in a nice way. This method is surprisingly simple and relies on a Fubini-like theorem. A key advantage of this method is its flexibility to deal with much more general point configurations that may have loops. 
 \vskip.05in

Regarding the second question, this seems to be the first article of its kind to extend the unit distance question to the setting of $k$-point configurations in the continuous setting (see \cite{FK, PSS} for the discrete setting and \cite{EIT, OO} for the $k=1$ case in the continuous setting).  It also appears to be the first article to examine a fractal variant of the regular value theorem for $k$-point configurations (see \cite{EIT} for the $k=1$ case). 
While existing techniques lend easily to results in the setting where the set $E$ is assumed to be Ahlfors-David regular (see Remarks \ref{rmk: AD} and \ref{implicit_ADresults} and the references there), this article presents new techniques that extend to sets with more relaxed regularity assumptions.  In particular, our techniques hold 
%not only for Ahlfors-David regular sets, but 
for some classic examples that fall outside the scope of AD regularity, such as the lattice example (see Example \ref{example: lattice}) and the train track example (see Example \ref{example: train}).

%%%%%%FIRST QUESTION

\subsection{On the first question: Lebesgue measure and dimension}

For $E\subset \mathbb{R}^d$, we write $\dim_{\mathcal{H}}(E)$ to denote the Hausdorff dimension of $E$, and we write  $\overline{{\rm dim}}_{\mathcal{M}}(E)$, $\underline{{\rm dim}}_{\mathcal{M}}(E)$, respectively, to denote the upper and lower Minkowski dimension of the set $E$. 
%%%% chain GIOW theorem
\begin{theorem}\label{main1}
Let $E\subset \mathbb{R}^2$ be a compact set satisfying $\dim_{\mathcal{H}}(E)>\frac{5}{4}$, then there exists a point $x\in E$ such that for all integers $k\geq 2$, all $k$-trees $\mathcal{T}_v^k$ of any shape pinned at any vertex, $|T^k_x(E)|_k>0$. In particular, $|S^k_x(E)|_k>0$. 
\end{theorem}

Theorem \ref{main1} generalizes the work of the first listed author with Guth, Iosevich and Wang \cite{GIOW} where the distance set case (i.e. $k=1$) is proved. 
When $k\geq 2$, the first such result concerning the pinned chains of $k$ distances in $\mathbb{R}^d$ is due to Bennett, Iosevich, and the second listed author \cite{BIT}, where it is required that $\dim_{\mathcal{H}}(E)>\frac{d+1}{2}$. In \cite{BIT}, the authors demonstrate that, for each $k$ and for each $E\subset \mathbb{R}^d$ of dimension greater than $\frac{d+1}{2}$,  there exist an interval worth of admissible gaps (dependent only on $k$) for which $E$ contains the vertices of a $k$-chain with side lengths in said interval.  Their argument establishes continuity of the Radon-Nikodym derivative of a natural measure on $S^k(E)$. 
%In this paper, we are more concerned with $L^{\infty}$ bounds.  
(Also see \cite{FKW90} and \cite{Z06}, where the problem is investigated for sets of positive upper Lebesgue density). Recently, it was obtained by Liu \cite{Liu18} that the dimensional threshold concerning pinned $k$-chains can be lowered to $\frac{4}{3}$ in the plane (in the case of $k=2$, the threshold $\frac{4}{3}$ was first achieved in \cite{Liu17} for the full $k$-chain set).

The case of trees is slightly more involved compared to chains, as the iteration procedure becomes more complicated due to the fact that a vertex may be connected to many edges. Theorem \ref{main1} improves the previously best known result of Iosevich and the second listed author \cite{IT} where the threshold $\frac{3}{2}$ is obtained.

Moreover, we also study the \textit{dimension} of the (pinned) tree sets and prove the following.

\begin{theorem}\label{main2}
Let $E\subset \mathbb{R}^2$ be a compact set satisfying $\dim_{\mathcal{H}}(E)>1$, then for all integers $k\geq 2$, we have
\[
\dim_{\mathcal{H}}(T^k(E))\geq \min\left\{ \frac{4k}{3}\dim_{\mathcal{H}}(E) -\frac{2k}{3}, k\right\}.
\]

Moreover, for all $\epsilon>0$, for each $k \geq 2$, there exists a point $x\in E$ such that 
\begin{equation}\label{chain_dim}
\dim_{\mathcal{H}}(T_x^k(E))\geq \min\left\{\frac{4k}{3}\dim_{\mathcal{H}}(E) -\frac{2k}{3}-\epsilon, k\right\}.
\end{equation}

Furthermore, if $1< \dim_{\mathcal{H}}(E)\leq \frac{5}{4}$, then for all sufficiently small $\epsilon>0$, there exists an $x\in E$ so that for all $k\geq 2$, 
\begin{equation}\label{infinite_chain_dim}
\dim_{\mathcal{H}}(    T^k_x(E)  )  \geq  \frac{4k}{3}(\dim_{\mathcal{H}}(E)-\epsilon)-\frac{2k}{3}> \frac{2k}{3}.  
\end{equation}

\end{theorem}

In particular, the result above holds for (pinned) chains. When $k=1$, this was proved by Liu \cite{Liu19}. Note that there is a minor inaccuracy in the statement of \cite[Theorem 1.1]{Liu19}, where an $\epsilon$ is in fact needed similarly as in (\ref{chain_dim}) of Theorem \ref{main2}. In fact, an improvement of the $k=1$ case when $\dim_{\mathcal{H}}(E)\in (1, 1.037)$ was recently obtained by Shmerkin \cite{Shm18}. Since the main contribution of our work is a method that allows one to extend the $k=1$ result automatically to all $k\geq 2$, we omit the statement of the slight improvement of Theorem \ref{main2} that can be implied by \cite{Shm18} when $\dim_{\mathcal{H}}(E)\in (1, 1.037)$. 

In addition, we also obtain a more general version of Theorem \ref{main2} on the exceptional set of $x$. For simplicity, we only state the next result for chains. 

\begin{theorem}\label{main2.5}
Given any compact set $E\subset \mathbb{R}^2$ and integer $k\geq 2$. Suppose that $\dim_{\mathcal{H}}(E)>1$.  Set 
\[
\tau_0^k=\tau^k_0(\dim_{\mathcal{H}}(E))=\begin{cases} \frac{4(k-1)}{3}\dim_{\mathcal{H}}(E)+\frac{5-2k}{3},&   1<\dim_{\mathcal{H}}(E)\leq \frac{5}{4}, \\ k,& \frac{5}{4}< \dim_{\mathcal{H}}(E)\leq 2.\end{cases}
\]Then, for each $\tau\in (0,\tau_0^k)$, 
\[
\begin{split}
&\dim_{\mathcal{H}}(\{x\in \mathbb{R}^2:\, \dim_{\mathcal{H}}(S^k_x(E))<\tau\})\\ 
\leq &\begin{cases} \max(2k+3\tau+(1-4k)\dim_{\mathcal{H}}(E),2-\dim_{\mathcal{H}}(E)), & 1<\dim_{\mathcal{H}}(E)\leq \frac{5}{4},\\
\max(5-3k+3\tau-3\dim_{\mathcal{H}}(E), 2-\dim_{\mathcal{H}}(E)), & \dim_{\mathcal{H}}(E)>\frac{5}{4}.\end{cases}
\end{split}
\]
\end{theorem}
\vskip.125in

When $k=1$, in which case $\tau_0^k=1$ and the two quantities at the end become the same, the above result also holds true and was obtained in \cite{Liu19}. In Theorem \ref{main2.5}, by setting $\tau$ to be equal to $\frac{4k}{3}\dim_{\mathcal{H}}(E)-\frac{2k}{3}-\epsilon$, one can immediately obtain not only Theorem \ref{main2}, but also the fact that the exceptional set in $E$ (consisting of bad pin points) always has lower dimension than $E$.

The proof of Theorem \ref{main2.5} can be generalized to study the case of trees of any given shape, though the exact bound is cumbersome to state since it depends on the shape of the tree and the vertex that one chooses to pin. For example, applying the same method in the proof of Theorem \ref{main2.5}, one can show that given a compact set $E\subset \mathbb{R}^2$ with $\dim_{\mathcal{H}}(E)>1$ and integer $k\geq 2$, let $T_x^k(E)$ denote all the ``$k$-star'' pinned at $x$ at the center generated by $E$, then for all $\tau\in (0,k)$, there holds
\[
\dim_{\mathcal{H}}(\{x\in \mathbb{R}^2:\, \dim_{\mathcal{H}}(T^k_x(E))<\tau\})\leq \max\left(2+\frac{3\tau}{k}-3\dim_{\mathcal{H}}(E),2-\dim_{\mathcal{H}}(E)\right).
\]Given that it is unlikely that the estimate obtained here is sharp, we skip the parallel statement of Theorem \ref{main2.5} for general trees. It is unclear to us whether the constraint on $\tau$ in Theorem \ref{main2.5} can be further relaxed.  
 In fact, it seems that the range of $\tau$ is closely tied to the iterative nature of our proof method.

%%%%%%Kites
\begin{remark}
The machinery developed in Theorem \ref{main1}, Theorem \ref{main2} and Theorem \ref{main2.5} can be used to ``glue'' together any variety of pinned $k$-point configurations (including those with loops, such as triangles) that are a priori known to exist within a compact set $E$ in $\mathbb{R}^d$.  We give an example here (see Proposition \ref{glue}) as a corollary of the method.  
\end{remark}

\begin{figure}[h]
\label{kite_figure}
\centering
\includegraphics[scale=.25]{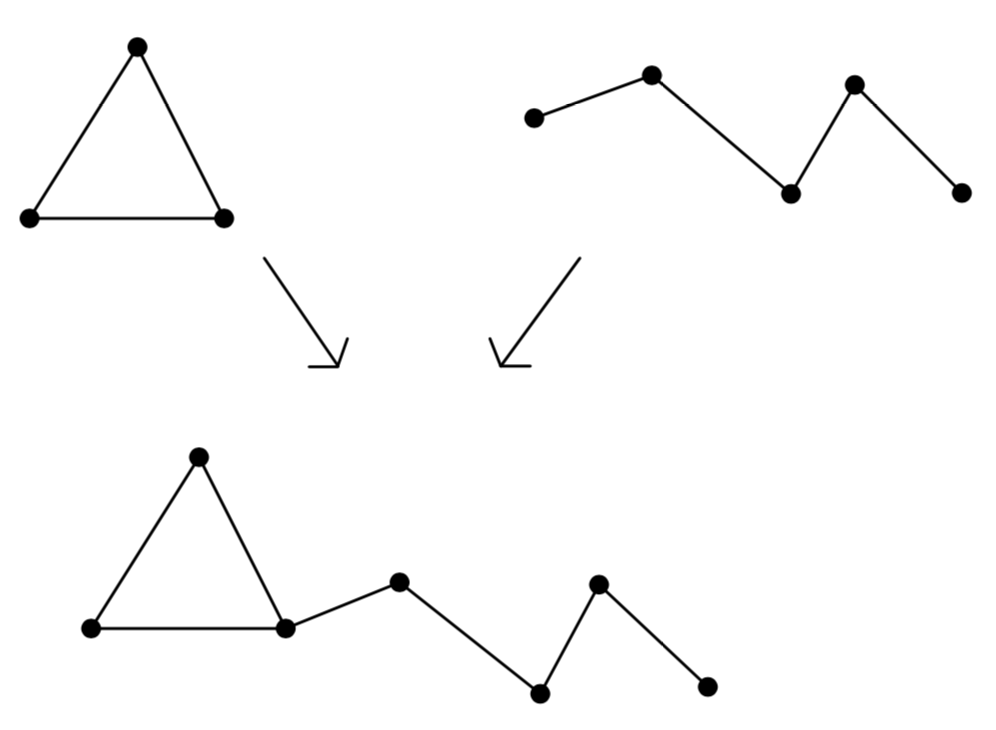}
\caption{Results for triangles and chains combine to give a kite.}
\end{figure}

\begin{proposition}\label{glue}
Let $E\subset \mathbb{R}^d$ be a compact set and let $\mu$ denote a Frostman measure on $E$.  
Suppose that there exist a pair of disjoint sets $E_1, E_2 \subset E$ so that $\mu(E_i )>0$, $i=1,2$, and, for some  $x \in E_1 $, 
$$\dim_{\mathcal{H}} 
(\{ ( |x-y_1|, |y_1-y_2|, |y_2 - x|):\, y_1, y_2 \in E_1,\, x, y_1, y_2 \text{ distinct} \})
  \geq \gamma_1>0,$$  
  and
for each $y_1 \in E_1$,
$$\dim_{\mathcal{H}}    (\{  (|y_1-y_3|, |y_3-y_4|):\, y_3, y_4 \in E_2 \text{ distinct}\})   \geq \gamma_2>0.$$
Then it holds, for some $x\in E$, that
\begin{equation}\label{kite}
\begin{split}
&\dim_{\mathcal{H}} 
(\{ (|x-y_1|,  |y_1-y_2|, |y_2 - x|, |y_1 - y_3|, |y_3-y_4|):\\
 &\qquad\qquad\qquad\qquad\qquad  y_1, y_2, y_3, y_4 \in E,\, x, y_i \text{ distinct}\}) \geq \gamma_1 + \gamma_2.
\end{split} 
\end{equation}   We call the set in \eqref{kite} a \textit{kite}.  
\end{proposition}

We note that similar results can be obtained for any choice of pinned point along the kite.  \vskip.08in

\smallskip
\subsection{On the second question: Prescribed distances}
We now turn to the second question, which can be viewed as an analogue of the unit distance problem. 
For $d,k\geq 2$, we are interested in determining the value of
\[
g_d(VS^k_{\vec{t}}, \alpha):=\sup\{\dim_{\mathcal{H}}(VS^k_{\vec{t}}(E)):\, E \text{ is a compact set in $\mathbb{R}^d$},\, \dim_{\mathcal{H}}(E)=\alpha\},
\]
where $VS^k_{\vec{t}}(E)$ is defined in (\ref{singledistance}),
and similarly the value when the set $VS^k_{\vec{t}}(E) $ is replaced by the vertex set of a tree, $ VT^k_{\vec{t}} $, or a triangle $V{\rm Tri}_{\vec{t}}(E)$ (see \eqref{vtri}). 
%and similarly $g_d(VT^k_{\vec{t}},\alpha)$ or the analogue of the quantity for more general point configurations such as triangles. 
We will omit the dependence on $\vec{t}$ from the notation above when the values of $\vec{t}$ are clear from the context.
\smallskip

%%% %%%%%  REGULAR VALUE THEOREM CONNECTION 
An additional motivation for our investigation arises from the regular value theorem from elementary differential geometry.  
The regular value theorem in elementary differential geometry says that if $\phi: X \to Y$, where $X$ is a smooth manifold of dimension $n$ and $Y$ is a smooth manifold of dimension $m<n$ with $\phi$ a submersion on the set ${\vec{\phi}}^{-1}(y)$, for $y\in Y$ fixed, then the set 
$$  {\vec{\phi}}^{-1}(y)=\{x \in X: \vec{\phi}(x)=y \}  $$
 is either empty or is a $n-m$ dimensional submanifold of $X$. 
 
 A fractal variant of the regular value theorem was obtained in \cite{EIT}, where it was shown that, under some reasonable hypotheses on $\phi: \mathbb{E\times E}\rightarrow \mathbb{R}^m$, the upper Minkowski dimension of 
$$\{(x,y) \in E \times E: \phi_l(x,y)=t_l, 1 \leq l \leq m \}$$ does not exceed $2\alpha-m$, where $E\subset \mathbb{R}^d$ is a set of Hausdorff dimension $\alpha$.  

Observe that, given any sequence of distances $\vec{t}=(t_1,\cdots, t_k)\in \mathbb{R}_+^k$, we may re-write 
%the $k$-chain set generated by $E$ with prescribed distances $\vec{t}$, 
the set $VS^k_{\vec{t}}(E)$ as
\begin{equation}\label{singledistance_regvalue}
VS^k_{\vec{t}}(E)=\{(x_1,\cdots,x_{k+1})\in E^{k+1}: \phi(x_1, \cdots, x_{k+1})=\vec{t},\,  \},
\end{equation}
where $\phi: \mathbb{R}^{k+1}\rightarrow \mathbb{R}_+^k$ is defined by $ \phi(x_1, \cdots, x_{k+1}) = (|x_1 - x_2|, \cdots, |x_k - x_{k+1}|)$. 
A direct analog of the regular value theorem would state that either $VS^k_{\vec{t}}(E)$ is empty or is a set of dimension $(k+1)\alpha - k$. We compare this to our Theorem \ref{main3}. 
\vskip.07in
%%%%%%%%%%%%%%%%%%%%%

%%REGULARITY DISCUSSION: defintion and comparison to AD regularity 
In the continuous case, similar questions for distance sets have been studied in \cite{EIT, EIT2016, OO}, and it was observed that one needs to assume some minimal regularity for the set $E$ in order for the question to be meaningful (see \cite[Page 253]{OO} for a discussion). We will work with sets that are ``$\{\delta_i\}$-discrete $\alpha$-regular'' below, which generalizes a concept that was first introduced in \cite{KT01} and used in \cite{OO} for the distance set ($k=1$) case.

Given a compact set $E\subset \mathbb{R}^d$ and $\delta>0$, denote $E_\delta:=E+B(0,\delta)$ the $\delta$-neighborhood of $E$. Let $\alpha>0$, $E$ is said to be \emph{$\{\delta_i\}$-discrete $\alpha$-regular} if there exists a sequence of positive numbers $\{\delta_i\}$ such that $\delta_i\to 0$, and for all $\epsilon>0$ sufficiently small, there holds
\begin{equation}\label{eqn: delta discrete}
|E_{\delta_i}\cap B(x,r)|\lesssim_\epsilon \left(\frac{r}{\delta_i}\right)^{\alpha+\epsilon} \delta_i^{d-\epsilon}
\end{equation}for any $x\in \mathbb{R}^d$ and $r\geq \delta_i$, $\forall i$, where $| \cdot | $ denotes the Lebesgue measure. For $X,Y\in \mathbb{R}$, we use $X\lesssim Y$ to denote the estimate $X\le c Y$ for some constant $c>0$.  

\begin{remark}\label{rmk: AD} It is not hard to see that this class of sets contains AD regular (Ahlfors-David regular) sets as special examples. Indeed, recall that if $E$ denotes an AD regular  set of Hausdorff dimension $\alpha$, then $E$ supports a Borel probability measure $\mu$ so that for each $0< r < diam(E)$ and for each $x\in E$, $cr^{\alpha} \le  \mu(B(x,r)) \le Cr^{\alpha}$, for universal constants $0<c<C$.  Letting $\delta \in (0,r]$ and $x\in E$, we can write $\mu(B(x,r)) \sim r^{\alpha} =  \delta^{\alpha}   (\frac{r}{\delta})^{\alpha}$, and deduce that the number of $\delta$-balls required to cover $E\cap B(x,r)$ is 
approximately $ (\frac{r}{\delta})^{\alpha}$.  This holds for any set of scales $(\delta, r)$ with $\delta \le r$.  It follows that, if $N$ denotes the number of $\delta$-balls needed to cover $E_{\delta}\cap B(x,r)$, then $N \sim (\frac{r}{\delta})^{\alpha}$.  Now $|E_{\delta}\cap B(x,r)| \le N \delta^d \lesssim  (\frac{r}{\delta})^{\alpha}\delta^d $.  
\end{remark}

%REGULARITY DISCUSSION: Two examples that are not AD regular.
Moreover, the class of sets considered here includes the class of $\delta$-discrete $\alpha$-regular sets considered in \cite{OO} (where (\ref{eqn: delta discrete}) is assumed to hold for all $\delta>0$ rather than only a sequence of scales $\{\delta_i\}$), and some examples that are not AD regular. 

\begin{example}\label{example: lattice}[The lattice example]
A non-trivial example of a set which is $\{\delta_i\}$-discrete $\alpha$-regular, but not AD regular is provided by a classic lattice-like construction. In more detail, we create a fractal subset of $\mathbb{R}^d$ that mimics an integer lattice as follows: 
For an integer $q_i$, consider the set $E_i:= \bigcup_{\vec{u} \in \{0, 1, \dots, q_i\}^d} \{\vec{x} \in \mathbb{R}^d: |\vec{x} - \frac{\vec{u}}{q_i}| < q_i^{-d/\alpha}\}$.  By taking a rapidly increasing sequence of $q_i$, and setting $E$ equal to the intersection of the sets $E_i$, one obtains a set of Hausdorff dimension $\alpha$. (For more details on this construction, see \cite[Theorem 8.15]{Falc86}.) 
The set $E$ is not AD regular (unless $\alpha=d$) but is $\{\delta_i\}$-discrete $\alpha$-regular by taking $\delta_i=q_i^{-d/\alpha}$ (see Appendix \ref{ADproperty} for details).  Verifying that $E$ is not AD regular can be accomplished in one of two ways:  noting that $E$ contains arbitrarily long arithmetic progressions, it follows from the work of Dyatlov and Zahl (\cite{DZahl16}, see Section 6.1.1) that it cannot be AD regular; direct computation shows that $E$ is not AD regular and we include this in Appendix \ref{ADproperty}.  
\end{example}

\begin{example}\label{example: train}[The train track example]
Another example of a set that is $\{\delta_i\}$-discrete $\alpha$-regular, but not AD regular, can be constructed from train tracks. Such an example was first studied by Katz--Tao \cite{KT01} and was further examined in \cite{GIOW} in connection with the Falconer distance problem. More precisely, consider a sequence $\{R_i\}$ rapidly increasing to infinity, and define a set $E=\cap E_i$ in the unit square in $\mathbb{R}^2$ as follows. Fix each $R_i$, $E_i$ is divided among several large $R_i^{-1/2}\times 1$ vertical rectangles, spaced by distance $R_i^{\frac{1}{2}-\frac{\alpha}{2}}$ ($1< \alpha < 2$). Within each of these large rectangles, the set $E_i$ consists of small evenly spaced parallel horizontal rectangles with dimensions $R_i^{-1}\times R_i^{-1/2}$. Each of these small horizontal rectangles is called a slat, and the spacing between two consecutive slats is $R_i^{-\frac{\alpha}{2}}$. By taking $\delta_i=R_i^{-1}$, one can verify that $E$ has Hausdorff dimension $\alpha$ and is indeed $\{\delta_i\}$-discrete $\alpha$-regular (see for instance \cite{KT01}). However, $E$ is not AD regular. To see this, suppose $E$ does support a probability measure $\mu$ satisfying $cr^\alpha\leq \mu(B(x,r))\leq C r^\alpha$, $\forall x\in E$. Let $r\in (R_i^{-1}, R_i^{-\frac{\alpha}{2}})$ and fix $x\in E$, then $E\cap B(x,r)$ is contained in a single slat hence can be covered by at most $rR_i$ balls of radius $R_i^{-1}$, which implies $\mu(B(x,r))\leq CrR_i^{1-\alpha}$. On the other hand, $\mu(B(x,r))\geq cr^{\alpha}$. Therefore $\frac{c}{C}< R_i^{(1-\frac{\alpha}{2})(1-\alpha)}$, which is impossible as $R_i\to \infty$.
\end{example}

%REGULARITY DISCUSSION: assumptions on E
Going forward in our discussion, we will assume the set $E$ to be $\{\delta_i\}$-discrete $\alpha$-regular. Without loss of generality, we will further assume $E$ to be contained in the unit ball throughout the article, hence, the condition (\ref{eqn: delta discrete}) implies in particular that $|E_{\delta_i}|\lesssim_\epsilon \delta_i^{d-\alpha-\epsilon}$, $\forall i$, $\forall \epsilon>0$ sufficiently small.
\vskip.07in

%%Discrete setting 
In the discrete setting, similar questions for chains have been very recently explored as well, see \cite{FK, PSS}. Note that in both works, the authors in fact study a slightly larger quantity than $g_d(VS^k_{\vec{t}},\alpha)$, where an additional supremum is taken over all choices of prescribed gaps $\vec{t}$. All the upper bound results we obtain in this paper also extend to this larger quantity, but we need to be more careful here for deriving the lower bound, which would depend on the particular choice of gaps $\vec{t}$. In fact, as will be explained later in the paper, different choices of gaps sometimes yield very different behaviors of $g_d(VS^k_{\vec{t}},\alpha)$.

Another distinct feature of our work, compared to the discrete setting, is that nontrivial results do exist in the $d\geq 4$ case; It is well known that the unit distance problem and the chain problems are trivial in the discrete case when $d\geq 4$, see the introduction of \cite{OO, PSS} for more details.

%RESULTS: CHAINS
\smallskip
\subsubsection{Chains}
Our first result in this direction concerns all $d\geq 2$.

\begin{theorem}\label{main3}
For all $d\geq 2$, $k\geq 1$, and prescribed gaps $\vec{t}$,
\[
g_d(VS^k_{\vec{t}}, \alpha)\begin{cases} =(k+1)\alpha-k, & \frac{d+1}{2}\leq \alpha\leq d,\\
\leq \frac{k(d-1)}{2}+\alpha, &  \alpha\leq\frac{d+1}{2}.\\
%\leq \frac{k(d-1)}{2}+\alpha, & \frac{d-1}{2} \le \alpha\leq\frac{d+1}{2}\\
%\leq    (k+1)\alpha, & \alpha\leq\frac{d-1}{2}.
\end{cases}
%\leq \frac{kd}{2}+\alpha-\frac{k}{2}, & \alpha\leq\frac{d+1}{2}.\end{cases}
\]
\end{theorem}

Note that there always holds the trivial estimate $g_d(VS_{\vec{t}}^k,\alpha)\leq (k+1)\alpha$ (see a justification in Section \ref{sec: trivial upper bound}), so the second estimate of Theorem \ref{main3} is only nontrivial when $\alpha>\frac{d-1}{2}$.

\begin{remark}\label{rmk: trivial examples} 
To put the second listed bound in Theorem \ref{main3} into context, we observe an additional upper bound which is also inferior in the regime $\alpha>\frac{d-1}{2}$. 
%and inferior to $(k+1)\alpha$ when $\alpha>\frac{d-1}{2}$. 
The proof techniques used in \cite{OO} can be used to show that, 
if $\dim_{\mathcal{H}}(E)=\alpha$, $E$ is $\{\delta_i\}$-discrete $\alpha$-regular, and $\alpha < \frac{d+1}{2}$, then the lower Minkowski dimension of the set 
$\{(x, y)\in E\times E: |x  - y | =1 \}$ is bounded above by $ \alpha + \frac{d-1}{2}$.
In more detail, the Lebesgue measure of the $\delta_i$-neighborhood of $\{(x, y)\in E\times E: |x  - y | =1 \}$ is bounded above by $\delta_i^{2d-(\alpha + \frac{d-1}{2})-\epsilon}$, $\forall i$. Observing that the vertex set of the $k$-chain set, $VS^k_{\vec{1}}(E)$, is contained in 
the set $\{(x^1, x^2)\in E\times E: |x^1 - x^{2}| =1 \} \times E^{k-1},$ and recalling that $|E_{\delta_i}|\lesssim \delta_i^{d-\alpha-\epsilon}$, $\forall i$, we use the fact that the Hausdorff dimension is bounded above by the lower Minkowski dimension to conclude that 
$g_d(VS_{\vec{1}}^k,\alpha)\leq  \alpha + \frac{d-1}{2} + (k-1)\alpha$. (A slightly better upper bound can be obtained in a similar way, by decomposing the $k$-chain into shorter components and utilizing the upper estimate for shorter chains or distances. However, results obtained in this way will always be inferior to Theorem \ref{main3} when $\alpha>\frac{d-1}{2}$.)
\end{remark}

For small values of $\alpha$, different features are displayed in different ambient dimensions.

\begin{theorem}\label{main4}
Let $k\geq 1$. 
For all $d\geq 4$ and $\alpha\leq \lfloor \frac{d}{2} \rfloor-1$, $g_d(VS_{\vec{t}}^k, \alpha)=(k+1)\alpha$, for all $\vec{t}$ satisfying $t_1\leq \cdots \leq t_k$. Moreover, when $d=2$, we have $g_2(VS^2_{\vec{t}}, \alpha)=2\alpha$ if $0<\alpha\leq 1$, for all $\vec{t}$.
\end{theorem}

\begin{remark}\label{rmk: special case}
The first lower bound in Theorem \ref{main4} relies on adapting a classic construction from the discrete setting, which utilizes orthogonal circles (see section \ref{orthogonal cirles} below for details). 
However, some relation on the $t$'s, such as $t_1\leq \cdots\leq t_k$, is required. 
It is not clear, for instance, whether the set of $3$-chains of gaps $(1,2,1)$ achieves the same lower bound (see section \ref{orthogonal cirles} below for details). However, it is possible to further relax the condition $t_1\leq \cdots\leq t_k$. For instance, when $k=3$, one can in fact show that $g_d(VS_{\vec{t}}^3,\alpha)=(k+1)\alpha=4\alpha$ whenever $t_2^2<t_1^2+t_3^2$. We give a sketch of the proof of this in Section \ref{orthogonal cirles} after the proof of Theorem \ref{main4}.
\end{remark}

When $d=2$ and $0<\alpha<1$, similar deduction as in the proof of Theorem \ref{main4} indeed gives rise to certain upper and lower bound of $g_2(VS_{\vec{t}}^k,\alpha)$, for general $k\geq 3$ as well. However, since these bounds are unlikely sharp, we omit the details. Similarly, one can straightforwardly extend the same method to study the case $d=3$, $0<\alpha<2$, but new ideas seem to be needed in order to fully solve the problem in three dimensions.

When $k=1$, and when $g_d(VS_{\vec{t}}^k,\alpha)$ is defined using only those sets $E$ satisfying a stronger regularity condition (i.e. estimate (\ref{eqn: delta discrete}) for all $\delta>0$), Theorem \ref{main3} and the first result in Theorem \ref{main4} in the above are proved in \cite{OO}. It is straightforward to see that their argument in fact also works for sets that are $\{\delta_i\}$-discrete $\alpha$-regular. Therefore we will only prove the $k\geq 2$ cases.  
We also point out that, when $k=1$ and when further restricting the sets considered to be AD regular, Theorem \ref{main3} (for $\alpha>\frac{d+1}{2}$) is first obtained in \cite{EIT}.

\begin{remark}
Both Theorem \ref{main3} and Theorem \ref{main4} readily extend to the case of trees. Since in most cases, the value of $g_d(VT^k_{\vec{t}},\alpha)$ seems to depend on the exact structure of the tree $\mathcal{T}^k$, we omit the statement of those results for the sake of simplicity and only comment on the necessary changes for the tree case along the proofs of Theorem \ref{main3} and \ref{main4}. For example, one can show that when $d\geq 2$, $k\geq 2$, and $\alpha\geq\frac{d+1}{2}$, $g_d(VT^k_{\vec{t}},\alpha)=(k+1)\alpha-k$. Another example is that when $d=2$ and $0<\alpha\leq 1$, if the $k$-tree is a star, i.e. all $k$ edges share the same vertex, then $g_2(VT^k_{\vec{t}},\alpha)=k\alpha$. 
\end{remark}
\smallskip
%RESULTS: triangles and loops

\subsubsection{Triangles and loops}
We now turn to studying finite point configurations containing closed loops.
Compared to the chain or tree case, the main difficulty here is that the existence of loops in the graph prevents one from applying an iterative proof scheme. Our proof is based on a Fourier analytic approach that involves the estimate of the decay of the Fourier transform of the surface measure of a hypersurface that encodes the structure of the point configuration. For the sake of simplicity, we only study the case of triangles in this article, even though the proof strategy can be extended to more general graphs with loops. Define
\[
g_d(V{\rm Tri}_{\vec{t}}, \alpha):=\sup \left\{\dim_{\mathcal{H}}(V{\rm Tri}_{\vec{t}}(E)):\, E \text{ is a compact set in $\mathbb{R}^d$},\, \dim_{\mathcal{H}}(E)=\alpha \right\},
\]where
\begin{equation}\label{vtri}
V{\rm Tri}_{\vec{t}}(E):=\{(x,y,z)\in E^3:\, |x-y|=t_1, |y-z|=t_2, |z-x|=t_3\}.
\end{equation}
Again, we always assume that the set $E$ is $\{\delta_i\}$-discrete $\alpha$-regular. 

\begin{theorem}\label{main5}
Let $d\geq 3$.  Then for all $\vec{t}$,
\[
g_d(V{\rm Tri}_{\vec{t}}, \alpha)\leq \begin{cases} 3\alpha-3,& \frac{2d}{3}+1\leq\alpha\leq d,\\ d+\frac{3\alpha}{2}-\frac{3}{2},& 0<\alpha\leq \frac{2d}{3}+1. \end{cases}
\]Moreover, if $d=2$, it holds for all $\vec{t}$ that
\[
g_2(V{\rm Tri}_{\vec{t}}, \alpha)\leq \begin{cases} 
3\alpha-3,& \frac{7}{4}\leq\alpha\leq 2,\\ 2\alpha-1,& \frac{3}{2}\leq\alpha < \frac{7}{4},\\ 
 \alpha+\frac{1}{2}, & 1\leq \alpha\leq \frac{3}{2},\\
 \min\{\frac{5\alpha}{3}, \frac{\alpha(2+\alpha)}{1+\alpha}\}, & 0<\alpha\leq 1. \end{cases}
\]Furthermore, if $d\geq 6$ and $\alpha\leq \lfloor \frac{d}{3}\rfloor-1$, then $g_d(V{\rm Tri}_{\vec{t}}, \alpha)=3\alpha$ whenever $\vec{t}=(t_1,t_2,t_3)$ forms an acute triangle.
\end{theorem}

We note that this theorem fits nicely into the current results in the field.  
In \cite{GI}, Greenleaf and Iosevich prove that if  $E$ is a compact subset of $\mathbb{R}^2$ of $\dim_{\mathcal{H}}(E) >\frac{7}{4}$, then the set of triples of distances formed by triangles in $E$, 
\begin{equation}\label{tri}
{\rm Tri}(E):=\{(|x-y|, |y-z|, |x-z|): x, y, z \in E \},
\end{equation}
has positive $3$-dimensional Lebesgue measure.  
The authors define a measure on the set  ${\rm Tri}(E)$ and proving that its density is in $L^{\infty}$.  
In a subsequent paper using group actions and an $L^2$ estimate on the density, it is shown that in all dimensions $d\geq 2$, if $\dim_{\mathcal{H}}(E) >\frac{2d+1}{3}$, then ${\rm Tri}(E)$ has positive Lebesgue measure (\cite{GILP}).  
In \cite{IL}, it is proved that for any $d\geq 4$, there exists an $\delta>0$ so that if $\dim_{\mathcal{H}}(E) >d-\delta$, then $E$ contains the vertices of an equilateral triangle.  Theorem \ref{main5} above shows that, for any $d\geq 2$, it is possible to control the number of occurrences of such triangles through an upper bound on the Hausdorff dimension.  

\begin{remark}
We note a simple transference mechanism between the unit distance problem and the distinct distance problem, which holds when we impose the additional regularity assumption that the set $E$ is AD regular. 
In this case, it is a straightforward exercise to verify that the results in Theorem \ref{main5} imply a lower bound on the set 
$$\dim_{\mathcal{H}}\left\{ (|x-y|, |y-z|, |x-z|): x,y, z \in E \right\}.$$
We omit this result, however, as superior lower bounds are obtained in \cite{Yu18}.
\end{remark}

\begin{remark}
Note that one always has the trivial estimate $g_d(V{\rm Tri}_{\vec{t}}, \alpha)$ is no more than $\min(3\alpha, 3d-3)$. Indeed, the first bound follows easily from $V{\rm Tri}_{\vec{t}}(E)\subset E^3$ and $E$ is $\{\delta_i\}$ discrete $\alpha$-regular (hence $\dim_{\mathcal{H}}(E^3)=3\alpha$), and the second bound can be obtained by considering the map $(x,y,z)\mapsto (x-y, y-z, z-x)$ and observing that the image set is contained in $t_1S^{d-1}\times t_2 S^{d-1}\times t_3 S^{d-1}$. Therefore, the bound $g_d(V{\rm Tri}_{\vec{t}}, \alpha)\leq d+\frac{3\alpha}{2}-\frac{3}{2}$ is only nontrivial if $\alpha>\frac{2d}{3}-1$. When $d=2$, the upper bound in the theorem above when $\alpha< \frac{7}{4}$ can be obtained from the trivial bound $g_2(V{\rm Tri}_{\vec{t}},\alpha)\leq g_2(VS_{t_i}^1, \alpha)$, $\forall i=1,2,3$ and the estimate for the unit distance problem in \cite[Theorem 1.3]{OO}. To see this, fix two vertices $x$ and $y$ at a distance $t_1$ apart from each other, and observe that there are at most two choices of the third vertex $z$. 
\end{remark}

\begin{remark}\label{lower triangle}
For $\alpha>d-1$, a lower bound of $\alpha+ (d-2) $ is achievable in all dimensions. 
Let $A\subset \mathbb{R}^d$ so that $\dim_{\mathcal{H}}(A)=:a \in (0,1)$.  
Set $E = A \bigcup\{A + S^{d-1}\}$, the Minkowski sum of $A$ and the unit sphere.  
Inspecting the energy integral of $E$ and observing that sum sets are Lipschitz images of Cartesian product sets, 
one may verify that $\dim_{H} (E) =: \alpha =  a + (d-1)$ (see, for instance, the work of the second listed author with K. Simon: \cite{SimonTaylordim}, \cite{SimonTaylorInt}).  
Observe that, for each $x\in A$, $\left( x+S^{d-1} \right) \bigcap E = x+S^{d-1} $.  Now, for each $x\in A$, 
\begin{align*}
&\{(y, z) \in E\times E: |x-y| = 1, |x-z|=1, \text{ and }\,\, |y-z|=1\}\\
= &\{(y, z) \in \left( x+S^{d-1} \right)   \times   \left( x+S^{d-1} \right) : |y-z|=1\},
\end{align*}
and this set clearly has Hausdorff dimension $(d-1) + (d-2)$.  
%Observing that 
%$${\rm VTri}_{\vec{1}}(E) \supset 
%\{(x, y, z) \in A\times E \times E: |x-y| = 1, |x-z|=1, \text{ and }\,\, |y-z|=1\}$$
Restricting $x$ to $A$, 
it follows by Corollary \ref{cor: Fubini} in Appendix \ref{sec: Fubini} that 
 $$\dim_{\mathcal{H}}(V{\rm Tri}_{\vec{1}}(E)) \geq a + (d-1) + (d-2) = \alpha + (d-2). $$

\end{remark}

%%% PHONG STEIN CONDITION
\bigskip
\subsubsection{Phong-Stein condition}\label{PS_section}
The question of whether it is possible to replace the Euclidean distance in the above discussion (for instance \eqref{singledistance}) with a more general metric was raised in \cite[Page 255]{OO}.  
In this section, we answer this question in the affirmative not only for the unit distance problem but also for the general $k$-chain. 
Moreover, we show that it is possible to 
 replace the Euclidean distance with a more general function $\phi(x,y)$ which satisfies the rotational curvature conditions introduced by Phong and Stein.  

In particular, we 
consider 
$\phi: \mathbb{R}^d \times \mathbb{R}^d \rightarrow \mathbb{R}$ to be a continuous, infinitely differentiable function, which satisfies
\begin{equation}\label{nondeg}  
|\nabla_x\phi(x,y)|\neq 0 \text{ \, \, and \, \,  } \ \  |\nabla_y\phi(x,y)|\neq 0. \end{equation}
Further, we assume that $\phi$ satisfies 
the non-vanishing Monge-Ampere determinant assumption:
\begin{equation} \label{mongeampere} det
\begin{pmatrix}
 0 & \nabla_{x}\phi \\
 -{(\nabla_{y}\phi)}^{T} & \frac{\partial^2 \phi}{dx_i dy_j}
\end{pmatrix} \end{equation} does not vanish on the set $\{(x,y): \phi(x,y)=t \}$, $t \neq 0$. 
\vskip.125in

\begin{remark} Examples of such functions include the dot product, $\phi(x,y) = x\cdot y$, as well as any norm, $\phi(x,y) = \|\cdot\|_{B}$, generated by a smooth convex body, $B$, with non-vanishing curvature.  
\end{remark}

Given any sequence of distances $\vec{t}=(t_1,\cdots, t_k)\in \mathbb{R}_+^k$,  we define 
\begin{equation}\label{singledistance phi}
VS^k_{\vec{t}, \phi}(E):=\{(x_1,\cdots,x_{k+1})\in E^{k+1}:\, \phi(x_i,x_{i+1})=t_i,\, i=1,\ldots, k,\, \{x_i\}\, \text{distinct}\}
\end{equation}
as the $(k,\phi)$-chain set generated by $E$ with prescribed gaps $\vec{t}$.

We are interested in determining the value of
\[
g_d(VS^k_{\vec{t}, \phi}, \alpha):=\sup\{\dim_{\mathcal{H}}(VS^k_{\vec{t},\phi}(E)):\, E \text{ is a compact set in $\mathbb{R}^d$},\, \dim_{\mathcal{H}}(E)=\alpha\}.
\]For ease of notation, 
%we set each $t_i=1$ throughout the discussion in this subsection, and 
we drop the subscript $\vec{t}$ throughout the discussion in this subsection.

% DELTA-DISCRETE REGULAR CASE
We now turn to the main result of this subsection, in which we show that, in the case that the set $E$ is assumed to be $\{\delta_i\}$-discrete $\alpha$-regular (see equation \eqref{eqn: delta discrete} for the definition), the following analogue of Theorem \ref{main3} holds.  

\begin{theorem}\label{main3PhongStein}
Let $\alpha>0$ and $d\geq 2$.
Let $\phi: \mathbb{R}^d \times \mathbb{R}^d \rightarrow \mathbb{R}$ denote a smooth function which satisfies 
the gradient conditions in \eqref{nondeg}, as well as 
the curvature condition in \eqref{mongeampere}. 
Then, for all $k\geq 1$, 
\[
g_d(VS^k_{\phi}, \alpha)\begin{cases} = (k+1)\alpha-k, & \frac{d+1}{2}\leq \alpha\leq d,\\
\leq \frac{k(d-1)}{2}+\alpha, & \alpha\leq\frac{d+1}{2}.
%\leq     (k+1)\alpha, & \alpha\leq\frac{d-1}{2}.
\end{cases}
%\leq \frac{kd}{2}+\alpha-\frac{k}{2}, & \alpha\leq\frac{d+1}{2}.\end{cases}
\]
\end{theorem}
\vskip.125in

The proof of Theorem \ref{main3PhongStein} can be found in Section \ref{proof of main3PhongStein}, and relies on merging the ideas introduced in \cite{ITU} and in \cite{OO}. 
Equality is attained using a simple adaptation of the proof presented in Section \ref{section lower bound}.  Note that there always holds the trivial estimate $g_d(VS^k_{\phi}, \alpha) \le (k+1)\alpha$, which is inferior to the bounds in Theorem \ref{main3PhongStein} provided  $\alpha>\frac{d-1}{2}$.

{\begin{remark} 
In Appendix \ref{parabolic}, we consider an example $E$ where $\phi$ is given by a paraboloid-like surface, and we show that the upper Minkowski dimension of $VS^k_{\phi}(E)$, for this choice of $\phi$ and $E$, is bounded below by 
$ \alpha + \frac{\alpha (d-1)k}{d+1}$.  
% It would be interesting to determine the exact value of $g_d^\mathcal{M}(VS^k_{\vec{t}},\alpha)$ (defined in the same way as $g_d(VS^k_{\vec{t}},\alpha)$ but with Hausdorff dimension replaced by the upper Minkowski dimension), and its analogues for other point configurations and general $\phi$. 
\end{remark}}

\begin{remark}\label{implicit_ADresults}
In the special case that the set $E$ is AD regular (as defined in Remark \ref{rmk: AD}), it is an immediate corollary of the work of the second listed author with Iosevich and Uriarte-Tuero \cite{ITU} that $\dim_{\mathcal{H}}(VS^k_{\phi}(E)) \leq (k+1) \alpha - k$, whenever $\frac{d+1}{2} < \alpha $.  Similarly, if $E$ is assumed to be $AD$ regular, then inspecting the proof of Greenleaf and Iosevich \cite{GI} and that of Iosevich and Liu \cite{IL} recovers part of Theorem \ref{main5}, 
when $\alpha>\frac{7}{4}$ ($d=2$) and when $\alpha>\frac{2d}{3}+1$ ($d\geq 4$), respectively.  
\end{remark}

Before ending the introduction, we point out that all the upper bound estimates in Theorems \ref{main3}, \ref{main4}, \ref{main5}, and \ref{main3PhongStein} work not only for Hausdorff dimension (of the set $VS^k_{\vec{t}}$ for instance) but also for the slightly larger \textit{lower Minkowski dimension}, which is straightforward to see from the proofs.

The article is organized as follows. We study the first question in Sections \ref{sec: Lebesgue} and \ref{sec: dim}: Theorem \ref{main1} is proved in Section \ref{sec: Lebesgue}, while Theorem \ref{main2}, Theorem \ref{main2.5}, and Proposition \ref{glue} are proved in Section \ref{sec: dim}. 
Sections \ref{sec: unit}, \ref{sec: triangle}, \ref{proof of main3PhongStein} are devoted to the study of the second question: chains and trees (Theorems \ref{main3} and \ref{main4}) are treated in Section \ref{sec: unit}, triangles  (Theorem \ref{main5}) appear in Section \ref{sec: triangle}, and the general $\phi$ case (Theorem \ref{main3PhongStein}) is studied in Section \ref{proof of main3PhongStein}.

%%%%%  PROOFS
\bigskip
\section{Lebesgue measure of pinned chains/ trees: Proof of Theorem \ref{main1}} \label{sec: Lebesgue}

The main ingredient of the proof is the following structure theorem, which works in all dimensions and does not require any assumption on the value of $\alpha$.

\begin{proposition}\label{structure1}
Let $d\geq 2$ and $\alpha>0$. Suppose for all pairs of compact sets $E_1, E_2\subset \mathbb{R}^d$ with positive $\alpha$-dimensional Hausdorff measure, letting $\mu_1$, $\mu_2$ be Borel probability measures supported on $E_1, E_2$ respectively which satisfy $\mu_i(B(x,r)) \lesssim r^{\alpha}$ for $i=1,2$, then
\[
\mu_2(G_{E_1}(E_2)):=\mu_2(\left\{  x \in E_2:  \left|   \Delta_x(E_1)\right| >0  \right\} )>0,
\]where $\Delta_x(E):=\{|x-y|:\,y\in E\}$ denotes the pinned distance set.

Then, for all integers $k\geq 1$, and all pairs of compact sets $E_1, E_2\subset \mathbb{R}^d$ with positive $\alpha$-dimensional Hausdorff measure, there exists $x\in E_2$ such that $|T_x^k(E_1)|_k>0$, for all $k$-trees $\mathcal{T}_v^k$  of any shape pinned at any vertex. In particular, $|S_x^k(E_1)|_k>0$.
\end{proposition}

The notation $G_{E_1}(E_2)$ above means the ``good pins'' in $E_2$ w.r.t. $E_1$. 

\begin{proof}
We will in fact prove a stronger result: for all integers $k\geq 1$, all pairs of compact sets $E_1 ,E_2 \subset \mathbb{R}^d$ with positive $\alpha$-dimensional Hausdorff measure, letting $\mu_1$, $\mu_2$ be Borel probability measures supported on $E_1, E_2$ respectively which satisfy $\mu_i(B(x,r)) \lesssim r^{\alpha}$ for $i=1,2$, then
\begin{equation}\label{eqn: Gk}
\mu_2(G^k_{E_1}(E_2)):=\mu_2(\{x\in E_2:\, |T^k_x(E_1)|_k>0,\, \text{for all $k$-trees }\mathcal{T}_v^k\})>0.
\end{equation}

Our strategy is to prove by induction. The base case $k=1$ is precisely the same as the assumption, hence is obviously true. Assume that the desired result holds for trees whose number of edges is no greater than $k-1$.  

Let $E_1, E_2\subset \mathbb{R}^d$ be two sets such that there exist probability measures $\mu_1, \mu_2$ supported on $E_1, E_2$ respectively with $\mu_i(B(x,r)) \lesssim r^{\alpha}$ for $i=1,2$. It suffices to show (\ref{eqn: Gk}). In fact, fix any particular $k$-tree $\mathcal{T}^k=\mathcal{T}_v^k$, it suffices to prove that
\begin{equation}\label{eqn: Gk1}
\mu_2(G^k_{E_1,\mathcal{T}^k}(E_2)):=\mu_2(\{x\in E_2:\, |T^k_x(E_1)|_k>0\})>0.
\end{equation}Indeed, for any $k$ fixed, there are only finitely many possibilities of shapes and vertices to be pinned, in other words, finitely many choices of $\mathcal{T}^k$. If (\ref{eqn: Gk1}) is true for some $\mathcal{T}^k$, then one can replace $E_2$ by $G^k_{E_1,\mathcal{T}^k}(E_2)$ and iterate the argument. We omit the details and fix a choice of $\mathcal{T}^k$  (with a fixed choice of a vertex to be pinned) from now on. Moreover, by properly shrinking $E_1, E_2$ if needed, one can assume without loss of generality that $E_1, E_2$ are disjoint.

If the fixed vertex is connected to only a single edge, then one applies the induction hypothesis inside the set $E_1$. More precisely, one can find two subsets $E_{1,1}$, $E_{1,2}$ of $E_1$ and measures $\mu_{1,1}$ and $\mu_{1,2}$ supported on them respectively such that $\mu_{1,i}(B(x,r)) \lesssim r^{\alpha}$ for $i=1,2$ \footnote{The existence of such sets is guaranteed for instance by Theorem 2.3 of \cite{Falc86}.}, and
\[
\begin{split}
&\mu_{1,2}(G^{k-1}_{E_{1,1}}(E_{1,2}))\\
=&\mu_{1,2}(\{x\in E_{1,2}:\, |T^{k-1}_x(E_{1,1})|_{k-1}>0,\, \text{for all $(k-1)$-trees }\mathcal{T}^{k-1}\})>0.
\end{split}
\] 

Moreover, according to the assumption of the theorem, for the pair of sets $G^{k-1}_{E_{1,1}}(E_{1,2})$ and $E_2$, one must have
\[
\mu_2(G_{G^{k-1}_{E_{1,1}}(E_{1,2})}(E_2))=\mu_2(\{x\in E_2:\, |\Delta_x(G^{k-1}_{E_{1,1}}(E_{1,2}))|>0\})>0.
\]

It is easy to see that (\ref{eqn: Gk1}) will be implied by $G_{G^{k-1}_{E_{1,1}}(E_{1,2})}(E_2)\subset G^k_{E_1, \mathcal{T}^k}(E_2)$. To see the inclusion, fix any $x\in G_{G^{k-1}_{E_{1,1}}(E_{1,2})}(E_2)$. By definition, this means $|\Delta_x(G^{k-1}_{E_{1,1}}(E_{1,2}))|>0$, and our goal is to prove $|T^k_{x}(E_1)|_k>0$. Since there is only one edge connecting to the vertex pinned at $x$, it is straightforward to see that
\[
|T^k_x(E_1)|_k\geq \int_{t\in \Delta_x(G^{k-1}_{E_{1,1}}(E_{1,2}))}  |(T^{k}_{x}(E_1))_t|_{k-1}\,dt
\]where $(T^{k}_{x}(E_1))_t$ denotes the slice of the set $T^k_x(E_1)$ with the first variable being fixed at $t$. According to the Fubini theorem, the integral on the right hand side is well defined and one can conclude that $|T^k_x(E_1)|_k>0$ if there holds $|(T^{k}_{x}(E_1))_t|_{k-1}>0$ for all $t$. For any fixed $t$, to see why $|(T^{k}_{x}(E_1))_t|_{k-1}>0$, let $y_t$ be any point in $G^{k-1}_{E_{1,1}}(E_{1,2})$ satisfying $|x-y|=t$, and $\mathcal{T}^{k-1}$ be a $(k-1)$-tree of a particular shape (determined by the shape of $\mathcal{T}^k$) with a particular vertex to be pinned. By definition of $G^{k-1}_{E_{1,1}}(E_{1,2})$, one has $|(T^{k}_{x}(E_1))_t|_{k-1}\geq |T^{k-1}_{y_{t}}(E_{1,1})|_{k-1}>0$. Hence, according to the Fubini theorem and observing that $k$-trees produced in this way are all non-degenerate, one obtains $|T^k_{x}(E_1)|_k>0$.

Next, suppose that the fixed vertex is connected to at least two edges of the tree $\mathcal{T}^k$, then one can decompose the tree $\mathcal{T}^k$ into two sub-trees $\mathcal{T}^{k_1}_1$, $\mathcal{T}^{k_2}_2$, each containing $k_i\geq 1$ edges, $i=1, 2$, so that $\mathcal{T}^k=\mathcal{T}^{k_1}_1\cup \mathcal{T}^{k_2}_2$, and $\mathcal{T}^{k_1}_1$, $\mathcal{T}^{k_2}_2$ only share the vertex to be pinned. It is easy to see that $k_1+k_2=k$ and $k_i\leq k-1$, $i=1,2$. 

For sets $E_1,E_2$ satisfying the assumption, one further finds two disjoint subsets $E_{1,1}, E_{1,2}\subset E_1$ and probability measures $\mu_{1,1}, \mu_{1,2}$ as before. Apply the induction hypothesis to the pair $E_{1,1}, E_2$ first. Since $k_1\leq k-1$, one has 
\[
\mu_2(G^{k_1}_{E_{1,1},\mathcal{T}_1^{k_1}}(E_2))=\mu_2(\{x\in E_2:\, |T_x^{k_1}(E_{1,1})|_{k_1}>0\})>0.
\]One then applies the induction hypothesis again, this time to the pair of sets $E_{1,2}, G^{k_1}_{E_{1,1},\mathcal{T}_1^{k_1}}(E_2)$. Since $k_2\leq k-1$, one obtains
\[
\begin{split}
&\mu_2\left(G^{k_2}_{E_{1,2}, \mathcal{T}^{k_2}_2}\Big(G^{k_1}_{E_{1,1},\mathcal{T}_1^{k_1}}(E_2)\Big)\right)\\
=&\mu_2\left(\Big\{x\in G^{k_1}_{E_{1,1},\mathcal{T}_1^{k_1}}(E_2):\, |T_x^{k_2}(E_{1,2})|_{k_2}>0\Big\}\right)>0.
\end{split}
\](Strictly speaking, before applying the induction hypothesis in the second step above, one should have adjusted the measure $\mu_2$ to make it into a probability measure on the smaller set $G^{k_1}_{E_{1,1},\mathcal{T}_1^{k_1}}(E_2)$ by multiplying a constant. We omit the treatment of this issue.)

For any $x\in G^{k_2}_{E_{1,2}, \mathcal{T}^{k_2}_2}\Big(G^{k_1}_{E_{1,1},\mathcal{T}_1^{k_1}}(E_2)\Big)$, observe that 
\[
|T^k_x(E_1)|_k\geq |T_x^{k_1}(E_{1,1})|_{k_1}\cdot |T_x^{k_2}(E_{1,2})|_{k_2}>0.\]Hence,
\[
\begin{split}
\mu_2(G^k_{E_1,\mathcal{T}^k}(E_2))=&\mu_2(\{x\in E_2:\, |T_x^k(E_1)|_k>0\})\\
\geq &\mu_2\left(G^{k_2}_{E_{1,2}, \mathcal{T}^{k_2}_2}\Big(G^{k_1}_{E_{1,1},\mathcal{T}_1^{k_1}}(E_2)\Big)\right)>0.
\end{split}
\]The proof of Proposition \ref{structure1} is complete. 
\end{proof}

Then, combined with the following Lemma \ref{lem: key}, the $E_1=E_2$ case of Proposition \ref{structure1} immediately implies Theorem \ref{main1}. Lemma \ref{lem: key}, although never stated explicitly, follows from the proof of Theorem 1.2 in \cite{GIOW}.

\begin{lemma}\label{lem: key}
Let $E_1, E_2\subset \mathbb{R}^2$ be a pair of compact sets with positive $\alpha$-dimensional Hausdorff measure for some $\alpha>\frac{5}{4}$.  
%Without loss of generality, assume that $E$ is contained in the unit disk.  
Further, suppose that there exist Borel probability measures $\mu_1$ and $\mu_2$ on $E_1$ and $E_2$ respectively which satisfy $\mu_i(B(x,r)) \lesssim r^{\alpha}$ for $i=1,2$.  
 Then
\begin{equation}\label{gold}\mu_2(G_{E_1}(E_2))=\mu_2(\left\{  x \in E_2:  \left|   \Delta_x(E_1)\right| >0  \right\} )>0.\end{equation}
\end{lemma}

In fact, even though the good pin point $x$ claimed in Proposition \ref{structure1} and hence Theorem \ref{main1} seems to depend on $k$, one can easily find a good pin point $x$ that works well for all $k$, as claimed in Theorem \ref{main1}. To see this, given a compact set $E$, let $E_1,E_2\subset E$ be as before, and let $G_k\subset E_2$ denote the set of ``good pins'' in $E_2$ such that $\mu_2(G_k)>0$ and $|T_x^k(E_1)|_k>0$, $\forall x\in G_k$, for all $k$-trees. Without loss of generality, one can assume that $G_k$ is compact. Then, repeat the process for $k+1$ with $E_2$ replaced by $G_k$. One can obtain a compact good pin set $G_{k+1}\subset G_k$. Iterate the process and let $G=\bigcap_{k=1}^\infty G_k\subset E_2$. By compactness, one has $G\neq\emptyset$, and it is obvious that any point $x\in G$ will guarantee that $|T_x^k(E)|_k>0$ for all $k$-tree and all $k\geq 1$.

For the sake of completeness, we conclude this section by sketching below the proof of Lemma \ref{lem: key}. We use the notation introduced in \cite{GIOW} below.

\begin{proof}[Proof of Lemma \ref{lem: key}]
The Lemma follows from the proof of the main result in \cite{GIOW}.  We briefly outline how this works. %A more detailed discussion of the argument below will be given in the proof of Theorem \ref{main2}.
 
Let $\alpha$, $E_1$ and $E_2$ be as in the statement of the Lemma. Without loss of generality, assume $E_1$ and $E_2$ have distance $\gtrsim 1$.
By Frostman's Lemma, each $E_i$ supports a Borel probability measure $\mu_i$ so that 
$$\mu_i(B(x,r)) \lesssim r^{\alpha}.$$
Set $d(x,y) = |x-y|$, and, for $x$ fixed and $i\in \{1,2\}$, denote the pushforward measure 
$$\int_{\mathbb{R}}\psi(t) d_*^x(\mu_i): = \int_E \psi(|x-y|) d\mu_i(y).$$ 
Now $d_*^x(\mu_i)$ is a probability measure on $\Delta_x(E_i)$.  

Let $\mu_{1,good}$ be the complex measure (dependent on  $\mu_1$) described on page 7 of \cite{GIOW}.  
Proposition 2.1 in \cite{GIOW} (see page 8) implies that there exists a set $E_2' \subset E_2$ so that $\mu_2(E_2') >1 - \frac{1}{1,000}$, and for each $x\in E_2'$, 
\begin{equation}\label{prop1}\|  d_*^x(\mu_{1})  - d_*^x(\mu_{1, good})  \|_{L^1}  <1/1000\end{equation} and so
$$\int |d_*^x(\mu_1, good) | \geq 1 - \frac{1}{1000}.$$
Proposition 2.2 in \cite{GIOW} implies that, for $\mu_2$-almost every $x\in E_2$, 
\begin{equation}\label{prop2}  \|  d_*^x(\mu_1, good) \|^2_{L^2}  <\infty.\end{equation} 

Let $E_2''$ denote the subset of $E_2$ for which (\ref{prop2}) holds, and set $\tilde{E_2} = E_2' \cap E_2''.$  Then, $\mu_2(\tilde{E_2})>1 - \frac{1}{1000}$ and, following the logic on page 8 of \cite {GIOW}, for each $x_2\in \tilde{E_2}$, 
$$|\Delta_{x_2}(E_1)|>0.$$
\end{proof}

\begin{remark}\label{lattice example} Just as in the case of distances, it follows by the classical lattice example that $\frac{d}{2}$ is indeed the lowest possible threshold to ensure that $|S^k(E)|_k>0$, $\forall k\geq 2$ where $E\subset \mathbb{R}^d$. Consider the lattice example given in Example \ref{example: lattice}, which is a set of dimension $\alpha$. Then $|S^k(E)|_k=0$, $\forall k\geq 1$, whenever $\alpha<\frac{d}{2}$; (This fact is proved in \cite{Falc86} (see Theorem 2.4) in the case $k=1$, and the same line of reasoning yields the result for general $k\geq 1$). This, in particular, suggests that if the pinned version of the Falconer distance conjecture, which says that there exists $x\in E$ such that $|\Delta_x(E)|>0$, whenever $\dim_{\mathcal{H}}(E)>\frac{d}{2}$, is confirmed, then our method would be able to extend it to fully resolve the analogous question for chains.
\end{remark}

\bigskip
\section{Dimension of pinned chain/tree sets: Proof of Theorems \ref{main2}, \ref{main2.5}}\label{sec: dim}

\subsection{Proof of Theorem \ref{main2}}
We first prove Theorem \ref{main2}. The key observation here is that there holds a partial version of the Fubini theorem that can be used to estimate the Hausdorff dimension of a set based on the dimensions of its slices. The exact statement of the theorem is presented in Appendix \ref{sec: Fubini}. 

We begin with the following lemma, which is rephrased from \cite[Theorem 1.1]{Liu19}. We refer the reader to \cite{Liu19} for its proof. In particular, the proof depends on the core idea of good and bad measures that is recalled earlier in the sketch of the proof of Lemma \ref{lem: key}.

\begin{lemma}\label{lem: dim}
Let $E\subset \mathbb{R}^2$ be a compact set with $\dim_{\mathcal{H}}(E)>1$ and $\tau\in (0,1)$. Then,
\begin{equation}\label{eqn: Liu1}
\dim_{\mathcal{H}}\{x\in \mathbb{R}^2:\, \dim_{\mathcal{H}}(\Delta_x(E))<\tau\}\leq \max(2+3\tau-3\dim_{\mathcal{H}}(E), 2-\dim_{\mathcal{H}}(E)).
\end{equation}In particular, let $E_1$ and $E_2$ be subsets in $\mathbb{R}^2$ with $\dim_{\mathcal{H}}(E_1)=\dim_{\mathcal{H}}(E_2)>1$, then for all $\epsilon>0$, there exists $E_2'\subset E_2$ with $\dim_{\mathcal{H}}(E_2\setminus E_2')<\dim_{\mathcal{H}}(E_2)$ so that 
\begin{equation}\label{eqn: Liu2}
\dim_{\mathcal{H}}(\Delta_x(E_1))\geq \min\left(\frac{4}{3}\dim_{\mathcal{H}}(E_1)-\frac{2}{3}-\epsilon, 1\right),\quad \forall x\in E_2'.
\end{equation}

\end{lemma}

Similarly as in the previous section, we have the following structural theorem that allows one to extend the dimension estimates of pinned distance sets to pinned tree sets. This theorem does not assume anything on the value of $\alpha$.

\begin{proposition}\label{structure2}
Let $d\geq 2$. Suppose for all compact sets $E_1, E_2\subset \mathbb{R}^d$ with $\dim_{\mathcal{H}}(E_1)=\dim_{\mathcal{H}}(E_2)>\alpha_0>0$, there exists $E_2'\subset E_2$ with $\dim_{\mathcal{H}}(E_2\setminus E_2')<\dim_{\mathcal{H}}(E_2)$ so that
\begin{equation}\label{structure2assumption}
\dim_{\mathcal{H}}(\Delta_x(E_1))\geq \gamma=\gamma(\dim_{\mathcal{H}}(E_1),d),\quad \forall x\in E_2'.
\end{equation}

Then, for all integers $k\geq 1$, and all compact sets $E_1, E_2\subset \mathbb{R}^d$ with $\dim_{\mathcal{H}}(E_1)=\dim_{\mathcal{H}}(E_2)>\alpha_0$, there exists $E_2'\subset E_2$ with $\dim_{\mathcal{H}}(E_2\setminus E_2')<\dim_{\mathcal{H}}(E_2)$ so that
\begin{equation}\label{structure2star}
\dim_{\mathcal{H}}(T_x^k(E_1))\geq k\gamma,\quad \forall x\in E_2',
\end{equation}
for all $k$-trees $\mathcal{T}^k_v$ of a particular shape pinned at a particular vertex. In particular, $\dim_{\mathcal{H}}(S_x^k(E_1))\geq k\gamma,\quad \forall x\in E_2'$.
\end{proposition}

It is easy to see that the second statement of Theorem \ref{main2} follows immediately from (\ref{eqn: Liu2}) and Proposition \ref{structure2} with $E_1=E_2$ and $\gamma=\min\left(\frac{4}{3}\dim_{\mathcal{H}}(E_1)-\frac{2}{3}-\epsilon, 1\right)$.  In order to prove the first statement of Theorem \ref{main2}, one simply takes a sequence $\{\epsilon_n\}$ that converges to $0$.

% Proof of dimensional infinite chain assertion 
Before moving on to the proof of Proposition \ref{structure2}, we momentarily take this result for granted and demonstrate the third assertion of Theorem \ref{main2}, equation \eqref{infinite_chain_dim}.  

Given a compact set $E$, suppose $1<\dim_{\mathcal{H}}(E)\leq\frac{5}{4}$. Let $E_1,E_2$ be two disjoint subsets of $E$, both with the same Hausdorff dimension as $E$. Let $\epsilon_0>0$ be a fixed small parameter such that $\dim_{\mathcal{H}}(E)-\epsilon_0>1$. Then, for all $\epsilon<\epsilon_0$, set 
$$G_1 = \left\{ x\in E_2: \,\dim_{\mathcal{H}}(  T_x^1(E_1) ) \geq \frac{4}{3}\dim_{\mathcal{H}}(E)-\frac{2}{3}-\epsilon  \right\}.$$  
By Lemma \ref{lem: dim} and Proposition \ref{structure2} applied to $E_1, E_2$ with $k=1$, $\dim_{\mathcal{H}}(E_2\setminus G_1)<\dim_{\mathcal{H}}(E)$, which in particular implies $\dim_{\mathcal{H}}( G_1 ) = \dim_{\mathcal{H}}(E).$

Let $\dim_{\mathcal{H}}(E) >\alpha_1 >\dim_{\mathcal{H}}(E)-\frac{\epsilon}{2}$ and choose $\tilde{G_1}\subset G_1$ so that $\tilde{G_1}$ is compact and $\mathcal{H}^{\alpha_1}(\tilde{G_1}) >0$ (such a choice is possible, for instance, by Corollary 4.12 in \cite{Falc95}). Then in particular, $\dim_{\mathcal{H}}(\tilde{G_1})\geq \alpha_1$, and for all $x\in \tilde{G_1}$,
\[
\dim_{\mathcal{H}}(T^1_x(E_1))\geq \frac{4}{3}\dim_{\mathcal{H}}(E)-\frac{2}{3}-\epsilon> \frac{4}{3}(\dim_{\mathcal{H}}(E)-\epsilon_0)-\frac{2}{3}.
\]

Replacing $(E_1,E_2)$ with $(E_1^{(2)},\tilde{G_1})$ where $E_1^{(2)}\subset E_1$ satisfies $\dim_{\mathcal{H}}(E_1^{(2)})=\dim_{\mathcal{H}}(\tilde{G_1})$ and repeating this process, one finds $\alpha_2$ so that $\dim_{\mathcal{H}}(\tilde{G_1})-\frac{\epsilon}{20}<\alpha_2<\dim_{\mathcal{H}}(\tilde{G_1})$ and a compact set $\tilde{G_2}\subset \tilde{G_1}$ satisfying
\begin{itemize}
\item[1. ] $\mathcal{H}^{\alpha_2}(\tilde{G_2})>0$;
\item[2. ] $\dim_{\mathcal{H}}(T^2_x(E_1))\geq \dim_{\mathcal{H}}(T^2_x(E_1^{(2)}))\geq \frac{4\cdot 2}{3}(\dim_{\mathcal{H}}(E)-\epsilon_0)-\frac{2\cdot 2}{3},\quad \forall x\in \tilde{G_2}$. 
\end{itemize}Continuing the process, there exists a sequence $\{\alpha_i\}$ and a sequence of nested compact sets $\{\tilde{G_i}\}$ satisfying 
\begin{itemize}
\item[1. ] $\dim_{\mathcal{H}}(\tilde{G}_{i-1})-\frac{\epsilon}{2}\cdot 10^{-i+1}<\alpha_i<\dim_{\mathcal{H}}(\tilde{G}_{i-1})$;
\item[2. ] $\tilde{G_i}\subset \tilde{G}_{i-1}$,\, $\mathcal{H}^{\alpha_i}(\tilde{G_i})>0$;
\item[3. ] $\dim_{\mathcal{H}}(T^i_x(E_1))\geq \dim_{\mathcal{H}}(T^i_x(E_1^{(i)}))\geq \frac{4 i}{3}(\dim_{\mathcal{H}}(E)-\epsilon_0)-\frac{2 i}{3},\quad \forall x\in \tilde{G_i}$. 
\end{itemize}
This implies, in particular, that there exists a point $x\in E$ so that $\dim_{\mathcal{H}}(  T_x^{k}(E) ) \geq \frac{4 k}{3}(\dim_{\mathcal{H}}(E)-\epsilon_0)-\frac{2 k}{3}$, for each integer $k\geq 1$.

\bigskip
\subsubsection{Proof of Proposition \ref{structure2}}
For the sake of simplicity, we only prove the chain case, as the general tree case can be treated in almost the same way with a slight modification similarly to the proof of Proposition \ref{structure1} in Section \ref{sec: Lebesgue}, which is left to the interested reader.

We prove by induction. According to the assumption, the base case $k=1$ is automatically true. Now, assume that the desired result holds for $k-1$. Let $E_1, E_2$ be the given sets in $\mathbb{R}^d$ with Hausdorff dimension $m$. Let $E_{1,1}, E_{1,2}\subset E_1$ be two subsets of $E_1$, so that $\dim_{\mathcal{H}}(E_{1,i})=m$, $i=1,2$, and the distance between them is positive. By properly shrinking $E_1, E_2$ without altering their dimension, one can assume without loss of generality that the distance from $E_{1,i}$ to $E_2$ is also positive, $i=1,2$.

By the induction hypothesis, there exists $E_{1,2}'\subset E_{1,2}$ so that $\dim_{\mathcal{H}}(E_{1,2}\setminus E_{1,2}')<m$ (in particular, $\dim_{\mathcal{H}}(E_{1,2}')=m$) and
\[
\dim_{\mathcal{H}}(S_y^{k-1}(E_{1,1}))\geq (k-1)\gamma,\quad \forall y\in E_{1,2}'.
\]

Now, applying the assumption (i.e. the base case $k=1$) to the sets $E_{1,2}'$ and $E_2$, one can find $E_2'\subset E_2$ satisfying $\dim_{\mathcal{H}}(E_2\setminus E_2')<m$, and such that
\[
\dim_{\mathcal{H}}(\Delta_x(E_{1,2}'))\geq \gamma,\quad \forall x\in E_2'.
\]

Fix $x\in E_2'$, and let $B=S_x^k(E_1)$ and $A=\Delta_x(E_{1,2}')$. For all $t_0\in \Delta_x(E_{1,2}')$, let $B_{t_0}$ denote the slice of $B$ at $t_0$ in the first variable. Observe that 
\[
S_x^k(E_1)\supset \left\{(|x-x_1|,\ldots,|x_k-x_{k+1}|):\, x_1\in E_{1,2}',\, x_2,\ldots,x_{k+1}\in E_{1,1} \,\text{distinct}\right\}.
\]Hence, one has for all $t_0\in A$ that
\[
\dim_{\mathcal{H}}(B_{t_0})\geq \dim_{\mathcal{H}}(S^{k-1}_{y_{t_0}}(E_{1,1}))\geq (k-1)\gamma,
\]for some $y_{t_0}\in E_{1,2}'$ satisfying $|y_{t_0}-x|=t_0$. Then, according to the Fubini-like theorem Corollary \ref{cor: Fubini}, this implies
\[
\dim_{\mathcal{H}}(B)\geq (k-1)\gamma+\dim_{\mathcal{H}}(A)\geq(k-1)\gamma+\gamma=k\gamma.
\]The proof is thus complete.

\begin{remark}
It is easy to see that, by following the same strategy as above, one can prove other versions of the structural theorem concerning more general point configurations, such as the \emph{kite} in Proposition \ref{glue}. Indeed, fix $x_0\in E_1$ such that the set $A: = A_{x_0}  = \{ ( |x_0-y_1|, |y_1-y_2|, |y_2 - x_0|): y_1, y_2 \in E_1\}\subset \mathbb{R}_+^3$ satisfies $\dim_{\mathcal{H}}(A)\geq \gamma_1$. Let $B$ be the set of kites pinned at $x$. Observe that, for each $(t_1, t_2, t_3) \in A$, $B_{(t_1, t_2, t_3)} \supset S^2_{y_1}(E_2)$, for some $y_1, y_2\in E_1$ such that $( |x_0-y_1|, |y_1-y_2|, |y_2 - x_0|)=(t_1,t_2,t_3)$, where $B_{(t_1, t_2, t_3)}$ denotes the slice of the set $B$ at $(t_1, t_2, t_3)$. Hence, one has $\dim_{\mathcal{H}}(B_{(t_1, t_2, t_3)})\geq \dim_{\mathcal{H}}(S^2_{y_1}(E_2))\geq \gamma_2$. One can then apply Corollary \ref{cor: Fubini} similarly as above to prove Proposition \ref{glue}. We omit the details. 
\end{remark}

\vskip.125in
\subsection{Proof of Theorem \ref{main2.5}}
We now turn to the proof of Theorem \ref{main2.5}, which also relies on the Fubini-like theorem in Appendix \ref{sec: Fubini}.

\begin{proof}[Proof of Theorem \ref{main2.5}]
This proof has a similar flavor as the above, but will make use of (\ref{eqn: Liu1}). It is direct to see that when $k=1$, the two bounds coincide with (\ref{eqn: Liu1}).

Assume that the desired result holds for $k-1$. More precisely, assume that for any set $E$ with $\dim_{\mathcal{H}}(E)>1$, and $\tau\in (0,\tau^{k-1}_0(\dim_{\mathcal{H}}(E)))$ with
\[
\tau^k_0(\dim_{\mathcal{H}}(E))=\begin{cases} \frac{4(k-1)}{3}\dim_{\mathcal{H}}(E)+\frac{5-2k}{3},&  \text{if } 1<\dim_{\mathcal{H}}(E)\leq \frac{5}{4}, \\ k,& \text{if }\frac{5}{4}< \dim_{\mathcal{H}}(E)\leq 2,\end{cases}
\]
there holds
\begin{equation}\label{eqn: dim2}
\dim_{\mathcal{H}}(\{x\in \mathbb{R}^2:\, \dim_{\mathcal{H}}(S^{k-1}_x(E))<\tau\})\leq \max(\Gamma^{k-1}(\tau,\dim_{\mathcal{H}}(E)),2-\dim_{\mathcal{H}}(E)),
\end{equation}where
\[
\begin{split}
\Gamma^{k}(\tau,\dim_{\mathcal{H}}(E)):= &\begin{cases} 2k+3\tau+(1-4k)\dim_{\mathcal{H}}(E), & \text{if }1<\dim_{\mathcal{H}}(E)\leq \frac{5}{4},\\
5-3k+3\tau-3\dim_{\mathcal{H}}(E), & \text{if }\frac{5}{4}<\dim_{\mathcal{H}}(E)\leq 2.\end{cases}
\end{split}
\]

Our goal is to show that for all $\tau\in (0, \tau_0^k)$
\[
\dim_{\mathcal{H}}(\{x\in \mathbb{R}^2:\, \dim_{\mathcal{H}}(S_x^k(E))<\tau\})\leq \max(\Gamma^k(\tau,\dim_{\mathcal{H}}(E)),2-\dim_{\mathcal{H}}(E)).
\]
In particular, it suffices to show that for all $\epsilon>0$,
\[
\dim_{\mathcal{H}}(\{x\in \mathbb{R}^2:\, \dim_{\mathcal{H}}(S_x^k(E))<\tau\})\leq \max(\Gamma^k(\tau,\dim_{\mathcal{H}}(E)),2-\dim_{\mathcal{H}}(E))+\epsilon.
\]

First, consider the range $\dim_{\mathcal{H}}(E)\leq\frac{5}{4}$ and fix $\tau\in (0, \frac{4(k-1)}{3}\dim_{\mathcal{H}}(E)+\frac{5-2k}{3})$. 

Fix $\epsilon>0$, it suffices to show that given any set $F\subset \mathbb{R}^2$ satisfying
\begin{equation}\label{eqn: dim5}
\dim_{\mathcal{H}}(F)>\max(\Gamma^k(\tau,\dim_{\mathcal{H}}(E)),2-\dim_{\mathcal{H}}(E))+\epsilon,
\end{equation}there exists a point $x\in F$ such that $\dim_{\mathcal{H}}(S_x^k(E))\geq \tau$.

Without loss of generality, assume the distance between $E,F$ is positive. Note that the induction hypothesis implies that there exists a subset $E'\subset E$ with the same dimension as $E$ so that 
\begin{equation}\label{eqn: dim4}
\dim_{\mathcal{H}}(S^{k-1}_y(E))\geq \min\left(\frac{4(k-1)}{3}\dim_{\mathcal{H}}(E)-\frac{2(k-1)}{3}-\frac{\epsilon}{3}, k-1\right),\quad \forall y\in E'.
\end{equation}Indeed, this can be derived from (\ref{eqn: dim2}), with ``$\tau$'' in the display being chosen such that $\max(\Gamma^{k-1}(\tau,\dim_{\mathcal{H}}(E)),2-\dim_{\mathcal{H}}(E))+\epsilon=\dim_{\mathcal{H}}(E)$ (not necessarily the same $\tau$ that we fixed above).

Define
\[
\tau_1:=\begin{cases} \frac{2}{3}\dim_{\mathcal{H}}(E)+\frac{\epsilon}{3},& \text{if }\tau\leq \frac{4k-2}{3}\dim_{\mathcal{H}}(E)-\frac{2k-2}{3},\\ \frac{2k-2}{3}+\tau-\frac{4k-4}{3}\dim_{\mathcal{H}}(E)+\frac{\epsilon}{3}, & \text{if } \tau> \frac{4k-2}{3}\dim_{\mathcal{H}}(E)-\frac{2k-2}{3}.\end{cases}
\]Note that the point here is to make sure that 
\[
\max(\Gamma^k(\tau, \dim_{\mathcal{H}}(E)), 2-\dim_{\mathcal{H}}(E))+\epsilon=\max(\Gamma^{1}(\tau_1,\dim_{\mathcal{H}}(E)),2-\dim_{\mathcal{H}}(E)).
\]Because of the bound of $\tau$, one also has from the definition that $\tau_1\in (0,1)$ (by letting $\epsilon$ sufficiently small depending on $\tau$). Note that this is the only place in the proof where the upper bound $\tau<\tau_0^k$ comes into play.

Recalling the definition of $F$, we have
\[
\dim_{\mathcal{H}}(F)>\max(\Gamma^1(\tau_1,\dim_{\mathcal{H}}(E')),2-\dim_{\mathcal{H}}(E')).
\]By assumption of the theorem, there thus exists a point $x\in F$ such that
\[
\dim_{\mathcal{H}}(\Delta_x(E'))\geq \tau_1.
\]

From the construction of $E'$, for all $y\in E'$, (\ref{eqn: dim4}) holds true. Letting $B=S^k_x(E)$ and $A=\Delta_x(E')$, one sees that for all $t_0\in A$, the slice of the set $B$ at $t_0$ in the first variable satisfies
\[
\begin{split}
\dim_{\mathcal{H}}(B_{t_0})\geq &\dim_{\mathcal{H}}(S^{k-1}_{y_{t_0}}(E))\\
\geq &\min\left(\frac{4(k-1)}{3}\dim_{\mathcal{H}}(E)-\frac{2(k-1)}{3}-\frac{\epsilon}{3}, k-1\right)\\
=&\frac{4(k-1)}{3}\dim_{\mathcal{H}}(E)-\frac{2(k-1)}{3}-\frac{\epsilon}{3},
\end{split}
\]for some $y_{t_0}\in E'$ satisfying $|x-y_{t_0}|=t_0$. Therefore, according to Corollary \ref{cor: Fubini}, one obtains
\[
\dim_{\mathcal{H}}(B)\geq \frac{4(k-1)}{3}\dim_{\mathcal{H}}(E)-\frac{2(k-1)}{3}-\frac{\epsilon}{3}+\tau_1\geq \tau.
\]The proof of the first case is complete.

Second, assume $\dim_{\mathcal{H}}(E)>\frac{5}{4}$ and $\tau\in (0,k)$. Again, the goal is to show that there exists $x\in F$ so that $\dim_{\mathcal{H}}(S_x^k(E))\geq \tau$. Same as before, the induction hypothesis implies for each $\epsilon>0$ the existence of a set $E'\subset E$ which satisfies (\ref{eqn: dim4}). Since $\dim_{\mathcal{H}}(E)>\frac{5}{4}$, by taking $\epsilon$ sufficiently small, one has in fact $\dim_{\mathcal{H}}(S_y^{k-1}(E))\geq k-1$, $\forall y\in E'$. Now, let
\[
\tau_1:=\begin{cases} \min(\frac{2}{3}\dim_{\mathcal{H}}(E),1),& \text{if }\tau\leq k-1,\\ 1-k+\tau, & \text{if } \tau> k-1.\end{cases}
\]It is straightforward to check that $\tau_1\in (0,1)$ and 
\[
\max(\Gamma^k(\tau, \dim_{\mathcal{H}}(E)), 2-\dim_{\mathcal{H}}(E))=\max(\Gamma^{1}(\tau_1,\dim_{\mathcal{H}}(E)),2-\dim_{\mathcal{H}}(E)).
\]

Therefore, there is $x\in F$ so that $\dim_{\mathcal{H}}(\Delta_x(E'))\geq \tau_1$. Applying Corollary \ref{cor: Fubini} and arguing as above, one has
\[
\dim_{\mathcal{H}}(S_x^k(E))\geq k-1+\tau_1\geq \tau,
\]which completes the proof.
\end{proof}

%%%%%%%%% PROOF OF CHAINS w BOCHEN (Hausdorff measure)

\bigskip
\section{Chains/trees with prescribed gaps: Proof of Theorems \ref{main3}, \ref{main4}}\label{sec: unit}
In this section we study the dimension of set of chains/trees that have prescribed gaps. Many of the upper bound estimates below extend the work of Eswarathasan--Iosevich--Taylor \cite{EIT} and Oberlin--Oberlin \cite{OO}, where the case $k=1$, i.e. the unit distance set, was considered (for slightly more restrictive classes of sets $E$). On the other hand, we will also see below that sometimes the chain/tree cases display very different properties compared to the distance case. For instance, in $\mathbb{R}^2$, when $0<\alpha\leq 1$, the best known estimate for the unit distance set is 
\[
\frac{3\alpha}{2}\leq g_2(VS_1^1, \alpha)\leq \min\left(\frac{5\alpha}{3}, \frac{\alpha(2+\alpha)}{1+\alpha}\right),
\]according to \cite{OO}. However, the $2$-chain set displays distinct features and we can completely determine the value of $g_2(VS^2_{\vec{t}},\alpha)$ without first estimating the unit distance set.

\subsection{Proof of Theorem \ref{main3}}

\subsubsection{Upper bound}\label{blahsingleEuc}
Let $k\geq 2$ and $E\subset \mathbb{R}^d$ be a compact $\{\delta_i\}$-discrete $\alpha$-regular set that is contained in the unit ball and has Hausdorff dimension $\alpha$. Given any $\delta_i$, we will show that $\forall \epsilon>0$, 
\begin{equation}\label{eqn: unit distance upp main}
\begin{split}
|D^{\delta_i}_k|:=&|\{(x_1,\cdots,x_{k+1})\in E_{\delta_i}^{k+1}:\, t_j-2\delta_i\leq |x_j-x_{j+1}|\leq t_j+2\delta_i,\, j=1,\ldots,k  \}|\\
\lesssim &\delta_i^{(k+1)d-u(k,d,\alpha)-\epsilon},
\end{split}
\end{equation}where
\[
u(k,d,\alpha):=\begin{cases} (k+1)\alpha-k, & \frac{d+1}{2}\leq \alpha\leq d,\\ \frac{kd}{2}+\alpha-\frac{k}{2},& \alpha\leq \frac{d+1}{2}.\end{cases}
\]
Since $D^{\delta_i}_k$ contains the $\delta_i$-neighborhood of $VS^k_{\vec{t}}(E)$, estimate (\ref{eqn: unit distance upp main}) would imply that $\dim_{\mathcal{H}}(VS^k_{\vec{t}}(E))\leq\underline{\rm dim}_{\mathcal{M}}(VS^k_{\vec{t}}(E))\leq u(k,d,\alpha)$, hence the desired upper bound in Theorem \ref{main3} follows.  
Indeed, letting $\gamma$ denote the lower Minkowski dimension of $VS^k_{\vec{t}}(E)$ and taking $\delta_i$ sufficiently small, it follows that $|  D^{\delta_i}_k|  \sim   \delta_i^{(k+1)d-\gamma}$, where $|\cdot |$ denotes the Lebesgue measure.  Thus, our goal is reduced to attaining an upper bound, in terms of a power of $\delta_i$, on $|  D^{\delta_i}_k|$. To simplify the notation, we will write $\delta=\delta_i$ in the following.  

Without loss of generality, assume that $E=-E$. Write $E_\delta=E+B(0,\delta)$ and define 
\[
A_{t,\delta}:=\{x\in \mathbb{R}^d:\, t-2\delta\leq |x|\leq t+2\delta\}.
\]
We re-write the set $D^\delta_k$ as follows:
\[
\begin{split}
&|D^\delta_k|\\
=&\int_{E_\delta}\cdots\int_{E_\delta} \prod_{i=1}^{k} \chi_{A_{t_i,\delta}}(x_i-x_{i+1})\,dx_1\cdots dx_{k+1}\\
=& \int_{E_\delta}\cdots\int_{E_\delta} \left(\prod_{i=1}^{k-1} \chi_{A_{t_i,\delta}}(x_i-x_{i+1})\right) 
\chi_{E_{\delta}}
\ast 
\chi_{A_{t_k,\delta}}
(x_k)\,dx_1\cdots dx_{k}\\
=& \int_{E_\delta}\cdots\int_{E_\delta} \left(\prod_{i=1}^{k-2} \chi_{A_{t_i,\delta}}(x_i-x_{i+1})\right) 
\left(
f_1
\chi_{E_\delta} \ast \chi_{A_{t_{k-1},\delta}}(x_{k-1})\right)\,dx_1\cdots dx_{k-1}\\
=& \cdots= \langle f_k, \chi_{E_\delta}\rangle,
\end{split}
\]where we have defined $f_1=\chi_{E_\delta}\ast \chi_{A_{t_k,\delta}}$, and $f_{n+1}=(f_n\chi_{E_\delta})\ast \chi_{A_{t_{k-n},\delta}}$, $\forall 2\leq n\leq k-1$.

The main estimate we will prove is the following $L^2$ bound:
\begin{lemma}\label{lemma: unit distance upp L2}
Let $f\in L^2(E_{\delta})$, and $t\sim 1$, $\delta>0$ as before. Then for all $\epsilon>0$,
\begin{equation}\label{eqn: unit distance upp L2}
\left(\int_{E_\delta} |(f\chi_{E_\delta})\ast \chi_{A_{t,\delta}}(x)|^2 \,dx \right)^{1/2}\leq C_\epsilon \delta^{\beta(d,\alpha)-\epsilon}\left(\int_{ E_\delta} |f(x)|^2   \,dx \right)^{1/2}, 
\end{equation}
where
\[
\beta(d,\alpha):=\begin{cases} d-\alpha+1,& \frac{d+1}{2}\leq \alpha\leq d,\\ \frac{d+1}{2},& \alpha\leq \frac{d+1}{2}.\end{cases}
\]
\end{lemma}
%\vskip.125in

Note that a special case of this estimate when $f=1$ was obtained in \cite{OO}. We first show how Lemma \ref{lemma: unit distance upp L2}, implies the desired (\ref{eqn: unit distance upp main}). Applying the Cauchy-Schwarz inequality, one has
\[
\begin{split}
\langle f_k, \chi_{E_\delta}\rangle=&\langle (f_{k-1}\chi_{E_\delta})\ast\chi_{A_{t_1},\delta}, \chi_{E_\delta}\rangle\\
\leq & \left(\int_{\mathbb{R}^d} |(f_{k-1}\chi_{E_\delta})\ast \chi_{A_{t_1,\delta}}(x)|^2 \chi_{E_\delta}(x)\,dx \right)^{1/2} |E_\delta|^{1/2}\\
\leq &C_\epsilon \delta^{\beta(d,\alpha)-\epsilon} \left(\int_{\mathbb{R}^d} |f_{k-1}(x)|^2\chi_{E_\delta}(x)\,dx \right)^{1/2}|E_\delta|^{1/2},
\end{split}
\]where we have applied (\ref{eqn: unit distance upp L2}) in the last step. 

By applying (\ref{eqn: unit distance upp L2}) iteratively, one ultimately obtains
\[
\langle f_k, \chi_{E_\delta}\rangle\lesssim_\epsilon \delta^{k(\beta(d,\alpha)-\epsilon)}|E_\delta|\lesssim \delta^{k(\beta(d,\alpha)-\epsilon)+d-\alpha-\epsilon},
\]where we have recalled the definition of the $\{\delta_i\}$-discrete $\alpha$-set. Estimate (\ref{eqn: unit distance upp main}) thus follows immediately.

\subsubsection{Proof of Lemma \ref{lemma: unit distance upp L2}  }
We now turn to proving Lemma \ref{lemma: unit distance upp L2}. 

Let $g$ be a testing function satisfying $\int |g|^2\chi_{E_\delta}=1$, then it suffices to show that
\begin{equation}\label{eqn: unit dist 1}
\left| \int_{E_\delta} \big[(f\chi_{E_\delta})\ast \chi_{A_{t,\delta}} \big](x)g(x)    \,dx \right|\lesssim_\epsilon \delta^{\beta(d,\alpha)-\epsilon} \left(\int_{E_\delta} |f|^2\right)^{1/2}.
\end{equation}

Without loss of generality, assume both $f$ and $g$ are nonnegative. Let $\rho$ be a symmetric Schwartz function satisfying
\begin{equation}\label{rho}
\chi_{B(0,C)}\leq \rho(x)\leq \sum_{j=1}^\infty 2^{-jd}\chi_{B(0,2^j)},
\quad \chi_{B(0,C')}\leq  |\hat{\rho}(\xi)|  \leq \chi_{B(0,2C')},
\end{equation}and denote $\rho_r(x)=r^{-d}\rho(\frac{x}{r})$.

Then, 
\[
\begin{split}
&\left| \int \big[(f\chi_{E_\delta})\ast \chi_{A_{t,\delta}} \big](x)g(x)\chi_{E_\delta}(x)\,dx \right|\\
=&\left| \langle (f\chi_{E_\delta})\ast (g\chi_{E_\delta}), \chi_{A_{t,\delta}}\rangle \right|\\
\lesssim &\delta \left|\langle (f\chi_{E_\delta})\ast (g\chi_{E_\delta}), \rho_\delta\ast \rho_\delta\ast \sigma_t \rangle  \right|,
\end{split}
\]where $\sigma_t$ denotes the surface measure on the sphere in $\mathbb{R}^d$ of radius $t$ (not normalized). Recalling that $t\sim 1$ and applying Plancherel, one has that the above is bounded by
\[
\begin{split}
\lesssim &\delta \int_{B(0,\frac{2C'}{\delta})} \left|\left[(f\chi_{E_\delta})\ast\rho_\delta\right]^\wedge(\xi)\right| \left|\left[(g\chi_{E_\delta})\ast\rho_\delta\right]^\wedge(\xi)\right| \,\frac{d\xi}{(1+|\xi|)^{\frac{d-1}{2}}}\\
\leq &\delta \left(\int_{B(0,\frac{2C'}{\delta})} \left| \left[(f\chi_{E_\delta})\ast\rho_\delta\right]^\wedge(\xi) \right|^2\,\frac{d\xi}{(1+|\xi|)^{\frac{d-1}{2}}}  \right)^{1/2}\cdot\\
&\qquad\qquad \left(\int_{B(0,\frac{2C'}{\delta})} \left| \left[(g\chi_{E_\delta})\ast\rho_\delta\right]^\wedge(\xi) \right|^2\,\frac{d\xi}{(1+|\xi|)^{\frac{d-1}{2}}}  \right)^{1/2}.
\end{split}
\]Here, the first inequality follows from the estimate $|\widehat{\sigma_t}(\xi)|\lesssim (1 + |\xi|)^{-\frac{d-1}{2}}$ (see for instance \cite[Corollary 6.7]{Wolffbook}).

Hence, estimate (\ref{eqn: unit dist 1}) results from the following estimate.

\begin{lemma}\label{lemma: unit dist 2} For $f\in L^2(E_{\delta})$, 
\begin{equation}\label{eqn: unit dist 2}
\left(\int_{B(0,\frac{2C'}{\delta})} \left| \left[(f\chi_{E_\delta})\ast\rho_\delta\right]^\wedge(\xi) \right|^2\,
\frac{d\xi}{     (1+|\xi|)^{\frac{d-1}{2}}}  \right)^{1/2}
\lesssim_\epsilon \delta^{\frac{\beta(d,\alpha)-1}{2}-\epsilon}\left(\int_{E_\delta} |f|^2\right)^{1/2},
\end{equation}
where
\[
\beta(d,\alpha):=\begin{cases} d-\alpha+1,& \frac{d+1}{2}\leq \alpha\leq d,\\ \frac{d+1}{2},& \alpha\leq \frac{d+1}{2}.\end{cases}
\]
\end{lemma}

This lemma follows from a slightly simpler statement: 
\begin{lemma}\label{lemma: unit dist alpha} For $f\in L^2(E_{\delta})$, 
\begin{equation}\label{eqn: unit dist alpha}
\left(\int_{B(0,\frac{2C'}{\delta})} \left| \left[(f\chi_{E_\delta})\ast\rho_\delta\right]^\wedge(\xi) \right|^2\,
\frac{d\xi}{     ((1+|\xi|) ) ^{d-\alpha}}  \right)^{1/2}
\lesssim_\epsilon \delta^{\frac{d-\alpha}{2}-\epsilon}\left(\int_{E_\delta} |f|^2\right)^{1/2}.
\end{equation}
\end{lemma}

To see that Lemma \ref{lemma: unit dist 2} follows from Lemma \ref{lemma: unit dist alpha}, we notice that, when $\alpha\geq\frac{d+1}{2}$, $d-\alpha\le \frac{d-1}{2}$, and the left-hand side of \eqref{eqn: unit dist 2} is dominated (up to a constant) by the left-hand side of \eqref{eqn: unit dist alpha}. 
 In the case that $\alpha <\frac{d+1}{2}$, for $|\xi| \le \frac{2C'}{\delta}$, we observe that
$$\frac{1}{|\xi|^{(d-1)/2}} \lesssim \frac{\delta^{\alpha - \frac{d+1}{2}}}{|\xi|^{d-\alpha}  }.$$
  Now,  the left-hand-side of \eqref{eqn: unit dist 2} is bounded above by $\delta^{\frac{1}{2}(\alpha - \frac{d+1}{2})}$ times the expression on the left-hand-side of \eqref{eqn: unit dist alpha}, and the Lemma follows.

\subsubsection{ Proof of Lemma \ref{lemma: unit dist alpha} }
We will focus on the demonstration of Lemma \ref{lemma: unit dist alpha} in the rest of the subsection.  For $r\geq \delta$, we first use interpolation to show that 
\begin{equation}\label{ball L2}
\|  f\chi_{E_{\delta}}  * \chi_{B(0,r)}  \|_{L^2}  \lesssim_\epsilon r^{\frac{d+\alpha+\epsilon}{2}}  \delta^{\frac{d-\alpha}{2}-\epsilon} \|  f \|_{L^2( E_{\delta}  ) }. 
\end{equation}
\vskip.125in

Observe that
\begin{align*}
\|  f\chi_{E_{\delta}}  * \chi_{B(0,r)}  \|_{L^{\infty}}  
=&\sup_{x \in E_{\delta}}   \left|  \int  f(y)  \chi_{E_{\delta}\cap B(x,r)}(y)  dy    \right|   \\
\le&     \|  f  \|_{L^{\infty}(E_{\delta}  )}    \left| E_{\delta}\cap B(x,r)  \right|  \\
\le&  \|  f  \|_{L^{\infty}(E_{\delta}  )}    \,\, r^{\alpha +\epsilon}  \delta^{d-\alpha-2\epsilon},
\end{align*}
where we have used the assumption that $E$ is $\{\delta_i\}$-discrete $\alpha$-regular in the last line.

Also, observe that 
\begin{align*}
\|  f\chi_{E_{\delta}}  * \chi_{B(0,r)}  \|_{L^1  (E_{\delta}  )} 
=&    \int     \left|  \int  f(y)  \chi_{E_{\delta}}(y)   \chi_{B(x,r)}(x-y)  dy    \right|   dx  \\
\le&     \|  f  \|_{L^1  (E_{\delta}  )}    \left| B(0,r)  \right|  \\
\le&  \|  f  \|_{L^1(E_{\delta}  )}    \,\, r^{d }.
\end{align*}

Interpolating these two estimates yields \eqref{ball L2}  (see, Riesz-Thorin interpolation theorem in \cite{St93}). 
\\

Next, by the definition of $\rho$, for $r\geq \delta$, we see that whenever $\alpha<d$,
\begin{equation}\label{eqn: use for triangle}
\begin{split}
\|  f \chi_{E_{\delta}}  * \rho_{r}  \|_{L^2}  
\lesssim& \sum_{j=1}^{\infty} 
2^{-j\frac{(d-\alpha)}{2}  } \, \|  f \|_{L^2( E_{\delta}  ) } \, r^{\frac{\alpha - d+\epsilon}{2}}  \delta^{\frac{d-\alpha}{2}-\epsilon} \\
\lesssim& \,
 \|  f \|_{L^2( E_{\delta}  ) } \, \left(  \frac{\delta}{r}\right)^{\frac{d-\alpha-\epsilon}{2}}\delta^{-\frac{\epsilon}{2}}.
 \end{split}
\end{equation}(The case $\alpha=d$ is implied by the second trivial upper bound given in Remark \ref{rmk: trivial examples}.)

It is a direct consequence of \eqref{rho} that, if $|\xi| \le \frac{C'}{r},$
then $| \widehat{  \rho_{\delta} } (\xi) | \le | \widehat{  \rho_r } (\xi) |$
(Indeed, $ | \widehat{  \rho_{r} } (\xi) |  =| \widehat{  \rho} (r \xi)|  \geq \chi_{B(0,C')}(r \xi) $ = 1, and when $|\xi| \le \frac{2C'}{\delta},$
$| \widehat{  \rho_{\delta} } (\xi) |
=| \widehat{  \rho} (\delta \xi)|
\leq \chi_{B(0,2C')}(\delta \xi)
$).\\

It follows that
\begin{equation}\label{est 2}
\begin{split}
\int_{  \frac{C'}{2r}  \le |\xi|  \le \frac{C'}{r}     } \left|  
 (f \chi_{E_{\delta}}  * \rho_{\delta}  )^{\wedge}
 (\xi)  \right|^2  d\xi
\le &
\int_{  \frac{C'}{2r}  \le |\xi|  \le \frac{C'}{r}     } \left| 
( f \chi_{E_{\delta}}  * \rho_{r}) )^{\wedge}
(\xi)  
 \right|^2  d\xi \\
\lesssim &
 \|  f \|^2_{L^2( E_{\delta}  ) } \, \left(  \frac{\delta}{r}\right)^{d-\alpha-\epsilon}\delta^{-\epsilon}.
 \end{split}
\end{equation}

With the estimates above in tow, we turn to estimating the left-hand-side of \eqref{eqn: unit dist alpha}:
\[ \left(\int_{B(0,\frac{2C'}{\delta})}
 \left| \left[(f \chi_{E_\delta})\ast\rho_\delta\right]^\wedge(\xi) \right|^2\, \frac{d\xi}{    (1+|\xi|)^{d-\alpha} } \right)^{1/2},
\]where the integration domain comes from the assumption that $|\widehat{\rho}(\delta\xi)| \le \chi_{B(0, 2C')}( \delta \xi)$.  
We consider the integral over the set $\{|\xi|  > C'\}   $ and $\{|\xi| \leq C'\}$ separately.  We have
\begin{align*}
& \sum_{\{ 1 \le 2^j \le \frac{1}{\delta}\}  }
 \int_{ \{  2^{j}  <\frac{   |\xi|}{C'} \le 2^{j+1} \}}  
 \left| \left[(f\chi_{E_\delta})\ast\rho_\delta\right]^\wedge(\xi) \right|^2\, \frac{d\xi}{    (1+|\xi|)^{d-\alpha} } \\
 \lesssim   &  \sum_{\{ 1 \le 2^j \le \frac{1}{\delta}\}  }  \|  f \|^2_{L^2( E_{\delta}  ) } \,  \left(  \delta 2^j \right)^{d-\alpha-\epsilon}\delta^{-\epsilon}  
2^{-j(d-\alpha)}  \\
\lesssim  & \,\,  \|  f \|^2_{L^2( E_{\delta}  ) }   \delta^{d-\alpha-\epsilon}    \log(1/\delta),
\end{align*}
 where we used \eqref{est 2} with $r=2^{-(j+1)}.$

While, 
\begin{align*}
 \int_{\{  |\xi|  \le C'\}  }   
 \left| \left[(f\chi_{E_\delta})\ast\rho_\delta\right]^\wedge(\xi) \right|^2\, d\xi 
 \lesssim &  \|\left((f\chi_{E_\delta})\ast\rho_\delta\right)^\wedge  \|_{L^{\infty}}^2
  \lesssim   \|\left(f\chi_{E_\delta}\right)^\wedge  \|_{L^{\infty}}^2\\
    \lesssim &  \|f\chi_{E_\delta} \, \|_{L^{1}}^2
\lesssim \| f\|^2_{L^2( E_{\delta} )}\cdot |E_{\delta}|  
\lesssim \| f\|^2_{L^2( E_{\delta} )}\delta^{d-\alpha-\epsilon},
\end{align*}
where we used Cauchy-Schwarz in the second to last step and the definition of $\{\delta_i\}$-discrete $\alpha$-regular in the last.  This concludes the proof of Lemma \ref{lemma: unit dist alpha}.
\\

\subsubsection{Lower bound}\label{section lower bound}
Let $k\geq 2$ and $d\geq 2$. 
Let $\vec{t}=(t_1,\cdots, t_k)\in\mathbb{R}^k_+$ denote arbitrary prescribed gaps.
We demonstrate the existence of a $\{\delta_i\}$-discrete $\alpha$-regular set $E\subset \mathbb{R}^d$ of Hausdorff dimension $\alpha$, so that the Hausdorff dimension of $VS^k_{\vec{t}}(E)$ is at least $(k+1)\alpha-k$. 
Note that, when $\alpha\geq \frac{d+1}{2}$, the upper bound and lower bound for $g_d(VS^k_{\vec{t}},\alpha)$ match. 
For this reason, we only state the result for this range, although the following example works for any value $\alpha\in (1,d)$.  

For the sake of simplicity, we only consider the case $k=2$ and assume $t_1\leq t_2$. 
The following example is adapted from \cite{OO}, where the $k=1$ case was studied. 
Let $C\subset B(0,\frac{t_1}{2})\subset \mathbb{R}^{d-1}$ be an AD regular set with Hausdorff dimension $\gamma\geq 0$. Define $E=C\times [0,4t_2]\subset \mathbb{R}^d$ and $\alpha:=\gamma+1$. Then it is easy to see that $\dim_{\mathcal{H}}(E)=\alpha$ and $E$ is $\{\delta_i\}$-discrete $\alpha$-regular. One also has
\[
\begin{split}
VS^2_{(t_1,t_2)}(E)=&\Big\{(c_1,s_1; c_2,s_2;c_3,s_3):\, c_i\in C,\, s_i\in [0,4t_2],\,\\
&\quad |s_1-s_2|=\sqrt{t_1^2-|c_1-c_2|^2},\, |s_2-s_3|=\sqrt{t_2^2-|c_2-c_3|^2}\Big\}.
\end{split}
\]With any $c_1,c_2,c_3,s_1$ fixed such that $c_1, c_2, c_3$ are distinct, at least a $s_2\in [0,4t_2]$ will be determined, which will in turn determine at least a $s_3\in [0,4t_2]$. Therefore,
\[
\dim_{\mathcal{H}}(VS^2_{(t_1,t_2)}(E))\geq \dim_{\mathcal{H}}(C\times C\times C)+1\geq 3\gamma+1=3\alpha-2.
\]
In general, it is an easy deduction to extend this example to longer chains to show that
\[
\dim_{\mathcal{H}}(VS^k_{\vec{t}}(E))\geq \dim_{\mathcal{H}}(C\times \cdots \times C)+1\geq (k+1)\alpha-k.
\]
The proof of Theorem \ref{main3} is complete.

Note that both the above upper and lower estimates extend to general trees. Indeed, following the iterative scheme introduced in \cite{IT} for trees, one can apply the $L^2$ estimate in Lemma \ref{lemma: unit distance upp L2} to obtain the upper bound. A similar construction as in the example above will produce the matching lower bound, where one still fixes distinct $c_1,\cdots, c_{k+1}$ first together with the $s_i$ corresponding to the root of the tree. We omit the details.

\subsection{Proof of Theorem \ref{main4}}

\subsubsection{Upper bound: }\label{sec: trivial upper bound}

Let $d\geq 4$, $\alpha\leq \lfloor \frac{d}{2} \rfloor-1$, and fix $\vec{t}\in \mathbb{R}^k_+$ with $t_i\sim 1$, $\forall i$. It is easy to see from the assumption that $g_d(VS_{\vec{t}}^k,\alpha)\leq (k+1)\alpha$. 

Indeed, for any set $E\subset \mathbb{R}^d$ that is $\{\delta_i\}$-discrete $\alpha$-regular and is contained in the unit ball, one has $|E_{\delta_i}|\lesssim_\epsilon \delta_i^{d-\alpha-\epsilon}$. Hence, the $\delta_i$-neighborhood of $VS^k_{\vec{t}}(E)$, being contained in $E_{\delta_i}\times\cdots \times E_{\delta_i}$, has Lebesgue measure bounded by $\delta_i^{(k+1)(d-\alpha-\epsilon)}$. This in particular shows that 
\[
\dim_{\mathcal{H}}(VS^k_{\vec{t}}(E))\leq \underline{\rm dim}_{\mathcal{M}}(VS^k_{\vec{t}}(E))\leq (k+1)\alpha.
\]

\subsubsection{Lower bound:}\label{orthogonal cirles}
Let $d\geq 4$ and $\alpha\leq \lfloor \frac{d}{2} \rfloor-1$.
To see the lower bound, we consider the following example, which is inspired by the well-known orthogonal-circles example for the unit distance problem in $d\geq 4$ (for instance, see \cite{Lenz}). Assume $d$ is even. Note that one can always reduce to this case by recalling $g_{d+1}\geq g_d$. %For the sake of simplicity, assume $t_1\leq t_2\leq\cdots\leq t_k$. 

We first iteratively choose sets $K_1,\cdots, K_k\subset \mathbb{R}^{\frac{d}{2}}$. More precisely, let $s_1=t_1$ and $K_1$ be an AD regular subset of $s_1 S^{\frac{d}{2}-1}\subset \mathbb{R}^{\frac{d}{2}}$ so that $\dim_{\mathcal{H}}(K_1)=\alpha$, where $s_1 S^{\frac{d}{2}-1}$ denotes the sphere in $\mathbb{R}^{\frac{d}{2}}$ of radius $s_1$ centered at the origin. Let $s_2>0$ be such that $\frac{s_1^2}{2}+\frac{s_2^2}{2}=t_2^2$ (which is possible since $t_2\geq t_1$). Choose $K_2$ to be an AD regular subset of $s_2 S^{\frac{d}{2}-1}\subset \mathbb{R}^{\frac{d}{2}}$. In general, for $2\leq i\leq k$, $K_i$ is an AD regular subset of $s_i S^{\frac{d}{2}-1}\subset \mathbb{R}^{\frac{d}{2}}$ so that $\dim_{\mathcal{H}}(K_i)=\alpha$, where $s_i>0$ satisfies 
\[
\frac{s_{i-1}^2}{2}+\frac{s_i^2}{2}=t_i^2.
\]It is easy to see that one indeed has $s_i>0$, because from the previous step one has $\frac{s_{i-1}^2}{2}< t_{i-1}^2\leq t_i^2$.

Now, define a set
\begin{equation}\label{set E}
\begin{split}
E=&\left\{\Big(\frac{y_i}{\sqrt{2}},0\Big)\in \mathbb{R}^{\frac{d}{2}}\times \mathbb{R}^{\frac{d}{2}}:\, y_i\in K_i,\,\text{for some } i\right\}\\
&\qquad\qquad\bigcup \left\{\Big(0, \frac{y_i}{\sqrt{2}}\Big)\in \mathbb{R}^{\frac{d}{2}}\times \mathbb{R}^{\frac{d}{2}}:\, y_i\in K_i,\,\text{for some } i\right\}.
\end{split}
\end{equation}It is straightforward to check that $E\subset \mathbb{R}^d$ is $\{\delta_i\}$-discrete $\alpha$-regular and has Hausdorff dimension $\alpha$. We claim that  $\dim_{\mathcal{H}}(VS^k_{\vec{t}}(E))\geq (k+1)\alpha$.

To see this, take the case that $k$ is odd as an example. One observes that $VS^k_{\vec{t}}(E)$ contains the following set as a subset:
\[
\left\{\left(\Big( \frac{\tilde{y}_1}{\sqrt{2}}, 0\Big), \Big( 0, \frac{y_1}{\sqrt{2}}\Big), \Big( \frac{y_2}{\sqrt{2}}, 0\Big), \Big( 0, \frac{y_3}{\sqrt{2}}\Big),\ldots, \Big( 0, \frac{y_k}{\sqrt{2}}\Big) \right):\, \tilde{y}_1\in K_1,\, y_i\in K_i,\, \forall i  \right\}.
\](In the case that $k$ is even, one has a similar result with the last component in the set above replaced by $\Big( \frac{y_k}{\sqrt{2}}, 0\Big)$.) Therefore, it implies that $\dim_{\mathcal{H}}(VS^k_{\vec{t}}(E))\geq (k+1)\alpha$, and the proof of the $d\geq 4$ case of Theorem \ref{main4} is complete.

Note that one can easily make the chains obtained above non-degenerate without lowering the dimension of the chain set. Moreover, the above argument extends to general trees with lengths of edges satisfying certain conditions. However, the method fails in certain cases, for instance, the case of $3$-chain of gaps $(1, 2, 1)$. In the discrete setting, a similar issue also exists. In fact, in the discrete setting, when $d\geq 4$, the orthogonal-circles example is the only type of sharp example for the unit distance problem that the authors are aware of, which seems to suggest that the lower bound for certain chains may be very difficult to obtain. 

The above argument does extend to more general chains that do not necessarily satisfy $t_1\leq\cdots\leq t_k$. For example, as mentioned in Remark \ref{rmk: special case} in the Introduction, when $k=3$, the same lower bound $(k+1)\alpha=4\alpha$ holds whenever $t_2^2<t_1^2+t_3^2$. The proof proceeds very similarly as the above, and we give a short sketch here for the sake of completeness. 

\begin{proof}[Proof of Remark \ref{rmk: special case}]
Fix such a vector $\vec{t}=(t_1, t_2, t_3)$. Without loss of generality, assume $t_1\leq t_3$. Again, assume $d$ is even. For $i=1,2,3$, let $K_i$ be an $\alpha$-dimensional AD regular subset of $s_i S^{\frac{d}{2}-1}\subset \mathbb{R}^{\frac{d}{2}}$, with $s_1, s_2, s_3>0$ to be determined later.

If $t_2^2\leq \min(t_1^2, t_3^2)=t_1^2$, choose $s_1=\sqrt{2t_1^2-t_2^2}$, $s_2=t_2$, and $s_3=\sqrt{2t_3^2-t_2^2}$. Then define the set $E$ as in (\ref{set E}), one can observe that $VS^3_{\vec{t}}(E)$ contains the following subset
\[
\left\{\left(\Big( \frac{y_1}{\sqrt{2}}, 0\Big), \Big( 0, \frac{y_2}{\sqrt{2}}\Big), \Big( \frac{\tilde{y}_2}{\sqrt{2}}, 0\Big), \Big( 0, \frac{y_3}{\sqrt{2}}\Big)\right):\, \tilde{y}_2\in K_2,\, y_i\in K_i,\, \forall i  \right\}.
\]hence $\dim_{\mathcal{H}}(VS^3_{\vec{t}}(E))\geq 4\alpha$.

If $t_2^2> \min(t_1^2, t_3^2)=t_1^2$, then since $t_2^2<t_1^2+t_3^2$, there exists $A>1$ (depending on $\vec{t}$) such that $t_2^2<\frac{t_1^2}{A}+t_3^2$. Choose $s_1=t_1$, $s_2=\sqrt{t_2^2-t_1^2/A}$, and $s_3=\sqrt{t_1^2/A+t_3^2-t_2^2}$. Let $\frac{1}{A}+\frac{1}{A'}=1$ and consider the set
\[
\begin{split}
E=&\left\{\Big(\frac{\tilde{y}_1}{\sqrt{A'}},0 \Big)\in \mathbb{R}^{\frac{d}{2}}\times \mathbb{R}^{\frac{d}{2}}:\, \tilde{y}_1\in K_1 \right\}\bigcup \left\{\Big(0, \frac{y_1}{\sqrt{A}} \Big)\in \mathbb{R}^{\frac{d}{2}}\times \mathbb{R}^{\frac{d}{2}}:\, y_1\in K_1 \right\}\\
& \quad\bigcup \left\{(y_2,0)\in \mathbb{R}^{\frac{d}{2}}\times \mathbb{R}^{\frac{d}{2}}:\, y_2\in K_2 \right\} \bigcup  \left\{(0, y_3)\in \mathbb{R}^{\frac{d}{2}}\times \mathbb{R}^{\frac{d}{2}}:\, y_3\in K_3 \right\}.
\end{split}
\]Then $VS^3_{\vec{t}}(E)$ contains the subset
\[
\left\{\left(\Big( \frac{\tilde{y}_1}{\sqrt{A'}}, 0\Big), \Big( 0, \frac{y_1}{\sqrt{A}}\Big), (y_2,0), (0,y_3)\right):\, \tilde{y}_1\in K_1,\, y_i\in K_i,\, \forall i  \right\}.
\]
and the desired estimate follows.
\end{proof}

\subsubsection{Upper bound: The case $d=2$, $0<\alpha\leq 1$, $k=2$}
When $d=2$ and $0<\alpha\leq 1$, we can completely determine the value of $g_2(VS^2_{\vec{t}},\alpha)$. 
%We first establish the upper bound $g_2(VS^2_{\vec{t}},\alpha)\leq 2\alpha$. 
Fix a set $E\subset \mathbb{R}^2$ that is $\{\delta_i\}$-discrete $\alpha$-regular and fix gaps $\vec{t}=(t_1,t_2)\in \mathbb{R}^2_+$. 
We consider the situation where $t_1 = t_2 =1$; the general case can be handled with only small modifications.  

Recall that $VS^2_{(1,1)}(E)=\{(x,y,z)\in E^3:\, |x-y|=1, |y-z|=1, \text{ and } x\neq z\}$. 
Write 
$$VS^2_{(1,1)}(E) = \bigcup_{n=0}^\infty D_n,$$ where
$$D_n = \left\{ (x,y,z) \in VS^2_{(1,1)}(E): \frac{1}{n} \le |x-z| \le 2- \frac{1}{n} \right\},\quad \forall n\geq 1$$ and
$$D_0 = \{ (x,y,z) \in VS^2_{(1,1)}(E):  |x-z| = 2\}.$$
Then, the desired upper bound ${\rm dim}_{\mathcal{H}}(VS^2_{(1,1)}(E))\leq 2\alpha$ will follow from ${\rm dim}_{\mathcal{H}}(D_n)\leq 2\alpha$, $\forall n\in\mathbb{N}$.

We first observe that $\dim_{\mathcal{H}}{(D_0)} \le 2 \alpha.$  
If $(x,y,z)\in D_0,$ then $y = \frac{x+z}{2}$.  Since the map $(x,z) \mapsto (x,\frac{x+z}{2},z)$ is Lipschitz, and $D_0$ is the image of $E\times E$ under this map, the claim follows.  
Here, we have used the assumption that $E$ is $\{\delta_i\}$-discrete $\alpha$-regular to conclude that $\dim_{\mathcal{H}}(E\times E) \le 2\alpha$ (see Section \ref{sec: trivial upper bound} for a deduction).  %Indeed, under this regularity assumption, the Hausdorff dimension and the upper Minkowski dimension of $E$ agree, $\dim_{\mathcal{H}}(E)  = \overline{\dim_M}(E)$, and so $\dim_{\mathcal{H}}(E\times E) \le \dim_{\mathcal{H}}(E) + \overline{\dim_M}(E) \le 2\alpha$ (see, for instance, \cite{Mat95} Theorem 8.10).

We use a similar line of reasoning to show that $\dim_{\mathcal{H}}{(D_n)} \le 2 \alpha$ for each $n\in \mathbb{N}$. 
The only change is that there can be two choices of $y$ so that $(x,y,z) \in D_n$, and expressing each of these choices as a function of $x$ and $z$ is slightly more involved.  
Fix $n$.  
For each $(x,y,z) \in D_n$ and for each  $i\in \{1,2\}$,
define
$$y_i = F_i(x,z)=  x + 
\begin{pmatrix}
\cos{\theta_i} & -\sin{\theta_i}\\
\sin{\theta_i} & \cos{\theta_i}
\end{pmatrix}
\frac{z-x}{|x-z|},$$
 where $\theta$ is the angle between $y-x$ and $z-x$, (so, $\cos{\theta} =\frac{1}{2}|x-z|  $), $\theta_1 = \theta$, and $\theta_2 = -\theta.$
%and 
%$$y_2 = F_2(x,z) =x+
%\begin{pmatrix}
%\cos{\theta} & \sin{\theta}\\
%-\sin{\theta} & \cos{\theta}
%\end{pmatrix}
%\frac{z-x}{|x-z|},$$
%where $\cos{\theta} =\frac{1}{2}|x-z|  $. 

Decompose $D_n$ into $D_n^1 \bigcup D_n^2$, where for $i\in \{1,2\}$,
$$D_n^i = \left\{(x,y_i, z) \in VS^2_{(1,1)}(E): \frac{1}{n} \le |x-z| \le 2- \frac{1}{n} \right \}.$$

By symmetry, 
it suffices to show that $\dim_{\mathcal{H}}{(D_n^1)} \le 2 \alpha.$
As above, it suffices to show that the maps $(x,z) \mapsto (x,F_1(x,z),z)$ is Lipschitz.
A computation shows that there exists a constant $C(n)$ so that
%$|\frac{\partial{F_1^1}{\partial{x_1}}|$
$|\frac{  \partial{F_1}}{\partial{x_1}  }|,    \cdots , |\frac{  \partial{F_1}}{\partial{z_2}  }| \le C(n)$, and the result follows.

%We claim that 
%\[
%\dim_{\mathcal{H}}(VS^2_{\vec{t}}(E))\leq \underline{\rm dim}_{\mathcal{M}}(VS^2_{\vec{t}}(E))\leq 2\alpha.\]To see this, let $N(A,\delta)$ denote the number of $\delta$-balls that are needed to cover the set $A$. Then it suffices to show that for the sequence $\{\delta_i\}$ satisfying the regularity condition (\ref{eqn: delta discrete}), there holds $N(VS^2_{\vec{t}}(E), \delta_i)\lesssim_\epsilon \delta_i^{-2\alpha-\epsilon}$ for sufficiently small $\epsilon>0$.

%

%Let $\{B_j\}_j$ be a covering of $E$ by at most $\delta_i^{-\alpha-\epsilon}$ $\delta_i$-balls.  
%Write 
%$$  \{(x,y,z)\in E:\, |x-y|=t_1, |y-z|=t_2, \text{ and } \} 
%= \bigcup_{n=1}^{  \infty}    \{(x,y,z)\in E^3:\, |x-y|=t_1, |y-z|=t_2\} $$

%The projection of $VS^2_{\vec{t}}(E)$ onto the $x$ variable can be covered by $\delta_i^{-\alpha-\epsilon}$ $\delta_i$-balls, since one has $|E_{\delta_i}|\lesssim \delta_i^{2-\alpha-\epsilon}$. Similarly, The projection of $VS^2_{\vec{t}}(E)$ onto the $z$ variable can be covered by $\delta_i^{-\alpha-\epsilon}$ $\delta_i$-balls. Note that for any fixed pair of points $x, z\in E$, the circles $S(x,t_1)$ and $S(z,t_2)$ intersect at at most two points in $E$. Therefore, in total one has
%\[
%N(VS^2_{\vec{t}}(E), \delta_i)\lesssim_\epsilon \delta_i^{\alpha-\epsilon}\cdot \delta_i^{\alpha-\epsilon}=\delta^{-2\alpha-2\epsilon},
%\]which concludes the proof.

We note that the argument above extends to longer chains, but it is an open problem to find matching upper and lower bounds when $k\geq 3$.

%To see this, fix a set $E\subset \mathbb{R}^2$ that is $\{\delta_i\}$-discrete $\alpha$-regular and gaps $\vec{t}=(t_1,t_2)\in \mathbb{R}^2_+$. Note that for any fixed pair of points $x_1,x_3\in E$, the circles $S(x_1,t_1)$ and $S(x_3,t_2)$ intersect at at most two points $p,q\in E$. Therefore, the map $(x_1, x_2,x_3)\mapsto (x_1,x_3)$ from $VS^2_{\vec{t}}(E)$ to $E\times E$ is at most $2$-to-$1$, hence
%\[
%\dim_{\mathcal{H}}(VS^2_{\vec{t}}(E))\leq \dim_{\mathcal{H}}(E\times E)\leq \underline{\rm dim}_{\mathcal{M}}(E\times E)\leq 2\alpha.
%\]Here, the last inequality follows from the assumption that $E$ is $\{\delta_i\}$-discrete $\alpha$-regular (which implies $|E_{\delta_i}|\lesssim \delta_i^{2-\alpha}$, $\forall i$).

\subsubsection{Lower bound: The case $d=2$, $0<\alpha\leq 1$, $k=2$}
Next, we prove that $g_2(VS^2_{\vec{t}},\alpha)\geq 2\alpha$ by constructing a sharp example, which is inspired by a construction studied in \cite{PSS}. Given gaps $\vec{t}=(t_1,t_2)\in \mathbb{R}^2_+$, for $i=1,2$, let $E_i$ be an AD regular subset of $t_iS^{1}$ so that $\dim_{\mathcal{H}}(E_i)=\alpha$, where $t_i S^1$ denotes the circle of radius $t_i$ centered at the origin. In the case that $t_1=t_2$, choose $E_1,E_2$ that are disjoint. Take another AD regular set $E_3$ that contains the origin with $\dim_{\mathcal{H}}(E_3)=\alpha$, and define $E=E_1\cup E_2\cup E_3$. It is easy to see that $\dim_{\mathcal{H}}(E)=\alpha$, $E$ is $\{\delta_i\}$-discrete $\alpha$-regular, and
\[
VS^2_{\vec{t}}(E)\supset \{(x,0,y):\, x\in E_1, y\in E_2\}.
\]Hence,
\[
\dim_{\mathcal{H}}(VS^2_{\vec{t}}(E))\geq \dim_{\mathcal{H}}(E_1)+\dim_{\mathcal{H}}(E_2)=2\alpha. 
\]The proof of this last case of Theorem \ref{main4} is complete.
\vskip.5in

%%%%%%%%%%%%%%PHONG STEIN THEOREM

\section{Triangles with prescribed gaps: Proof of Theorem \ref{main5}}\label{sec: triangle}

\subsection{The large $\alpha$ case:}
Let $d\geq 3$ and $t_i\sim 1$. In this section, we prove the bound
\[
g_d(V{\rm Tri}_{\vec{t}}, \alpha)\leq \begin{cases} 3\alpha-3,& \frac{2d}{3}+1\leq\alpha\leq d,\\ d+\frac{3\alpha}{2}-\frac{3}{2},& 0<\alpha\leq\frac{2d}{3}+1. \end{cases}
\]The proof strategy is adapted from that of Theorem \ref{main3}, but in the triangle case, one cannot expect to do iteration. Instead, we will directly prove an $L^2$ bound that is adapted to the triangle case, which involves an estimate of the Fourier transform of the surface measure of a surface that is not the sphere anymore. 

To begin with, let $E$ be a $\{\delta_i\}$-discrete $\alpha$-regular set with $\dim_{\mathcal{H}}(E)=\alpha$. Without loss of generality, assume that $E=-E$. We aim to prove that for all $i$ and $\epsilon>0$,
\begin{equation}\label{eqn: triangle 1}
\begin{split}
|D^{\delta_i}|:=&|\{(x,y,z)\in E_{\delta_i}^3:\, t_1-2\delta_i\leq |x-y|\leq t_1+2\delta_i,\,\\
& \qquad\qquad t_2-2\delta_i\leq |y-z|\leq t_2+2\delta_i,\, t_3-2\delta_i\leq |x-z|\leq t_3+2\delta_i\}|\\
\lesssim_\epsilon &\delta_i^{3d-\gamma(d,\alpha)-\epsilon},
\end{split}
\end{equation}where
\[
\gamma(d,\alpha):=\begin{cases} 3\alpha-3,& \frac{2d}{3}+1\leq \alpha\leq d,\\ d+\frac{3\alpha}{2}-\frac{3}{2},& \alpha\leq \frac{2d}{3}+1. \end{cases}
\]Again, we denote $\delta=\delta_i$ for the sake of simplicity.

It is direct to see that
\[
\begin{split}
|D^\delta|=& \int_{E_\delta} \int_{E_\delta}\int_{E_\delta} \chi_{A_{t_1,\delta}}(y-x) \chi_{A_{t_2,\delta}}(z-y)\chi_{A_{t_3,\delta}}(x-z)\,dxdydz\\
\sim&\int \int \int \chi_{E_\delta}(z)\chi_{E_\delta}(z+x)\chi_{E_\delta}(z+y)\chi_{S_{\vec{t},\delta}}(x,y)\,dxdydz,
\end{split}
\]where $S_{\vec{t},\delta}$ denotes the $\delta$-neighborhood of the surface
\[
S_{\vec{t}}:=\{(x,y)\in \mathbb{R}^d\times \mathbb{R}^d:\, |x|=t_1,\, |y|=t_2,\, |x-y|=t_3\}.
\]

Let $\rho$ be the symmetric Schwartz function on $\mathbb{R}^d$ introduced above in \eqref{rho}, and set $\rho_r(x)=r^{-d}\rho(\frac{x}{r})$. Then, one has
\[
\chi_{S_{\vec{t},\delta}}(x,y)\lesssim \delta^3 (\rho_\delta\otimes\rho_\delta)\ast \sigma_{S_{\vec{t}}}(x,y)=\delta^3\int\int \rho_\delta(x-u)\rho_\delta(y-v)\,d\sigma_{S_{\vec{t}}}(u,v),
\]where $\sigma_{S_{\vec{t}}}(x,y)$ denotes the normalized surface measure of $S_{\vec{t}}$. Therefore,
\[
|D^\delta|\lesssim \delta^3\int\int \left(\int \chi_{E_\delta}(z)\chi_{E_\delta}(z+x)\chi_{E_\delta}(z+y)\,dz \right) (\rho_\delta\otimes \rho_{\delta})\ast \sigma_{S_{\vec{t}}}(x,y)\,dxdy.
\]

With $x,y$ fixed, the inner integral is
\[
\begin{split}
&\int \chi_{E_\delta}(z)\chi_{E_\delta}(z+x)\chi_{E_\delta}(z+y)\,dz\\
=& \int\int\int\int \hat\chi_{E_\delta}(\xi)\hat\chi_{E_\delta}(\eta)\hat\chi_{E_\delta}(\zeta)e^{i[(z,z,z)+(x,y,0)]\cdot (\xi, \eta,\zeta)}\,d\xi d\eta d\zeta dz.
\end{split}
\]Hence, 
\[
\begin{split}
|D^\delta|\lesssim &\delta^3 \int\int\int\int \hat\chi_{E_\delta}(\xi)\hat\chi_{E_\delta}(\eta)\hat\chi_{E_\delta}(\zeta) e^{i(z,z,z)\cdot(\xi,\eta,\zeta)}\cdot\\
&\qquad\qquad\qquad\qquad\hat\sigma_{S_{\vec{t}}}(-\xi,-\eta)\hat{\rho_\delta}(-\xi)\hat{\rho_\delta}(-\eta)\, d\xi d\eta d\zeta dz\\
=& \delta^3\int\int \hat\chi_{E_\delta}(\xi)\hat\chi_{E_\delta}(\eta) \hat\chi_{E_\delta}(-\xi-\eta)\hat\sigma_{S_{\vec{t}}}(-\xi,-\eta) \hat{\rho_\delta}(-\xi)\hat{\rho_\delta}(-\eta)\, d\xi d\eta.
\end{split}
\]In the last step above, we have used the observation that
\[
\int e^{i(z,z,z)\cdot(\xi,\eta,\zeta)}\,dz= \delta_0(\xi+\eta+\zeta),
\]where $\delta_0$ denotes the Dirac $\delta$ function at the origin.

Similarly as in the proof of Theorem \ref{main3}, we will estimate the above integral by decomposing it into single scales. At each single scale, one assumes $|\xi|+|\eta|\sim \frac{C}{2^j\delta}$, where $\frac{1}{2}\leq 2^j\leq \frac{1}{\delta}$. Observe that in this case, at least two of $|\xi|, |\eta|, |\xi+\eta|$ are $\sim \frac{C}{2^j\delta}$.

To see why it suffices to reduce to the range $\frac{1}{2}\leq 2^j\leq \frac{1}{\delta}$, one observes that $2^j \geq \frac{1}{2}$ follows from $|\xi|+|\eta|\lesssim\frac{C}{\delta}$, which is a consequence of the Fourier decay property of $\rho_\delta$ and $t\sim 1$. On the other hand, $2^j\leq \frac{1}{\delta}$ follows from $|\xi|+|\eta|\geq C$. This is because the integral above over the domain $|\xi|+|\eta|\leq C$ can be trivially bounded by
\[
 |E_\delta|^3 \|\sigma_{S_{\vec{t}}}\|_{L^1}\lesssim \delta^{3(d-\alpha-\epsilon)},
\]which implies that the contribution of this part to $|D^\delta|$ is bounded by $\delta^{3d-(3\alpha-3)-\epsilon}$ as desired. (Note that the case for $\alpha<\frac{2d}{3}+1$ then also trivially follows since in that range one has $3\alpha-3<d+\frac{3\alpha}{2}-\frac{3}{2}$.)

In the following, we discuss two sub-cases for the single scale $|\xi|+|\eta|\sim \frac{C}{2^j\delta}$, with the first one being the main case.

\subsubsection{Case $|\xi|\sim|\eta|\sim \frac{C}{2^j\delta}$}

By symmetry and dropping the $\delta^3$ for now, it suffices to estimate
\begin{equation}\label{eqn: triangle 2}
\int\int_{|\xi|\sim|\eta|\sim \frac{C}{2^j\delta}} \hat\chi_{E_\delta}(\xi)\hat\chi_{E_\delta}(\eta) \hat\chi_{E_\delta}(\xi+\eta)\hat\sigma_{S_{\vec{t}}}(\xi,\eta) \hat{\rho_\delta}(\xi)\hat{\rho_\delta}(\eta)\, d\xi d\eta.
\end{equation}We will be using the following estimate, slightly generalizing Lemma 2.3 of \cite{IL}:

\begin{lemma}
Let $S_{\vec{t}}$ be the surface defined above, and let $\theta$ denote the angle between the $t_1$-side and the $t_2$-side of the triangle. Suppose $|\xi|\sim |\eta|$, then
\[
|\hat\sigma_{S_{\vec{t}}}(\xi,\eta)|\lesssim_{\vec{t}} \left|\xi+\frac{t_2}{t_1}g_{\theta}(\eta)\right|^{-\frac{1}{2}}|\xi|^{-(d-2)}(\sin\langle \xi,\eta\rangle)^{-\frac{d-2}{2}},
\]where $g_{\theta}\in O(d)$ is some rotation by $\theta$ and $\langle \xi,\eta\rangle$ denotes the angle between $\xi,\eta$.
\end{lemma} 

The proof of the lemma follows exactly the same lines as \cite[Lemma 2.3]{IL}, where the equilateral triangle case was discussed. We omit the proof of the general case.

Applying this lemma, one obtains 
\[
\begin{split}
(\ref{eqn: triangle 2})
\lesssim &(2^j\delta)^{d-2} \int\int_{|\xi|\sim|\eta|\sim \frac{C}{2^j\delta}}  |\hat\chi_{E_\delta}(\xi)||\hat\chi_{E_\delta}(\eta)| |\hat\chi_{E_\delta}(\xi+\eta)| |\hat{\rho_\delta}(\xi)\hat{\rho_\delta}(\eta)|\cdot \\
&\qquad\qquad\qquad \left|\xi+\frac{t_2}{t_1}g_{\theta}(\eta)\right|^{-\frac{1}{2}}(\sin\langle \xi,\eta\rangle)^{-\frac{d-2}{2}}\,d\xi d\eta\\
= & (2^j\delta)^{d-2} \int\int_{|\xi|\sim|\eta|\sim \frac{C}{2^j\delta}}  |\widehat{\chi_{E_\delta}\ast\rho_\delta}(\xi)||\widehat{\chi_{E_\delta}\ast\rho_\delta}(\eta)| |\hat\chi_{E_\delta}(\xi+\eta)| \cdot\\
& \qquad\qquad\qquad \left|\xi+\frac{t_2}{t_1}g_{\theta}(\eta)\right|^{-\frac{1}{2}}(\sin\langle \xi,\eta\rangle)^{-\frac{d-2}{2}}\,d\xi d\eta.
\end{split}
\]

We claim that for any fixed $\eta$,
\begin{equation}\label{eqn: triangle 3}
\int_{|\xi|\sim \frac{C}{2^j\delta}} |\hat\chi_{E_\delta}(\xi+\eta)|  \left|\xi+\frac{t_2}{t_1}g_{\theta}(\eta)\right|^{-\frac{1}{2}}(\sin\langle \xi,\eta\rangle)^{-\frac{d-2}{2}}\,d\xi \lesssim_\epsilon 2^{j\frac{\alpha-2d+1+\epsilon}{2}}\delta^{\frac{1-\alpha-3\epsilon}{2}},
\end{equation}and similarly, when $\xi$ is fixed, 
\begin{equation}\label{eqn: triangle 3.5}
\int_{|\eta|\sim \frac{C}{2^j\delta}} |\hat\chi_{E_\delta}(\xi+\eta)|  \left|\xi+\frac{t_2}{t_1}g_{\theta}(\eta)\right|^{-\frac{1}{2}}(\sin\langle \xi,\eta\rangle)^{-\frac{d-2}{2}}\,d\eta \lesssim_\epsilon 2^{j\frac{\alpha-2d+1+\epsilon}{2}}\delta^{\frac{1-\alpha-3\epsilon}{2}}.
\end{equation}

Assume (\ref{eqn: triangle 3}) and (\ref{eqn: triangle 3.5}) for now, then by Schur's test, one obtains
\[
\begin{split}
(\ref{eqn: triangle 2})\lesssim_\epsilon &(2^j\delta)^{d-2}2^{j\frac{\alpha-2d+1+\epsilon}{2}}\delta^{\frac{1-\alpha-3\epsilon}{2}} \int_{|\xi|\sim \frac{C}{2^j\delta}} |\widehat{\chi_{E_\delta}\ast\rho_\delta}(\xi)|^2\,d\xi\\
\lesssim_\epsilon & (2^j\delta)^{d-2}2^{j\frac{\alpha-2d+1+\epsilon}{2}}\delta^{\frac{1-\alpha-3\epsilon}{2}} 2^{j(\alpha-d+\epsilon)}\delta^{d-\alpha-3\epsilon}=(2^j\delta)^{\frac{3\alpha}{2}-d-\frac{3}{2}+\frac{3}{2}\epsilon}\delta^{3d-3\alpha-6\epsilon},
\end{split}
\]where the second inequality follows from the $L^2$ estimate used above in the proof of Theorem \ref{main3}, more precisely from  (\ref{eqn: use for triangle}).

When $\alpha\geq \frac{2d}{3}+1$, since $2^j\delta\leq 1$, one has $(2^j\delta)^{\frac{3\alpha}{2}-d-\frac{3}{2}}\leq 1$, hence the contribution of this case to $|D^\delta|$ is bounded by
\[
\lesssim_\epsilon \delta^3\sum_{\frac{1}{2}\leq 2^j\leq \frac{1}{\delta}} \delta^{3d-3\alpha-\epsilon}\leq \log(\frac{1}{\delta})\delta^{3d-(3\alpha-3)-\epsilon},
\]which matches the desired estimate. When $\alpha< \frac{2d}{3}+1$, one has  
\[
(\ref{eqn: triangle 2})\lesssim_\epsilon (2^j\delta)^{\frac{3\alpha}{2}-d-\frac{3}{2}+\frac{3}{2}\epsilon}\delta^{3d-3\alpha-6\epsilon}=(2^j)^{\frac{3\alpha}{2}-d-\frac{3}{2}+\frac{3}{2}\epsilon}\delta^{2d-\frac{3\alpha}{2}-\frac{3}{2}-6\epsilon}.
\]Since $\alpha< \frac{2d}{3}+1$, the sum over $j$ converges when $\epsilon$ is sufficiently small, therefore, the contribution of this case to $|D^\delta|$ is bounded by
\[
\lesssim_\epsilon \delta^3\delta^{2d-\frac{3\alpha}{2}-\frac{3}{2}-\epsilon}=\delta^{3d-(d+\frac{3\alpha}{2}-\frac{3}{2})-\epsilon},
\]which completes the proof of this case.

It remains to prove estimates (\ref{eqn: triangle 3}) and (\ref{eqn: triangle 3.5}). We only prove (\ref{eqn: triangle 3}) below, as the other one can be obtained in the same way after a change of variable $\xi\mapsto \frac{t_1}{t_2}g_{-\theta}(\xi)$. By Cauchy-Schwarz, one has (\ref{eqn: triangle 3}) bounded by
\[
\left(\int_{|\xi|\sim \frac{C}{2^j\delta}} |\hat\chi_{E_\delta}(\xi+\eta)|^2\,d\xi \right)^{1/2}\left(\int_{|\xi|\sim \frac{C}{2^j\delta}}  \left|\xi+\frac{t_2}{t_1}g_{\theta}(\eta)\right|^{-1}(\sin\langle \xi,\eta\rangle)^{-(d-2)}\,d\xi\right)^{1/2}.
\]
According to Lemma 2.3 of \cite{IL}, and observing that the same bound holds true in the case of general triangles, the second factor is bounded by $(\frac{C}{2^j\delta})^{\frac{-1+d}{2}}$. 

The first factor can be estimated similarly as above. Since $|\xi|\sim|\eta|\sim \frac{C}{2^j\delta}$, after a change of variable, one has
\[
\int_{|\xi|\sim \frac{C}{2^j\delta}} |\hat\chi_{E_\delta}(\xi+\eta)|^2\,d\xi\lesssim \int_{|\xi|\lesssim \frac{C}{2^j\delta}} |\hat\chi_{E_\delta}(\xi)|^2\,d\xi\lesssim \int_{|\xi|\lesssim \frac{C}{2^j\delta}} |\widehat{\chi_{E_\delta}\ast \rho_r}(\xi)|^2\,d\xi,
\]which follows from the choice $r\sim 2^j\delta$ and the property $\hat{\rho_r}(\xi)=\hat{\rho}(r\xi)\gtrsim 1$ when $|\xi|\lesssim \frac{C}{2^j\delta}$. According to (\ref{eqn: use for triangle}) (also see estimate (2.11) of \cite{OO}), the above is further
\[
\lesssim \|\chi_{E_\delta}\ast \rho_r\|_{L^2}^2\lesssim_\epsilon r^{\alpha-d+\epsilon}\delta^{2(d-\alpha-2\epsilon)}\sim 2^{j(\alpha-d+\epsilon)}\delta^{d-\alpha-3\epsilon}.
\]

Therefore, one obtains
\[
(\ref{eqn: triangle 3})\lesssim_\epsilon 2^{j\frac{\alpha-d+\epsilon}{2}}\delta^{\frac{d-\alpha-3\epsilon}{2}} (2^j\delta)^{\frac{1-d}{2}}=2^{j\frac{\alpha-2d+1+\epsilon}{2}}\delta^{\frac{1-\alpha-3\epsilon}{2}},
\]and the proof of the first case is complete.

\subsubsection{Case $|\xi|\sim |\xi-\eta|\sim \frac{C}{2^j\delta}$}\label{sec: triangle symmetry}

The second case can be reduced to the first case above by making use of the symmetry of the surface measure $\sigma_{S_{\vec{t}}}$. Such an argument in the equilateral triangle case was derived in \cite{IL} for the study of a related problem, and we only sketch the general case here for the sake of completeness. 

For any $(x^0,y^0)$ such that $\Delta_{x^0 0 y^0}$ forms a triangle of sidelengths $t_1,t_2,t_3$ (to be more specific, we assume $|x^0|=t_1$, $|y^0|=t_2$, and $|x^0-y^0|=t_3$). One observes that
\[
\int f(x,y)d\sigma_{S_{\vec{t}}}(x,y)=\int_{O(d)}f(gx^0,gy^0)\,dg,
\]where $O(d)$ denotes the orthogonal group in $\mathbb{R}^d$ and $dg$ denotes the normalized Haar measure on $O(d)$. A direct computation shows that
\[
\hat\sigma_{S_{\vec{t}}}(\xi,\eta)=\int e^{-2\pi i(g(y^0-x^0)\cdot (-\xi)+gy^0\cdot (\xi+\eta))}\,dg.
\]

Observe that $\Delta_{(y^0-x^0)0 y^0}$ is a triangle with a permutation of the original three sides. Write $\vec{s}:=(t_3, t_2, t_1)$, then the point $(y^0-x^0,y^0)\in S_{\vec{s}}$. One thus has $\hat\sigma_{S_{\vec{t}}}(\xi,\eta)=\hat\sigma_{S_{\vec{s}}}(-\xi,\xi+\eta)$. Therefore, by a change of variable $\zeta=\xi+\eta$, one can reduce the estimate to the first case above (with $|\xi|\sim |\zeta|\sim \frac{C}{2^j\delta}$ and a new surface $S_{\vec{s}}$). The estimate in the first case obviously still holds for the surface $S_{\vec{s}}$, hence the proof is complete.

\vskip.125in
\subsection{The large $\alpha$ case, $d=2$}

As discussed in the introduction, when $\alpha\leq \frac{7}{4}$, the upper estimate follows from the corresponding bounds for the unit distance problem, obtained in \cite[Theorem 1.3]{OO}. Therefore, in this subsection, it suffices to prove that $g_2(V{\rm Tri}_{\vec{t}},\alpha)\leq 3\alpha-3$ if $\alpha\geq \frac{7}{4}$. 

Note that the argument in the previous section regarding the $d\geq 3$ case still works when $d=2$. (One does need to slightly change the argument in Section \ref{sec: triangle symmetry}, where the rotation symmetry only holds after decomposing the surface $S_{\vec{t}}$ to two parts.) However, in order to obtain the upper bound $3\alpha-3$ for $\alpha\geq \frac{7}{4}$, the argument fails to be sufficient. 

Below we adapt a method originated in \cite{GI}.

Recall from the previous subsection that, letting $D^\delta$ be as in (\ref{eqn: triangle 1}), one has
\[
|D^\delta|=\int_{E_\delta} \int_{E_\delta}\int_{E_\delta} \chi_{A_{t_1,\delta}}(y-x) \chi_{A_{t_2,\delta}}(z-y)\chi_{A_{t_3,\delta}}(x-z)\,dxdydz.
\]Let $\sigma_t$ denote the surface measure on the circle of radius $t$ in $\mathbb{R}^2$ and $\sigma_t^\delta:=\sigma_t\ast \rho_\delta$. Here, $\rho_\delta(\cdot):=\delta^{-2}\rho(\frac{\cdot}{\delta})$ is an approximate identity with $\rho\in C_0^\infty([0,1]^2)$, $\rho\geq 0$ and $\int \rho=1$. Then, one has that
\[
|D^\delta|\lesssim \delta^3 \int\int\int \chi_{E_{\delta}}\ast\rho_\delta (x) \chi_{E_{\delta}}\ast\rho_\delta (y) \chi_{E_{\delta}}\ast\rho_\delta (z)\sigma_{t_1}^\delta(y-x) \sigma_{t_2}^\delta(z-y)\sigma_{t_3}^\delta(x-z)\,dxdydz.
\]

In general, define
\[
\Lambda^\delta_{\vec{t}}(f_1,f_2,f_3):=\int\int\int\sigma_{t_1}^\delta(y-x) \sigma_{t_2}^\delta(z-y)\sigma_{t_3}^\delta(x-z) f_1(x) f_2(y) f_3(z)\,dxdydz.
\]Then one has $|D^\delta|\lesssim \delta^3\Lambda^\delta_{\vec{t}}(\chi_{E_\delta}\ast \rho_\delta,\chi_{E_\delta}\ast\rho_\delta, \chi_{E_\delta}\ast\rho_\delta)$.

Define $f_\delta^\beta(x):=\frac{1}{\Gamma(\beta/2)}\chi_{E_\delta}\ast\rho_\delta\ast |\cdot|^{-2+\beta}(x)$ (initially defined for ${\rm Re}(\beta)>0$ and extended to the complex plane by analytic continuation, see Gelfand--Shilov \cite[p.74]{GS}). When $\beta=0$, $f_\delta^0$ becomes a constant multiple of $\chi_{E_\delta}\ast\rho_\delta$ convolved with the Dirac delta function, and hence is equal to $\chi_{E_\delta}\ast\rho_\delta$ up to a constant (since convolving with the delta function gives back the original function). 
 In the following, we will only be interested in the range ${\rm Re}(\beta)\leq \frac{1}{4}$. Note that $f_\delta^\beta\in L^2(\mathbb{R}^2)$, which follows from Plancherel and properties of $\rho_\delta$, as well as the estimate $|E_\delta|\lesssim \delta^{d-\alpha-\epsilon}$. 

Now, define $F(\beta):=\Lambda^\delta_{\vec{t}}(f_\delta^{-\beta}, \chi_{E_\delta}\ast \rho_\delta, f_\delta^{\beta})=:\langle B(f_\delta^{-\beta}, \chi_{E_\delta}\ast \rho_\delta), f_\delta^\beta\rangle$, where the bilinear operator
\[
B(g,h):=\int\int g(x-u) h(x-v) \sigma_a^\delta(u) \sigma_b^\delta(v) \sigma^\delta(u-v)\,dudv.
\]Here, we have rescaled the triangle to make it have side lengths $(a,b,1)$, where $0<a,b\lesssim 1$. Our main estimate is the following.

\begin{lemma}\label{lem: triangle 2D}
Let $\alpha\geq \frac{7}{4}$. Then for all $\beta$ satisfying $-\frac{1}{4}\leq {\rm Re}(\beta)\leq \frac{1}{4}$, we have $|F(\beta)|\lesssim_\epsilon \delta^{3(2-\alpha)-\epsilon}$.
\end{lemma}

It is easy to see that the lemma would imply that $|D^\delta|\lesssim \delta^3F(0)\lesssim \delta^{3+3(2-\alpha)-\epsilon}$, which would imply that $\dim_{\mathcal{H}}({\rm VTri}_{\vec{t}},\alpha)\leq 3\alpha-3$ when $\alpha\geq\frac{7}{4}$.  

\begin{proof}[Proof of Lemma \ref{lem: triangle 2D}]
Note that the domain of the integral defining $F(\beta)$ is compact, hence one has a trivial upper bound of $F(\beta)$ (depending on $\delta$). According to the Three Line Lemma, it suffices to check the desired bound at ${\rm Re}(\beta)=\pm\frac{1}{4}$. Furthermore, it suffices to study the ${\rm Re}(\beta)=\frac{1}{4}$ case, since by interchanging the role of the first and the third input in $F(\beta)$, the other case is symmetric.

%One first observes, by taking the modulus in the definition, that $|f_\delta^\beta|\lesssim f_\delta^{{\rm Re}(\beta)}$, which implies that
One first has
\[
|F(\beta)|= |\langle B(f_\delta^{-\beta}, \chi_{E_\delta}\ast\rho_\delta), f_\delta^{\beta} \rangle|\lesssim \|B(f_\delta^{-\beta}, \chi_{E_\delta}\ast\rho_\delta)\|_{L^1(\mathbb{R}^2)}\|f_\delta^{\beta}\|_{L^\infty(\mathbb{R}^2)}.
\]It is proved in \cite[Theorem 3.1]{GI} that
\[
B:\, L^2_{-\beta_1}(\mathbb{R}^2)\times L^2_{-\beta_2}(\mathbb{R}^2)\to L^1(\mathbb{R}^2),\quad \text{if } \beta_1+\beta_2=\frac{1}{2},\, \beta_1,\beta_2\geq 0,
\]with constant independent of $\delta$. Here, $\|f\|^2_{L^2_s}:=\int |\hat{f}(\xi)|^2 (1+|\xi|)^{2s}\,d\xi$. Combined with the above, this implies that
\[
|F(\beta)|\lesssim \|f_\delta^{-\beta}\|_{L^2_{-\frac{3}{8}}}\|\chi_{E_\delta}\ast\rho_\delta\|_{L^2_{-\frac{1}{8}}}\|f_\delta^{\beta}\|_{L^\infty}.
\]Then the desired bound easily follows from Lemma \ref{lem: triangle 2D 1} and \ref{lem: triangle 2D 2} below.
\end{proof}

\begin{lemma}\label{lem: triangle 2D 1}
If $\alpha\geq \frac{7}{4}$ and ${\rm Re}(\beta)=\frac{1}{4}$, then $\|f_\delta^{\beta}\|_{L^\infty(\mathbb{R}^2)}\lesssim_\epsilon \delta^{2-\alpha-\epsilon}$.
\end{lemma}

\begin{proof}
One first observes, by taking the modulus in the definition, that $|f_\delta^\beta|\lesssim f_\delta^{{\rm Re}(\beta)}$ when ${\rm Re}(\beta)>0$. This, combined with assumption on $\rho_\delta$ and the definition of $f_\delta^\beta$, implies that 
\[
\|f_\delta^{\beta}\|_{L^\infty(\mathbb{R}^2)}\lesssim \|f_\delta^{\frac{1}{4}}\|_{L^\infty(\mathbb{R}^2)}\leq\|\chi_{E_\delta}\ast |\cdot|^{-2+\frac{1}{4}}\|_{L^\infty}\|\rho_\delta\|_{L^1}\lesssim \sup_x\int \chi_{E_\delta}(y)|x-y|^{-\frac{7}{4}}\,dy.
\]

Fix any $x\in\mathbb{R}^2$, one has from the trivial estimate that
\[
\int_{|x-y|\gtrsim 1} \chi_{E_\delta}(y)|x-y|^{-\frac{7}{4}}\,dy\lesssim |E_\delta|\lesssim_\epsilon \delta^{2-\alpha-\epsilon}.
\]For the other part, we decompose the integral as
\[
\int_{|x-y|\lesssim 1} \chi_{E_\delta}(y)|x-y|^{-\frac{7}{4}}\,dy\lesssim\sum_{k=1}^\infty 2^{\frac{7}{4}k}\int_{|x-y|\sim 2^{-k}}\chi_{E_\delta}(y)\,dy.
\]

If $2^{-k}\geq \delta$, by the $\{\delta_i\}$-discrete $\alpha$-regular assumption, 
\[
\int_{|x-y|\sim 2^{-k}}\chi_{E_\delta}(y)\,dy\lesssim_\epsilon 2^{-k(\alpha+\epsilon)}\delta^{2-\alpha-\epsilon}.
\]If $2^{-k}<\delta$, one simply uses $\int_{|x-y|\sim 2^{-k}}\chi_{E_\delta}(y)\,dy\lesssim 2^{-2k}$. Combining the two cases together, one obtains
\[
\begin{split}
& \int_{|x-y|\lesssim 1} \chi_{E_\delta}(y)|x-y|^{-\frac{7}{4}}\,dy\\
\lesssim_\epsilon & \sum_{k=1}^{\log(\delta^{-1})} 2^{\frac{7}{4}k}2^{-k\alpha}\delta^{2-\alpha-\epsilon}+\sum_{k=\log(\delta^{-1})}^\infty 2^{\frac{7}{4}k} 2^{-2k}\lesssim \delta^{2-\alpha-\epsilon}.
\end{split} 
\]

\end{proof}

\begin{lemma}\label{lem: triangle 2D 2}
If $\alpha\geq \frac{7}{4}$ and ${\rm Re}(\beta)=\frac{1}{4}$, then $\|f_\delta^{-\beta}\|_{L^2_{-\frac{3}{8}}(\mathbb{R}^2)}, \,\|\chi_{E_\delta}\ast\rho_\delta\|_{L^2_{-\frac{1}{8}}(\mathbb{R}^2)}\lesssim_\epsilon  \delta^{2-\alpha-\epsilon}$.
\end{lemma}

\begin{proof}
By definition and Fourier inversion, one has
\[
\|\chi_{E_\delta}\ast\rho_\delta\|^2_{L^2_{-\frac{1}{8}}(\mathbb{R}^2)}\lesssim\int |\hat{\chi}_{E_\delta}|^2(\xi) |\xi|^{-\frac{1}{4}}\,d\xi\sim \int_{E_\delta}\int_{E_\delta} |x-y|^{-\frac{7}{4}}\,dxdy.
\]Applying the bound in Lemma \ref{lem: triangle 2D 1} for each fixed $x$, one has the integral on the right hand side above is
\[
\lesssim |E_\delta|\cdot \delta^{2-\alpha-\epsilon}\lesssim \delta^{2(2-\alpha)-\epsilon}.
\]

The other term can be estimated similarly. 
\[
\begin{split}
\|f_\delta^{-\beta}\|^2_{L^2_{-\frac{3}{8}}(\mathbb{R}^2)}\lesssim \int \left|(f_\delta^{-\beta})^\wedge(\xi)\right|^2|\xi|^{-\frac{3}{4}}\,d\xi\lesssim &\int |\hat{\chi}_{E_\delta}|^2(\xi) |\xi|^{2\beta-\frac{3}{4}}\,d\xi\\
\lesssim &\int |\hat{\chi}_{E_\delta}|^2(\xi) |\xi|^{-\frac{1}{4}}\,d\xi.
\end{split}
\]Hence, the desired bound follows in the same way as above.
\end{proof}

\subsection{The small $\alpha$ case, $d\geq 6$}

In this section, we restrict to the case $d\geq 6$ and $\alpha\leq \lfloor \frac{d}{3}\rfloor-1$, and our goal is to prove that $g_d(V{\rm Tri}_{\vec{t}},\alpha)=3\alpha$, whenever $\vec{t}=(t_1,t_2,t_3)$ forms an acute triangle. The upper bound is trivial, which can be shown by the same argument in Section \ref{sec: trivial upper bound}, hence it suffices to find an example establishing the lower bound. 

Since $g_{d+1}\geq g_d$, it suffices to consider the case that $d$ is an integer multiple of $3$. Since the triangle is acute, there exist $A,B,C>0$ satisfying
\[
A+B=t_1^2,\quad, B+C=t_2^2,\quad C+A=t_3^2.
\]Let $K_A\subset A^{1/2}S^{\frac{d}{3}-1}$ be an AD regular set of Hausdorff dimension $\alpha$, and similarly define $K_B\subset B^{1/2}S^{\frac{d}{3}-1}, K_C\subset C^{1/2}S^{\frac{d}{3}-1}$. Define the set $E=E_A\cup E_B\cup E_C$ where
\[
E_A:=\{(x,0,0)\in \mathbb{R}^{\frac{d}{3}}\times \mathbb{R}^{\frac{d}{3}}\times \mathbb{R}^{\frac{d}{3}}:\, x\in K_A\},
\]
\[
E_B:=\{ (0, x,0)\in \mathbb{R}^{\frac{d}{3}}\times \mathbb{R}^{\frac{d}{3}}\times \mathbb{R}^{\frac{d}{3}}:\, x\in  K_B\},
\]
\[
E_C:=\{(0,0,x)\in \mathbb{R}^{\frac{d}{3}}\times \mathbb{R}^{\frac{d}{3}}\times \mathbb{R}^{\frac{d}{3}}:\, x\in K_C\}.
\]

Then, it is easy to see that any three points in $E_A, E_B, E_C$ form a triangle of the given sidelength $t_1,t_2,t_3$. Therefore, $\dim_{\mathcal{H}}(V{\rm Tri}_{\vec{t}}(E))\geq 3\alpha$.

\bigskip
\section{Phong-Stein condition: Proof of Theorem \ref{main3PhongStein} }\label{proof of main3PhongStein}

Let $k\geq 1$ and $E\subset \mathbb{R}^d$ be a compact $\{\delta_i\}$-discrete $\alpha$-regular set that is contained in the unit ball and has Hausdorff dimension $\alpha$. 
Assume that $\vec{t}=(t_1,\cdots, t_{k})\in \mathbb{R}^k_+$ and $t_i\sim 1$, $i=1,\ldots, k$. 
Define
$$D_{k, \phi}^\delta  : =   \{ (x_1, \cdots, x_{k+1})\in E^{k+1}_{\delta} : | \phi(x_i, x_{i+1} ) - t_i|  \le \delta \text{  for each } i =1, \dots, k\},$$
where $\phi$ is as in the statement of Theorem \ref{main3PhongStein} and $E^{k+1}_{\delta}$ denotes the $(k+1)$-fold Cartesian product of the set $E_{\delta}$. Note that we have suppressed the subscript $\vec{t}$ to keep the notation more concise.

Given any $i$, we will show that $\forall \epsilon>0$, 

\begin{equation}\label{eqn: unit distance upp main phi}
|D_{k, \phi}^{\delta_i} | \lesssim \delta_i^{(k+1)d-u(k,d,\alpha)-\epsilon},
\end{equation}where
\[
u(k,d,\alpha):=\begin{cases} (k+1)\alpha-k, & \frac{d+1}{2}\leq \alpha\leq d,\\ \frac{kd}{2}+\alpha-\frac{k}{2},& \alpha\leq \frac{d+1}{2}.\end{cases}
\]  

Since $D_{k,\phi}^{\delta_i}$ contains the $\delta_i$-neighborhood of $VS^k_{\phi}(E)$, estimate (\ref{eqn: unit distance upp main phi}) implies that $\underline{\rm dim}_{\mathcal{M}}(VS^k_{\phi}(E))\leq u(k,d,\alpha)$, where $\underline{\rm dim}_{\mathcal{M}}$ denotes the lower Minkowski dimension, hence the desired upper bound in Theorem \ref{main3PhongStein} follows. We write $\delta=\delta_i$ below.

Letting $\Psi: \mathbb{R} \rightarrow \mathbb{R}$ and $\Psi_0:\mathbb{R}^d\times \mathbb{R}^d \rightarrow \mathbb{R}$ denote non-negative smooth bump functions centered at the origin in $\mathbb{R}$ and $\mathbb{R}^{2d}$ respectively, the $(k+1)$-fold Lebesuge measure of $D_{k, \phi}^\delta$  is comparable to
\begin{equation}\label{blahsingle}
\int_{E_{\delta}} \cdots  \int_{E_{\delta}} \prod_{i=1}^{k} \Psi  \left( \frac{  \phi(x_i,x_{i+1}) - t_i}{\delta}  \right)\Psi_0(x_i,x_{i+1})  \, dx_1\cdots dx_{k+1}.
  \end{equation}
%\vskip.125in

Setting
\begin{equation}\label{eqn op}
Tf(x): =T_{\phi}^{\delta}(f)(x) :=   \frac{1}{\delta}  \int  f(y) \Psi\left( \frac{  \phi(x,y) - t}{\delta}  \right)  \Psi_0(x,y)  \,dy, 
\end{equation}
%the expression in \eqref{blahsingle} becomes
we have 
\begin{equation}\label{blahblah}
|  D_{k, \phi}^\delta | \sim 
\delta^k \int_{E_{\delta}}  T( \chi_{E_{\delta}} f_{k-1})(x)\, dx,
\end{equation}
where we set $f_0(x) =\chi_{E_{\delta}} (x)$ and  $f_{i}(x) = T(f_{i-1} \chi_{E_{\delta}} )(x)$, for $i\in \{1, \dots, k\}$, and we drop the dependence on $t$ as it does not change the calculation.

Applying Cauchy-Schwarz, we have
\begin{equation}\label{eqn key single PS}
|  D_{k, \phi}^\delta |  \lesssim 
\delta^{k}  |E_{\delta}|^{1/2} \left(   \int_{E_{\delta}}  \left|    T( \chi_{E_{\delta}} f_{k-1})(x)  \right|^2  \,dx   \right)^{1/2}.
\end{equation}
\vskip.125in

To bound this expression, we recall that our assumption on $E$ guarantees that $|E_{\delta}| \lesssim \delta^{d-\alpha-\epsilon}$, and we then iteratively apply the following lemma.

\begin{lemma}\label{lemma key single PS} With $T$ as in \eqref{eqn op}, and with $\phi$ satisfying \eqref{nondeg} and (\ref{mongeampere}), we have
$$  \left(   \int_{E_{\delta}}   |  T(f(x))|^2  \,dx  \right)^{1/2}   \leq C_\epsilon \delta^{\beta(d,\alpha)-1-\epsilon}\left(\int_{ E_\delta} |f(x)|^2   \,dx \right)^{1/2}, 
$$
where
\[
\beta(d,\alpha)-1:=\begin{cases} d-\alpha,& \frac{d+1}{2}\leq \alpha\leq d,\\ \frac{d-1}{2},& \alpha\leq \frac{d+1}{2}.\end{cases}
\]
\end{lemma}\vskip.125in

Applying the Lemma $k$-times to the right-hand-side of \eqref{eqn key single PS}, we obtain
$$\delta^{k}  |E_{\delta}|^{1/2} \delta^{k(\beta(d,\alpha)-1)}  |E_{\delta}|^{1/2}  \lesssim \delta^{ k + (d-\alpha) + k(\beta(d,\alpha)-1-\epsilon)  },$$
which agrees with \eqref{eqn: unit distance upp main phi}. 
\vskip.125in

We rely on the following Theorem, due to Phong and Stein \cite{PhSt91}, which is stated here without proof. 
\begin{theorem} \label{phongsteinth} Let $T_{\phi}^{\delta}$ be defined as above with $\phi$ satisfying assumptions (\ref{nondeg}) and (\ref{mongeampere}). Then 
$$ T_{\phi}^{\delta}: L^2({\Bbb R}^d) \to L^2_{\frac{d-1}{2}}({\Bbb R}^d) \ \text{with constants independent of} \ \delta,$$ where $L^2_{\gamma}({\Bbb R}^d)$ denotes the Sobolev space of functions with $\gamma$ (generalized) derivatives in $L^2({\Bbb R}^d)$. 
\end{theorem} 
\vskip.125in

\subsection{Proof of Lemma \ref{lemma key single PS}}
For the sake of simplicity, assume $t=1$ (the same argument works for all $t\sim 1$). Let $g$ be a nonnegative test function in $L^2(E_{\delta}) $.  It suffices to show that
$$\langle T(f \chi_{E_{\delta}}), g\chi_{E_{\delta}}  \rangle  \lesssim  \delta^{\beta(d,\alpha)-1 -\epsilon} \|f\|_{  L^2(E_{\delta})  }\cdot  \|g\|_{  L^2(E_{\delta})  }.$$
\\
Let $\rho$ be as in \eqref{rho}
and denote $\rho_r(x)=r^{-d}\rho\left(\frac{x}{r}  \right)$.
Since 
$f\chi_{E_{\delta}}  (x) \lesssim (f\chi_{E_{c\delta}})* \rho_{c\delta} (x)$, for $f$ non-negative and continuous, where $c>0$ is an absolute constant, 
we can bound the left-hand-side of this expression by
\begin{equation}\label{innerprod}
\langle T(  (f \chi_{E_{c\delta}})*\rho_{c\delta}   ), (g \chi_{E_{c\delta}} )*\rho_{c\delta}    \rangle.
\end{equation}  
\vskip.125in

For ease of notation, we write this as 
$$\langle  TF, G\rangle,$$
where we set
$ F: = F_{\delta}= (f \chi_{E_{\delta}})*\rho_{\delta} $,
$G: = G_{\delta}= (g\chi_{E_{\delta}})*\rho_{\delta},$ and dropped the subscript $c$.

Let $\eta_0(\xi)$ and $\eta$ be smooth cut-off functions such that $\eta_0$ is supported in the ball $\{|\xi|<4\}$, $\eta$ is supported in the annulus $\{1/2\le |\xi| \le 4\}$, and $\eta_0(\xi) + \sum_j \eta(2^{-j} \xi) \equiv 1$. Set $\eta_j(\cdot) = \eta(2^{-j} \cdot)$.  
For $f \in L^2(dx)$, define $\widehat{P_jf}$, 
the classical Littlewood-Paley projection (see, for instance, \cite{St93} pages 241-243), by the relation 
$$ \widehat{P_jf}=\widehat{f} \cdot \eta(2^{-j} \cdot).$$

Let $P_j$ and $P_k$ be Littlewood-Paley operators. Now 
$$   \langle TF, G\rangle 
\le   \sum_{k,j =0}^{\infty} \left| \langle T(  P_j F ) , P_k G \rangle\right|.$$

Applying Parseval's identity, 
$$\sim     \sum_{k,j =0}^{\infty}  
\left|\langle  
(T(  P_j F )   )^{\wedge}, 
(P_k G)^{\wedge}
\rangle\right|.$$

Since $\eta_k \sim (\eta_k)^2$, we can write
$$\sim \sum_{k,j =0}^{\infty}  \left|\langle   
(  P_k \left(  T(  P_j F\right)  ))^{\wedge},
(P_k G )^{\wedge}
 \rangle\right|.$$

 From the second term in this inner product, we see that the sum in $k$ is restricted to $2^k \le C\frac{1}{\delta}$.  
 Indeed, recalling that 
 $\widehat{ P_k    G  } (\xi)  =
\left( P_k \left[ (g \chi_{E_{\delta}} )*\rho_{\delta}   \right]\right)^\wedge (\xi) = \eta(2^{-k} \xi) (g \chi_{E_{\delta}} )^\wedge (\xi)    \widehat{\rho}(\delta \xi)$, it follows that $2^k \lesssim \frac{1}{\delta}$.  We will see below that we may also restrict to summing over $j$ so that $2^j \lesssim \frac{1}{\delta}$.

 Applying Cauchy-Schwarz, we have
$$
\langle  TF, G\rangle \,\,   \lesssim \,\,\sum_{j =0}^{\infty}   \sum_{ 2^k \lesssim \frac{1}{\delta} }
\| ( P_k \left(  T(  P_j  F\right)  )^{\wedge}\|_{L^2(E_{\delta})}
\cdot \|  ( P_k   G )^{\wedge}  \|_{L^2(E_{\delta})}.  
$$

The second term was handled in the proof of Theorem \ref{main3}, and it can be easily deduced from 
Lemma \ref{lemma: unit dist alpha} that it is bounded by 
$\left(  2^k \delta \right)^{\frac{(d-\alpha)}{2}}  \|g \|_{L^2(E_{\delta})}.$ 
Now
 \begin{equation}\label{sum}
\langle  TF, G\rangle \,\,   \lesssim \,\,\sum_{j =0}^{\infty}   \sum_{ 2^k \lesssim \frac{1}{\delta} }
\|  ( P_k (  T(  P_j  F)  ))^{\wedge}\|_{L^2(E_{\delta})}
\cdot\left(  2^k \delta \right)^{\frac{(d-\alpha)}{2}}  \|g \|_{L^2(E_{\delta})}.  
\end{equation}

The remainder of this section is dedicated to bounding the first term in the summand. For
$K\in \mathbb{N}$, to be determined, we handle the case when $|j-k|\le K$ and $|j-k|>K$ separately.

\vskip.125in

 \noindent
\textbf{Case 1: $|j-k|\le K$: }
We first bound
  \begin{equation}\label{revisedsum}
\sum_{\substack{j,k:\, 2^k \lesssim \frac{1}{\delta} \\ j\geq 0,\, |j-k|\leq K}}
\|  ( P_k (  T(  P_j  F)  ))^{\wedge}\|_{L^2(E_{\delta})}
\cdot\left(  2^k \delta \right)^{\frac{(d-\alpha)}{2}}  \|g \|_{L^2(E_{\delta})} 
\end{equation}
Write
\begin{equation}
\int  \left|  ( P_k(  T(  P_j F  )  ))^{\wedge}(\xi)    \right|^2 \,d\xi   
%$$\int  \left|  \widehat{  P_k \left(  T(  P_j \left[  (f \chi_{E_{\delta}})*\rho_{\delta}   )  \right] \right) }(\xi)    \right|^2 d\xi   $$
\sim 2^{-k(d-1)}
\int_{|\xi| \sim 2^k}  \left|  (T(  P_j F  ))^{\wedge}(\xi)    \right|^2 
|\xi|^{d-1}  \, d\xi.   
\end{equation}

Applying Theorem \ref{phongsteinth}, we can bound the above by
$$\lesssim  2^{-k(d-1)}
\int  \left|  ( P_j F  )^{\wedge}(\xi)    \right|^2 
\, d\xi    $$

 $$\sim  2^{-k(d-1)}
 \int_{|\xi| \sim 2^j< \frac{2C'}{\delta}} 
  \left|  ( (f \chi_{E_{\delta}})*\rho_{\delta}      )^{\wedge}(\xi)    \right|^2 
 \,d\xi.    $$
 
 Now
 \begin{equation}\label{secondterm}
 \| ( P_k (  T(  P_j F )))^\wedge\|_{L^2(E_{\delta})}
 \lesssim  2^{-k\frac{(d-1)}{2}} 
 2^{j\frac{(d-\alpha)}{2}}\delta^{\frac{(d-\alpha)}{2}}  \|f \|_{L^2(E_{\delta})}.
 \end{equation}

 Plugging this expression into \eqref{revisedsum} and rearranging terms yields

  \begin{equation}\label{boundedrevisedsum}
\|f \|_{L^2(E_{\delta})}   \|g \|_{L^2(E_{\delta})} \delta^{(d-\alpha)} 
\sum_{\substack{j,k:\, 2^k \lesssim \frac{1}{\delta} \\ j\geq 0,\, |j-k|\leq K}}
2^{-k\frac{(d-1)}{2}} 
 2^{k\frac{(d-\alpha)}{2}}
   2^{j\frac{(d-\alpha)}{2}}.   
\end{equation}

When $\alpha>\frac{d+1}{2}$, we bound the sum in \eqref{boundedrevisedsum} (up to a constant dependent on $K$) by 
$$ \sum_{\substack{j:\, 2^j \lesssim \frac{1}{\delta} \\ j\geq 0}}
2^{-j\frac{(d-1)}{2}} 
 2^{j\frac{(d-\alpha)}{2}}
   2^{j\frac{(d-\alpha)}{2}} \lesssim
   1,$$
   and when $\alpha \le \frac{d+1}{2}$,  we have
$$   \sum_{\substack{j:\, 2^j \lesssim \frac{1}{\delta} \\ j\geq 0\, }}
2^{-j\frac{(d-1)}{2}} 
 2^{j\frac{(d-\alpha)}{2}}
   2^{j\frac{(d-\alpha)}{2}} \lesssim \log\left( \frac{1}{\delta}\right)   \left(  \frac{1}{\delta}  \right)^{d-\alpha - \frac{d-1}{2}}.$$

We conclude that the left-hand side of \eqref{revisedsum} is bounded by $ \delta^{\beta(d,\alpha)-1- \epsilon} \|f\|_{  L^2(E_{\delta})  }\cdot  \|g\|_{  L^2(E_{\delta})  }$, as desired. \\ 

%
%\[
%\sum_{\substack{j,k:\, 2^k \lesssim \frac{1}{\delta} \\ j\geq 0,\, |j-k|\leq K}}
%\|  ( P_k (  T(  P_j  F)  ))^{\wedge}\|_{L^2(E_{\delta})}
%\cdot\left(  2^k \delta \right)^{\frac{(d-\alpha)}{2}}  \|g \|_{L^2(E_{\delta})}\lesssim \delta^{d-\alpha}   \|f \|_{L^2(E_{\delta})}   \|g \|_{L^2(E_{\delta})}.
%\]{\color{red}{(Is this estimate what you meant? I'm a bit confused. Does it hold for all $\alpha$ or only those $\alpha>\frac{d+1}{2}$?)}}
%%, provided that $\alpha >\frac{d+1}{2}$.  
%In the regime $|j-k|\le K$ and $\alpha \le \frac{d+1}{2}$, setting $j=k$ and summing in $2^k \lesssim \frac{1}{\delta}$ yields
%$$\sum_{k=0}^{\frac{C}{\delta}} 2^{ (d-\alpha) -\left(\frac{d-1}{2} \right)}  
%\sim  \left( \frac{1}{\delta} \right)^{(d-\alpha) -\left(\frac{d-1}{2} \right)},$$ and plugging this into \eqref{sum} yields
%\[
%\begin{split}
%&\left( \frac{1}{\delta} \right)^{(d-\alpha) -\left(\frac{d-1}{2} \right)} \delta^{\frac{(d-\alpha)}{2}}  \|f \|_{L^2(E_{\delta})}
%\delta^{\frac{(d-\alpha)}{2}}  \|g \|_{L^2(E_{\delta})}\\
%=&\delta^{\beta(d, \alpha) -1}  \|f \|_{L^2(E_{\delta})}   \|g \|_{L^2(E_{\delta})}.
%\end{split}
%\] 

\noindent
\textbf{Case 2: $|j-k|> K$: }
In the case that $|j-k|>K$, we will expand the Fourier transform $(T(  P_j \left[  (f \chi_{E_{\delta}})*\rho_{\delta}   \right] )  )^{\wedge}(\xi) $ and examine the critical points of the phase function.  We follow the proof presented in \cite{ITU}, with the only major change being that we utilize our assumption that $E$ is $\{\delta_i\}$-discrete $\alpha$-regular.
Using Fourier inversion, we can write 
$$T(  P_j \left[  (f \chi_{E_{\delta}}   )*\rho_{\delta}    \right])(x)  = 
\frac{1}{\delta} \int   \Psi  \left(  \frac{  \phi(x,y) - 1}{\delta}  \right)   P_j ( f\chi_{E_{\delta}}* \rho_{\delta})(y)  \, \Psi_0(x,y)\,dy  $$

$$=  \int   e^{  2\pi i [( \phi(x,y) -1  )s + \zeta y ]}  \, 
\widehat{  \Psi}( \delta s) \, ( P_j ( f\chi_{E_{\delta}}* \rho_{\delta}) )^{\wedge}  (\zeta) \Psi_0(x,y) \,   ds d\zeta  dy.$$

Now
\[
\begin{split}
&(T(  P_j \left[  (f \chi_{E_{\delta}})*\rho_{\delta}    \right]  ) )^{\wedge}(\xi)\\
\sim  &\int    e^{  2\pi i [( \phi(x,y) -1  )s + \zeta y - \xi x]}   \widehat{\Psi}(\delta s) 
( P_j ( f\chi_{E_{\delta}}* \rho_{\delta}) )^{\wedge}  (\zeta)   \Psi_0(x,y) \,   ds d\zeta dy dx.
\end{split}
\]

 Multiplying both sides by $\eta(2^{-k}\cdot)$, we see that 
\[
\begin{split}
&\eta(2^{-k}\xi) \cdot
\,(  T(  P_j \left[  (f \chi_{E_{\delta}})*\rho_{\delta}    \right] )  )^{\wedge}(\xi)  \\
\sim &\eta(2^{-k}\xi) 
 \int I(s, \xi, \zeta) \,   \widehat{\Psi}(\delta s)   
( P_j ( f\chi_{E_{\delta}}* \rho_{\delta}) )^{\wedge}  (\zeta)  \,ds d\zeta,  
\end{split}
\]where $I(s, \xi, \zeta)  : = \int e^{  2\pi i [( \phi(x,y) -1  )s + \zeta y - \xi x]} \,  \Psi_0(x,y)\,dx dy $.

We note that $j$ is restricted to $2^j \le C\frac{1}{\delta}$.  
 Indeed, recalling that 
 \[
 (P_j \left[ (g \chi_{E_{\delta}} )*\rho_{\delta}   \right] )^{\wedge} (\zeta) = \eta(2^{-j} \zeta) (g \chi_{E_{\delta}}  )^{\wedge}(\zeta)    \widehat{\rho}(\delta \zeta),
 \]it follows that $2^j \lesssim \frac{1}{\delta}$.

% LEMMA 2.5 FROM ITU
We use the following Lemma, which appears in \cite{ITU} (see Lemma 2.5), to bound $| I(s, \xi, \zeta)|$. 

\begin{lemma}\label{IBP} 
Suppose that $|\zeta|\sim 2^j$ and $|\xi| \sim 2^k$.  
Then there exists a $K>0$ so that if $|j-k|>K$, then for each positive integer $M$, there exists a positive constant $c_M>0$ so that 
$$|I(s,\xi,\zeta)|  \le c_M\, \inf \left\{|s|^{-M}, 2^{-j\,M}, 2^{-k\,M} \right\}. $$ 
\end{lemma}
%\vskip.125in

With Lemma \ref{IBP} in tow, we return to the estimate above.  Plugging in the estimate from the lemma and integrating in $s$, we have 
\[
\begin{split}
&\left|   \eta(2^{-k}\xi) \cdot \,
( T(  P_j \left[  (f \chi_{E_{\delta}})*\rho_{\delta}    \right]  ) )^{\wedge}(\xi)  \right|\\
\lesssim &
 \eta(2^{-k}\xi)   \min \left\{ 2^{-j\,(M-1)}, 2^{-k\,(M-1)} \right\}  
  \int_{|\zeta| \sim 2^j}   
 | ( f\chi_{E_{\delta}}* \rho_{\delta})^\wedge (\zeta)|   \,d\zeta.  
 \end{split}
 \]
 
 Finally, applying Cauchy-Schwarz, we can bound this expression above by
 \[
 \begin{split}
 \lesssim &
 \eta(2^{-k}\xi)    \min\left\{ 2^{-j\,(M-1)}, 2^{-k\,(M-1)} \right\}  
 2^{\frac{jd}{2}}\,
\left(  \int_{|\zeta| \sim 2^j}   
 |  ( f\chi_{E_{\delta}}* \rho_{\delta})^{\wedge}  (\zeta)|^2   \,d\zeta \right)^{1/2}\\
 \lesssim &
 \eta(2^{-k}\xi)   \min \left\{ 2^{-j\,(M-1)}, 2^{-k\,(M-1)} \right\}  
 2^{\frac{jd}{2}}\,
\left(  2^j \delta \right)^{\frac{(d-\alpha)}{2}}  \| f\|_{L^2(E_{\delta})}.
\end{split}
\]

It follows that 
$$\|  ( P_k \left(  T(  P_j F \right) ) )^{\wedge}\|_{L^2(E_{\delta})}
\lesssim 
 2^{\frac{kd}{2}}\,
 \min  \left\{ 2^{-j\,(M-1)}, 2^{-k\,(M-1)} \right\}  
 2^{\frac{jd}{2}}\,
\left(  2^j \delta \right)^{\frac{(d-\alpha)}{2}}  \| f\|_{L^2(E_{\delta})},
$$
and since $M$ can be taken arbitrarily large, the result follows.

\appendix

\bigskip
\section{Higher dimensional analogue of a Fubini-like theorem}\label{sec: Fubini}

Although a direct substitute of Fubini's theorem in a fractal setting is not available, the following theorem acts as a sort of substitute; it can be found in \cite{Falc86} (see page 72, Theorem 5.8).  
%\vskip.125in

% Falconer's book                       
\begin{theorem}
Let  $A$ be any subset of the $x$-axis and let $B$ be a subset of the plane.  For $x\in \mathbb{R}$, let $B_x$ denote the linear set $\{y: (x,y) \in B\}$.  
Suppose that there exists a constant $c$ so that, for each $x\in A$, it holds that $\mathcal{H}^t(B_x) \geq c$.  
Then 
$$\mathcal{H}^{s+t}(B) \geq bc \mathcal{H}^s(A),$$
where $b$ depends only on $s$ and $t$.  
\end{theorem}
%\vskip.125in

In this paper, we require a higher-dimensional analogue of this theorem,  which to the best of our knowledge does not seem to appear in the literature. For this reason, we include the statement and outline the proof for completeness. 
The utility of this theorem is demonstrated by Theorem \ref{main2}, \ref{main2.5}, and Proposition \ref{glue}.

\begin{theorem}\label{higherFubgen}
Let $d\geq 2$ and $1\le k\le (d-1) $.  
Let  $A$ be any Borel subset of $\mathbb{R}^{d-k}$ and let $B$ be a Borel subset of $ \mathbb{R}^d$. 
For $x\in \mathbb{R}^{d-k}$, set 
$$B_x: =  \{(y_1, \dots, y_{k}): (x, y_1, \dots, y_k) \in B\}.$$
Suppose that there exists some constant $c>0$ so that, for each $x\in A$, it holds that $\mathcal{H}^t(B_x)  \geq c$. 
Then 
$$\mathcal{H}^{s+t}(B) \geq bc \mathcal{H}^s(A),$$
where $b$ depends only on $k$, $s$, $t$, and the ambient dimension $d$.  
\end{theorem}
%\vskip.125in

The following is an immediate corollary.  

\begin{corollary}\label{cor: Fubini}
Let  $A$, $B$, and $B_x$ as in Theorem \ref{higherFubgen}.
 Suppose that there exists $t\geq 0$ so that, for each $x\in A$, it holds that $\dim_{\mathcal{H}}(B_x) \geq t$.  
Then 
$$\dim_{\mathcal{H}}(B) \geq  t+ \dim_{\mathcal{H}}(A) .$$
\end{corollary}
%\vskip.125in
 
Rather than working with Hausdorff measures directly, we work with a comparable measure that is defined by coverings of a set by binary intervals.  A binary half-open cube is a set of the form
$$[2^{-k} m_1, 2^{-k}(m_1 + 1) ) \times [2^{-k} m_2, 2^{-k}(m_2 + 1) ) \times \cdots \times [2^{-k} m_d, 2^{-k}(m_d + 1) ),$$
where $m_1, \dots, m_d$ are integers and $k$ is a non-negative integer.  
Define an outer measure on $\mathbb{R}^d$ by 
\begin{equation}\label{binarymeasure}
\mathcal{M}^s(F) = \lim_{\delta \rightarrow 0} \mathcal{M}^s_{\delta}(F),
\end{equation}
where 
\begin{equation}\label{apxbinarymeasure}
\mathcal{M}^s_{\delta}(F)  = \inf \sum_{i=1}^{\infty} |C_i|^s,
\end{equation} where the infimum is taken over all countable $\delta$-coverings of $F$ by half-open binary cubes $\{C_i\}$ with length denoted by $|C_i|$.  
%\\

Observe that, for each Borel set $F$, the measure $\mathcal{M}^s(F)$ is comparable to $\mathcal{H}^s(F)$. Indeed, since a covering of a set $F$ by binary intervals is an admissible cover in the definition of Hausdorff measures, it follows that 
$\mathcal{H}^s_{\delta}(F) \le \mathcal{M}^s_{\delta}(F)$.  

Conversely, since any axes parallel box has the property that each side (call such a side $I$), is contained in the union of two consecutive half-open binary intervals of length at most $2|I|$, we have 
$2^{s+1} \mathcal{H}^s_{\delta}(F) \geq \mathcal{M}^s_{\delta}(F)$.

The following Lemma is a higher dimensional analogue of Lemma 5.7 in \cite{Falc86}.  It is a straightforward exercise to verify that the proof there can be recycled to prove Lemma \ref{lemma5.7gen}. 

\begin{lemma}\label{lemma5.7gen}
Let $d\geq 1$ and $0\le s \le d$. 
Let $A$ be any subset of $\mathbb{R}^d$, let $\{ I_i\}$ be a countable $\delta$-cover of $A$ by binary cubes, and let $\{a_i \}$ be a sequence of positive numbers.  Suppose $c$ is a constant such that 
$$\sum_{i:\, x\in I_i}  a_i >c$$ for all $x\in A$.  Then $$\sum_i a_i |I_i|^s \geq c \mathcal{M}_{\delta}^s (A).$$
\end{lemma}
%\vskip.125in

\begin{proof}[Proof of Theorem \ref{higherFubgen}]
For simplicity of presentation, we consider the case when $A$ is a subset of the $x$-axis.  
We present the proof in such a way that the more general case is a natural extension.  
%Only slight modifications of Theorem 5.8 of \cite{Falc86} are required, but 
%we modify the presentation so that the proof more readily extends to a more general setting.  
Let $A$ and $B$ be as in the statement of the theorem. 
Let $c>0$ so that, for each $x\in A$, it holds that 
\begin{equation}\label{hyp}  \mathcal{H}^t(B_x)  \geq c.\end{equation}
Let $\delta>0$ and let $\{C_i\}$ be a countable covering of $B$ by cubes of diameter at most $\delta$.  
Denote 
$$\pi(C_i):= \pi_1(C_i): = \{ y_1: (y_1,y_2, \dots, y_{d}) \in C_i,\, \text{for some } y_2,\ldots,y_d \},$$
and 
$$(C_i)_x: =C_i\cap \{(x, y_2, \dots, y_{d}): (y_2, \dots, y_{d}) \in \mathbb{R}^{d-1}\},$$
the projection of $C_i$ onto the first coordinate and the slice of $C_i$ by an axis-parallel hyperplane determined  by $x$, respectively.  
%\vskip.125in

The main ingredients of the proof are hypothesis \eqref{hyp} coupled with the following simple observations. 
Observe that, for each $x\in \mathbb{R}$, the diameters of $C_i$, $(C_i)_x$, and $\pi(C_i)$ are all comparable.  (We will use $|\cdot |$ to  denote the diameter of a cube.)
Observe that the arbitrary covering of $B$ by cubes $\{C_i\}$ yields a natural cover of $A$ and of the sets $(B)_x$.  
Indeed, 
$$A\subset \bigcup_i \pi(C_i)$$
and, 
for each $x\in A$,
\begin{equation}\label{obs2}B_x \subset \bigcup_i (C_i)_x.\end{equation}

We begin by verifying the hypothesis of Lemma \ref{lemma5.7gen} with $a_i = |C_i|^t$ and $I_i = \pi(C_i)$.  
%There is no harm in assuming that the cubes $\{C_i\}$ are apriori binary cubes.  
By equation \eqref{obs2} and the definition of Hausdorff measure, 
$$\mathcal{H}_{\delta}^t(B_x) \le \sum_i |  (C_i)_x|^t.$$

Set $A_{\delta} = \{x \in A: \mathcal{H}_{\delta}^t(B_x)\geq c\}.$
Now, if $x\in A_{\delta} $, we have
$$c \le \mathcal{H}_{\delta}^t(B_x) \le \sum_i |  (C_i)_x|^t \sim \sum_{ i:\, x\in \pi(C_i)} |  C_i|^t.$$
 
The above argument is true if we restrict to a covering of the set $B$ by binary cubes, which we will do to apply Lemma \ref{lemma5.7gen}.

Applying Lemma \ref{lemma5.7gen}, we obtain the following lower bound:
$$\sum_i |  C_i|^{s+t}  \sim \sum_i |  C_i|^t |  \pi(C_i)|^s \gtrsim c \mathcal{M}_{\delta}^s(A_{\delta}).$$

This is true for any covering of $B$ by binary cubes $\{B_i\}$, so 
$$\mathcal{M}^{s+t}(B) \gtrsim c \mathcal{M}^s(A_{\delta}).$$

We note that $A_{\delta}$ increases to $A$ as $\delta$ decreases to $0$, and apply a limiting argument to complete the proof.  In conclusion, we have
$$\mathcal{H}^{s+t}(B) \gtrsim c \mathcal{H}^s(A).$$

\end{proof}

\section{Regularity of the lattice example \ref{example: lattice}}\label{ADproperty}
In this section, we prove that if $E$ a set of Hausdorff dimension $\alpha$ constructed as in Example \ref{example: lattice}, then $E$ is not AD regular (as defined in Remark \ref{rmk: AD}) unless $\alpha=d$, but is $\{\delta_i\}$-discrete $\alpha$-regular (as defined in (\ref{eqn: delta discrete})).

\subsection{$E$ is not $AD$ regular}
In particular, we show that $E$ cannot support a Borel probability measure $\mu$ so that for each $r>0$ and for each $x\in E$, $cr^{\alpha} \le  \mu(B(x,r)) \le Cr^{\alpha}$, for universal constants $0<c<C$. 

Suppose such a $\mu$ exists. Take $r=\delta_i:=q_i^{-d/\alpha}$ and fix $x\in E$, then one has $\mu(B(x,\delta_i))\geq c \delta_i^\alpha$. On the other hand, by the lattice construction, $E\cap B(x, \delta_i)$ is contained in the union of at most $(\delta_i q_{i+1})^d$ balls of radius $\delta_{i+1}$. This then implies that
\[
c\delta_i^\alpha\leq \mu(B(x,\delta_i))\leq (\delta_i q_{i+1})^d\cdot C\delta_{i+1}^\alpha.
\]Plugging $\delta_i:=q_i^{-d/\alpha}$ into the above inequality, one obtains $\frac{c}{C}\leq q_i^{d(1-\frac{d}{\alpha})}$. Taking $q_i$ sufficiently large, we arrive at a contradiction when $\alpha<d$.

\subsection{$E$ is $\{\delta_i\}$-discrete $\alpha$-regular}
Choose $\delta_i=q_i^{-d/\alpha}$, then by definition one has $\delta_i\to 0$. For any $r\geq \delta_i$ and $x\in \mathbb{R}^d$, our goal is to show that 
\[
|E_{\delta_i}\cap B(x,r)|\lesssim \left(\frac{r}{\delta_i}\right)^\alpha \delta_i^d.
\]

To see this, one first observes that $E_{\delta_i}\subset \tilde{E_i}$, where $E_i$ is as defined in Example \ref{example: lattice} and $\tilde{E}_i$ is its slight enlargement by a constant. More precisely, one can define 
\[
\tilde{E}_i=\bigcup_{\vec{u} \in \{0, 1, \dots, q_i\}^d} \left\{\vec{x} \in \mathbb{R}^d: |\vec{x} - \frac{\vec{u}}{q_i}| < c_0q_i^{-d/\alpha}\right\}
\]for some universal constant $c_0$. Therefore, 
\[
|E_{\delta_i}\cap B(x,r)|\leq |\tilde{E}_i\cap B(x,r)|\lesssim_{c_0} \max (1, (rq_i)^d)\cdot \delta_i^d.
\]The right hand side above is obviously smaller than
\[
\max\left(1, r^{d-\alpha}\left(\frac{r}{\delta_i}\right)^\alpha\right)\cdot \delta_i^d\lesssim \left(\frac{r}{\delta_i}\right)^\alpha \delta_i^d,
\]where we have observed that $\alpha\leq d$ and $\delta_i\leq r\lesssim 1$.
  
%\vskip3in

\section{An example concerning the parabolic metric}\label{parabolic}

In this subsection, we give a lower bound related to Theorem \ref{main3PhongStein} for $\alpha< \frac{d+1}{2}$. 
The plan is to constrcut a $\{\delta_i\}$-discrete $\alpha$-regular set along with a metric $\phi$ so that the distance $1$ is repeated often.  
In particular, using the notation introduced in Section \ref{PS_section}, we give construction of a metric $\phi$ and a set $E$ so that the upper Minkowski dimension of $VS_{\vec{t}}^{\phi,k}(E)$ is bounded below by $\alpha + \frac{\alpha(d-1)k}{d+1}$.  We set $\vec{t}=\vec{1}$, but the example presented here can be modified to work for any $\vec{t}\in \mathbb{R}^k$ with positive components.  
We modify the construction used in \cite{EIT} to the setting of chains.  
\\

To define the set $E$, let 
$\{q_i\}_{i\in \mathbb{N}}$ be a sequence of positive integers such that $q_{i+1}=q_i^i$ and $q_1=2$.
Set $L_i$ equal to the truncated and non-isotropically scaled lattice:
$$ \left\{\left(\frac{x_1}{q_i^{\frac{d}{d+1}}},\ldots,\frac{x_{d-1}}{q_i^{\frac{d}{d+1}}}, \frac{x_d}{q_i^{\frac{2d}{d+1}}}  \right):\, x\in {\Bbb Z}^d, 0 \leq x_1, \dots, x_{d-1} \leq q_i^{\frac{d}{d+1}} \text{ and } 0 \leq x_d \leq q_i^{\frac{2d}{d+1}} \right\},$$ and set
$E_i$ equal to the $q_i^{-\frac{d}{\alpha}}$ neighborhood of $L_i$. It is known that $E=\bigcap_i E_i$ is a construction of a set of Hausdorff dimension $\alpha$ (see, for example, \cite[Chapter 8, Theorem 8.15]{Falc86}), and it is a simple calculation to check that $E$ is $\{\delta_i\}$-discrete $\alpha$-regular (with $\delta_i=q_i^{-\frac{d}{\alpha}}$).

The function $\phi$ is defined using a variant of the incidence construction due to P. Valtr \cite{V05}.  In particular, $\phi(x,y)={||x-y||}_B$, where ${|| \cdot ||}_B$ is the norm induced by a convex body $B$.  
Roughly speaking, $B$ is created by glueing two copies of a paraboloid together. 
Explicitly, let
\[
\begin{split}
B_U=&\{(x_1, x_2, \dots, x_d) \in \mathbb{R}^d : x_i \in [-1,1], \text{ for } 1 \leq i \leq d-1, \\
&\qquad \qquad\qquad \qquad \qquad \text{ and } x_d = 1-\left( x_1^2+x_2^2+ \dots + x_{d-1}^2 \right)\},
\end{split}
\]
and
\[
\begin{split}
B_L=&\{(x_1, x_2, \dots, x_d) \in \mathbb{R}^d : x_i \in [-1,1], \text{ for } 1 \leq i \leq d-1,\\
&\qquad \qquad\qquad \qquad \qquad \text{ and } x_d = -1+ x_1^2+x_2^2+ \dots + x_{d-1}^2\}.
\end{split}
\]
Now, let
\[
\begin{split}
B'=& \left(B_U \cap \left\lbrace (x_1, x_2, \dots, x_d) \in \mathbb{R}^d :x_d \geq 0\right\rbrace \right) \cup \\
& \qquad \qquad\qquad \qquad \qquad\left( B_L \cap \left\lbrace (x_1, x_2, \dots, x_d) \in \mathbb{R}^d :x_d \leq 0 \right\rbrace \right).\end{split}
\]
Finally, define $B$ to be the convex body $B'$, with the ridge at the transition between $B_U$ and $B_L$ smoothed.

The problem of calculating $\overline{{\rm dim}}_{{\mathcal M}}(VS_{\vec{t}}^{\phi,k}(E))$ reduces to the problem of calculating the number of balls of radius $q_i^{-\frac{d}{\alpha}}$ needed to cover $VS_{\vec{t}}^{\phi,k}(L_i)$. Specifically, for each $i$ and for each $x\in L_i$, 
$$N\left( \left\{y \in L_i : {||x-y||}_B=1 \right\}, q_i^{-\frac{d}{\alpha}}\right) \sim q_i^{\frac{d(d-1)}{d+1}},$$ where $N(A, \delta)$ denotes the number of $\delta$-balls needed to cover a compact set $A$.  More generally, 
$$N\left( VS_{x, \vec{t}}^{\phi,k}(L_i), q_i^{-\frac{d}{\alpha}}\right) \sim q_i^{\frac{kd(d-1)}{d+1}},$$
where $VS_{x, \vec{t}}^{\phi,k}(L_i)$ denotes the vertex set of the $k$-chain with $x$, the first vertex, pinned.  
Thus, 
the number of balls of radius $q_i^{-\frac{d}{\alpha}}$ needed to cover $VS_{\vec{t}}^{\phi,k}(L_i)$
and, consequently, $VS_{\vec{t}}^{\phi,k}(E)$, is 
$$ \gtrsim q_i^d \cdot q_i^{\left(\frac{d(d-1)}{d+1}\right) k}
= {(q_i^{-\frac{d}{\alpha}})}^{-\frac{\alpha[(d+1) + (d-1)k]}{d+1}}.$$
Here, we used the observation that, for each $i$, $E$ contains $\sim q_i^d$ many elements of the set of scaled lattice points, $L_i$.
It follows that $\overline{  {\rm dim}}_{\mathcal{M}}(VS_{\vec{t}}^{\phi,k}(E))$ is at least $\frac{\alpha[(d+1) + (d-1)k]}{d+1}$. 
%This number is greater than $2s-1$ if $s<\frac{d+1}{2}$. 
%%%% end 

%%%%%%%%%%%%%%%%%%%%%%%%%%%%%

%%%%%%%%PROOF OF TRIANGLES w GIOW (Lebesgue measure)

%%%%%%%%PROOF OF TRIANGLES w BOCHEN (Hausdorff measure)

\end{document}